\pdfoutput=1

\documentclass[a4paper, twoside]{article}

\usepackage{amssymb}
\usepackage{amsmath}
\usepackage{amsthm}

\usepackage[T1]{fontenc}   
\usepackage{lmodern}       
\usepackage{mathrsfs}      
\usepackage{upgreek}       
\usepackage{microtype}     
\usepackage{xurl}          
\usepackage{hyperref}      
\usepackage[nottoc]{tocbibind}     

\newcounter{enumi_memory}  

\theoremstyle{definition}
\newtheorem*{intro-definition}{Definition}

\theoremstyle{plain}
\newtheorem{intro-example}{Example}
\newtheorem{Thm}{Theorem}

\newtheorem{intro-corollary}{Corollary}
\newtheorem{reg-hyp}{Regularity hypotheses}

\swapnumbers
\theoremstyle{plain}
\newtheorem{theorem}{Theorem}[section]
\newtheorem{lemma}[theorem]{Lemma}
\newtheorem{corollary}[theorem]{Corollary}

\theoremstyle{remark}
\newtheorem{remark}[theorem]{Remark}

\theoremstyle{definition}
\newtheorem{definition}[theorem]{Definition}
\newtheorem{miniremark}[theorem]{}
\newtheorem{example}[theorem]{Example}

\newtheoremstyle{citing}
  {3pt}
  {3pt}
  {\itshape}
  {}
  {}
  {\textbf.}
  {.5em}
  {\thmnote{#3}}
\theoremstyle{citing}
\newtheorem*{citing}{}

\DeclareMathOperator{\without}{\sim}

\newcommand{\restrict}{\mathop{\llcorner}}

\newcommand{\ud}{\,\mathrm{d}}

\newcommand{\etalchar}[1]{$^{#1}$}

\newcommand{\noopsort}[1]{}  
 \def\cfac#1{\ifmmode\setbox7\hbox{$\accent"5E#1$}\else
	\setbox7\hbox{\accent"5E#1}\penalty 10000\relax\fi\raise 1\ht7
	\hbox{\lower1.15ex\hbox to 1\wd7{\hss\accent"13\hss}}\penalty 10000
	\hskip-1\wd7\penalty 10000\box7}
\def\polhk#1{\setbox0=\hbox{#1}{\ooalign{\hidewidth
			\lower1.5ex\hbox{`}\hidewidth\crcr\unhbox0}}}

\DeclareMathOperator{\card}{card}

\DeclareMathOperator{\trace}{trace} 
\DeclareMathOperator{\with}{:}      
\DeclareMathOperator{\reach}{reach} 
\DeclareMathOperator{\Unp}{Unp}     
\DeclareMathOperator{\Bdry}{Bdry}   
\DeclareMathOperator{\Clos}{Clos}   
\DeclareMathOperator{\Tan}{Tan}     
\DeclareMathOperator{\Nor}{Nor}     
\DeclareMathOperator{\spt}{spt}     
\DeclareMathOperator{\im}{im}       
\DeclareMathOperator{\Int}{Int}     
\DeclareMathOperator{\diam}{diam}   
\DeclareMathOperator{\Lip}{Lip}     
\DeclareMathOperator{\sign}{sign}   
\DeclareMathOperator{\Number}{N}    
\DeclareMathOperator{\dmn}{dmn}     
\DeclareMathOperator{\dist}{dist}   
\DeclareMathOperator{\Hom}{Hom}     
\DeclareMathOperator{\grad}{grad}   
\DeclareMathOperator{\Der}{D}       
\DeclareMathOperator{\weakD}{\mathbf{D}}  
\DeclareMathOperator{\boundary}{\partial} 
\DeclareMathOperator{\pt}{pt}       
\DeclareMathOperator{\ap}{ap}       

\title{A priori bounds for geodesic diameter. Part II. Fine connectedness
properties of varifolds}

\author{Ulrich Menne \and Christian Scharrer}

\begin{document}

\maketitle 

\begin{abstract}
	For varifolds whose first variation is representable by integration,
	we introduce the notion of indecomposability with respect to locally
	Lipschitzian real valued functions.  Unlike indecomposability, this
	weaker connectedness property is inherited by varifolds associated
	with solutions to geometric variational problems phrased in terms of
	sets, $G$ chains, and immersions; yet it is strong enough for the
	subsequent deduction of substantial geometric consequences therefrom.
	Our present study is based on several further concepts for varifolds
	put forward in this paper: real valued functions of generalised
	bounded variation thereon, partitions thereof in general, partition
	thereof along a real valued generalised weakly differentiable function
	in particular, and local finiteness of decompositions.
\end{abstract}

\paragraph{MSC-classes 2020}
	49Q15 (Primary); 46E35, 49Q05, 49Q20, 53A07, 53A10, 53C42, 54D05
	(Secondary).


\tableofcontents

\section{Introduction}

Throughout the introduction, we suppose the following assumptions to be valid.

\begingroup \hypertarget{II-gen-hyp}{}
	\begin{citing} [\textbf{General hypotheses}]
		Suppose $m$ and $n$ are positive integers, $m \leq n$, $U$ is
		an open subset of $\mathbf R^n$, $V$ is an $m$ dimensional
		varifold in $U$, and its first variation, $\updelta V$, is
		representable by integration.
	\end{citing}
\endgroup

After reviewing the notion of sets of locally finite perimeter on varifolds,
we describe our contributions---five concepts, four examples,
eight theorems, and three
corollaries---in their logical order.

\subsection*{Sets of locally finite perimeter}

A fundamental concept for the present study is that of \emph{distributional
$V$ boundary} of a  $\| V \| + \| \updelta V \|$ measurable set
$E$ (equivalently, a set $E$ which is measurable with respect
to both measures, $\| V \|$ and $\| \updelta V \|$) defined by
\begin{equation*}
	V \boundary E = ( \updelta V ) \restrict E - \updelta ( V \restrict E
	\times \mathbf G (n,m) ) \in \mathscr D' ( U, \mathbf R^n )
\end{equation*}
in \cite[5.1]{MR3528825}.  With respect to the geometry of $V$, the
distribution $V \boundary E$ and its Borel regular variation measure, $\| V
\boundary E\|$, could be understood as the distributional gradient of the
characteristic function of $E$ and the measure-theoretic perimeter of $E$,
respectively.  Thus, one may say that the perimeter of $E$ with respect to $V$
is locally finite if and only if $V \boundary E$ is representable by
integration or, equivalently, $\| V \boundary E \|$ is a Radon measure.

\subsection*{Real valued functions of generalised bounded variation}

For a real valued $\| V \| + \| \updelta V \|$ measurable function, we shall
apply the preceding concept of locally finite perimeter to its subgraph to
obtain a definition of a function possessing generalised bounded variation
with respect to $V$.  Alternatively, as suggested in \cite[8.1,
8.10]{MR3528825}, we could base a definition of this notion on a distribution
in $U \times \mathbf R$ of type $\mathbf R^n$ associated with the
distributional boundaries of superlevel sets of the function in question.
Both approaches are equivalent by the following theorem, Theorem
\ref{Thm:GBV}, which includes a characterisation---by means of a natural
absolute continuity condition---of the subclass $\mathbf T (V)$ of generalised
$V$ weakly differentiable functions introduced in \cite[8.3]{MR3528825}.

\begin{Thm} [see \protect{\ref{thm:equiv_bv} and \ref{corollary:equiv_bv}}]
	\label{Thm:GBV}
	Suppose $V$ fulfils the \hyperlink{II-gen-hyp}{general hypotheses},
	$f$ is a real valued $\| V \| + \| \updelta V \|$ measurable function,
	\begin{equation*}
		E = \{ (x,y) \with f(x)>y \},
	\end{equation*}
	$W$ is the Cartesian product of $V$ with the one-dimensional varifold
	in $\mathbf R$ associated with $\mathscr L^1$, the distribution $T \in
	\mathscr D' ( U \times \mathbf R, \mathbf R^n )$ satisfies
	\begin{equation*}
		T ( \phi ) = {\textstyle\int} V \boundary \{ x \with (x,y) \in
		E \} ( \phi ( \cdot, y ) ) \ud \mathscr L^1 \, y \quad
		\text{for $\phi \in \mathscr D ( U \times \mathbf R, \mathbf
		R^n )$},
	\end{equation*}
	and $F : \dmn f \to U \times \mathbf R$ is defined by $F(x) =
	(x,f(x))$ for $x \in \dmn f$.
	
	Then, $W \boundary E$ is representable by integration if and only if
	$T$ is representable by integration.   Moreover, $f \in \mathbf T (V)$
	if and only if $\| W \boundary E \|$, or equivalently $\| T \|$, is
	absolutely continuous with respect to $F_\# \| V \|$.
\end{Thm}

Coarea formulae for functions $f$ of generalised $V$ bounded variation in
terms of $T$---which acts as distributional derivative of $f$---follow, see
\ref{remark:equiv_bv}; in particular, for characteristic functions, we recover
the concept of locally finite $V$ perimeter of the corresponding set.  These
developments are in line with those originating from \cite{MR1152641} and
\cite{MR1354907} for the special case $m = n$ and $\| V \| (A) = \mathscr L^n
(A)$ for $A \subset U$, see \cite[8.8, 8.19]{MR3528825} and \ref{remark:GBV}.

\subsection*{Partitions}

We employ sets of zero distributional $V$ boundary to introduce the notion of
partition of $V$ which generalises that of decomposition of $V$.

\begin{intro-definition} [see \protect{\ref{def:partition}}]
	Under the \hyperlink{II-gen-hyp}{general hypotheses}, a subfamily
	$\Pi$ of
	\begin{equation*}
		\{ V \restrict E \times \mathbf G (n,m) \with \text{$E$ is $\|
		V \| + \| \updelta V \|$ measurable, $V \boundary E = 0$, $\|
		V \| (E) > 0$} \}
	\end{equation*}
	is termed a \emph{partition} of $V$ if and only if
	\begin{equation*}
		V (k) = \sum_{W \in \Pi} W (k) \quad \text{and} \quad \|
		\updelta V \| (f) = \sum_{W \in \Pi} \| \updelta W \| (f)
	\end{equation*}
	for $k \in \mathscr K ( U \times \mathbf G (n,m) )$ and $f \in
	\mathscr K ( U)$.
\end{intro-definition}

In view of \ref{remark:char-partition}, the previous notions of
indecomposability, component, and decomposition from \cite[6.2, 6.6,
6.9]{MR3528825} can be rephrased: $V$ is \emph{indecomposable} if and only if
either $V = 0$ or $\{ V \}$ is the only partition of $V$; $W$ is a
\emph{component} of $V$ if and only if $W$ is an indecomposable member of some
partition of $V$; and, $\Xi$ is a \emph{decomposition} of $V$ if and only if
$\Xi$ is a partition of $V$ consisting of components of $V$.  Decompositions
were introduced for the study of $\mathbf T (V)$ to characterise those members
$f$ of that space which have vanishing generalised $V$ weak derivative, $V
\weakD f$.  In fact, in the case $V$ is rectifiable, \emph{$f \in \mathbf T
(V)$ satisfies $V \weakD f = 0$ if and only if there exist some decomposition
$\Xi$ of $V$ and $\upsilon : \Xi \to \mathbf R$ such that
\begin{equation*}
	f(x) = \upsilon (W) \quad \text{for $\| W \| + \| \updelta W \|$
	almost all $x$},
\end{equation*}
whenever $W \in \Xi$}, see \cite[6.10, 8.24, 8.34]{MR3528825}.  Decompositions
are in general non-unique; in fact, one may consider three lines intersecting
in a common point at equal angles, see \cite[6.13]{MR3528825}.  For
nonrectifiable $V$, elementary examples show that decompositions may fail to
exist even if $\updelta V = 0$, see \cite[4.12\,(3)]{MR3777387}.

The structure of functions in $\mathbf T (V)$ with non-vanishing generalised
$V$ weak derivative may be more complex:  Whereas it is possible to define a
member $f$ of $\mathbf T (V)$ by firstly selecting a partition $\Pi$ of $V$
and then, for each $W \in \Pi$, a member of $\mathbf T (W)$ subject only to a
natural summability condition, see \cite[4.14]{MR3777387} and
\ref{remark:partition-earlier}, our first example shows that those $f$
may not exhaust $\mathbf T (V)$.

\begin{intro-example} [see \protect{\ref{example:no-adapted-decomposition}}]
	\label{Example:no-adapt}
	There exists a nonzero two-dimensional varifold
	$X$ in $\mathbf R^3$ with $\| \updelta X \| \leq \kappa \| X \|$ for
	some $0 \leq \kappa < \infty$ and $\boldsymbol \Uptheta^2 ( \| X \|, x
	) = 1$ for $\| X \|$ almost all $x$ such that, for some $f \in \mathbf
	T ( X )$, the function $f$ does not belong to $\mathbf T (W)$ for any
	component $W$ of $X$.
\end{intro-example}

This varifold $X$ has four components, each associated with a properly
embedded two-dimensional submanifold of class $\infty$ of
$\mathbf R^3$, and two distinct decompositions;  there is a straight line
which is the intersection of the submanifolds corresponding to the two
components of each decomposition; $f$ possesses generalised bounded variation
with respect to $W$ for each component $W$ of $X$; and the singular parts of
the distributional derivative of $f$ with respect to $W$---that is, those not
representable with respect to $F_\# \| W \|$---are concentrated along this
line so as to cancel after summing over the two members $W$ of any
decomposition of $X$.  We have included Theorem \ref{Thm:GBV} in our present
paper because it allows us to employ the geometric viewpoint of boundaries of
subgraphs in this construction.

\subsection*{Examples of decomposable varifolds}

We shall demonstrate that---when viewed by means of the existing notion of
indecomposability---the varifolds associated with compact $(\mathscr H^m,m)$
rectifiable sets, indecomposable $G$ chains (see
\ref{definition:indecomposable_chains}), and immersions may fail to inherit
the connectedness properties of these objects. We firstly introduce a related
concept.

\begin{intro-definition} [see
	\protect{\ref{def:locally-finiteness-decompositions}}]
	A family $\Xi$ of $m$ dimensional varifolds in $U$ is termed
	\emph{locally finite} if and only if
	\begin{equation*}
		\card ( \Xi \cap \{ W \with K \cap \spt \| W \| \neq
		\varnothing \} ) < \infty
	\end{equation*}
	whenever $K$ is a compact subset of $U$.
\end{intro-definition}

\begin{intro-example} [Two touching spheres, see 
	\ref{example:indecomposability_of_type_f} and
	\ref{remark:indecomposability_of_type_f}]
	\label{intro-example:touching-spheres} Suppose $m=n-1$, $A = \{ a_1,
	a_2 \} \subset \mathbf R^n$, where $a_1 = 0$ and $a_2 = (0, \ldots, 0,
	2 )$,
	\begin{equation*}
		M = \mathbf R^n \cap \{ x \with \dist (x,A) = 1 \},
	\end{equation*}
	and $V$ is the $m$ dimensional varifold in $\mathbf R^n$ associated
	with $M$.
	
	Then, $V$ satisfies the \hyperlink{II-gen-hyp}{general hypotheses},
	$\| \updelta V \| \leq m \| V \|$, $M = \spt \| V \|$ is a connected
	compact $(\mathscr H^m,m)$ rectifiable set, and $V$ is decomposable.
\end{intro-example}

In fact, $V$ admits a unique decomposition whose elements correspond to the
varifolds associated with the two  spheres whose union equals $M$.  In the
case $m=1$, slightly modifying the shape of the set $M$, one may require that
$V$ is associated with some figure-eight immersion $F : \mathbf S^1 \to
\mathbf R^2$ of class $\infty$.

\begin{intro-example} [Three line segments, see \ref{example:four-arcs}]
	\label{intro-example:three-line-segments}
	Suppose $B = \{ b_1,b_2, b_3 \} \subset \mathbf R^2$, where $b_1 =
	(1,0)$, $b_2 = (0,1)$, and $b_3 = (-1,0)$,
	\begin{equation*}
		M = \{ t b \with 0 \leq t \leq 1, b \in B \},
	\end{equation*}
	and $V$ is the one-dimensional varifold in $\mathbf R^2$ associated
	with $M$.
	
	Then, $V$ is a varifold satisfying the \hyperlink{II-gen-hyp}{general
	hypotheses}, $V$ is associated with an indecomposable integral
	$\mathbf Z/3 \mathbf Z$ chain, and $V$ is decomposable.
\end{intro-example}

The unique decomposition of this varifold $V$ consists of the varifolds
associated with the two line segments $\{ tb_1+(1-t)b_3 \with 0 \leq t \leq 1
\}$ and $\{ t b_2 \with 0 \leq t \leq 1 \}$.

Similar situations do occur with integer coefficients; in fact, using a
variant of the map $f$ in \cite[p.\,426]{MR41:1976}, one may construct two
real projective planes with $m=2$ embedded into $\mathbf R^6$ touching along a
common bounding projective line.  In Examples
\ref{intro-example:touching-spheres} and
\ref{intro-example:three-line-segments}, $\spt \| V \|$ is connected only
through a single point.  This illustrates the challenge involved in creating a
\emph{measure-theoretic} notion of connectedness for varifolds which treats
these varifolds as connected and is yet powerful enough to entail meaningful
geometric consequences; this will be accomplished by the present paper and the
final paper, see~\cite{arXiv:1709.05504v3}.

We now exhibit an example of a varifold associated with a properly
\emph{immersed} submanifold-with-boundary in which decompositions are not
locally finite.  Even for general varifolds not necessarily associated with an
immersion, this behaviour can be precluded by either of the
natural Regularity hypotheses \ref{reg-hyp:Dirichlet} and
\ref{reg-hyp:Neumann} below involving properly \emph{embedded} boundary data
of class $2$ of Dirichlet or Neumann type, respectively, see Theorems
\ref{Thm:criterion-local-finiteness} and
\ref{Thm:criterion-local-finiteness-II} below.

\begin{intro-example} [see \ref{example:inifinite_components}]
	\label{Thm:immersion}
	There exist a two-dimensional man\-i\-fold-with-bound\-ary $M$ and a
	proper immersion $F : M \to \mathbf R^3$ of class $\infty$ such that
	no decomposition of the two-dimensional varifold in $\mathbf R^3$
	associated with $F$ is locally finite.
\end{intro-example}

One may partially visualise the construction as follows: $M$ equals the
disjoint union of four half-planes and one plane; $F$ restricted to any of
these components of $M$ is an embedding; the images of the four half-planes
meet the isometrically embedded plane at an angle of $60^\circ$ in two curved
lines (two half-planes per line); the two lines tangentially meet in a
prescribed closed set; the two lines enclose an infinite number of
accumulating open topological discs between each other; and, each of these
discs occurs as component in the unique decomposition of $V$.

\subsection*{Properties of indecomposability with respect to a family of
generalised weakly differentiable real valued functions}

The most central concept for our present study is the following notion of
indecomposability of $V$ with respect to a subfamily $\Psi$ of $\mathbf T (
V)$.

\begin{intro-definition} [see \ref{definition:indecomposability}]
	If $V$ satisfies the \hyperlink{II-gen-hyp}{general
	hypotheses} and $\Psi \subset \mathbf T (V)$, then we term $V$
	\emph{indecomposable of type $\Psi$} if and only if, whenever $f \in
	\Psi$, the set of $y \in \mathbf R$ such that $E(y) = \{ x \with f(x)
	> y \}$ satisfies
	\begin{equation*}
		\| V \| ( E (y)) > 0, \quad \| V \| ( U \without E (y) ) > 0,
		\quad V \boundary E (y) = 0
	\end{equation*}
	has $\mathscr L^1$ measure zero.
\end{intro-definition}

Equivalently, only for $y$ in an exceptional set of $\mathscr L^1$ measure
zero, we allow
\begin{equation*}
	\{ V \restrict E (y) \times \mathbf G (n,m), V \restrict ( U \without
	E (y)) \times \mathbf G (n,m) \}
\end{equation*}
to be a partition of $V$.  The smaller $\Psi$, the less restrictive is the
corresponding notion of indecomposability.  We shall consider the following
six families as $\Psi$.
\begin{center}
	\begin{tabular}{rl}
		$\mathbf T (V)$ & Generalised $V$ weakly differentiable real
		valued functions.\\
		$\Gamma$ & Continuous functions $f \in \mathbf T(V)$ with
		domain $\spt \|V\|$. \\
		$\Lambda$ & Locally Lipschitzian functions $f : U \to \mathbf
		R$. \\
		$\mathscr E (U,\mathbf R )$ & Functions $f : U \to \mathbf R$
		of class $\infty$. \\
		$\mathscr D ( U, \mathbf R )$ & Functions $f : U \to \mathbf
		R$ of class $\infty$ with compact support in $U$. \\
		$\{ f \}$ & A particular function $f \in \mathbf T (V)$.
	\end{tabular}
\end{center}
The strongest condition---indecomposability of type $\mathbf T
(V)$---coincides with the existing notion of indecomposability, see
\ref{remark:indecomposability_implications}.  In the case that $V$ is
rectifiable and every decomposition of $V$ is locally finite, connectedness of
$\spt \| V \|$ implies indecomposability of type $\Gamma$, see Theorem
\ref{Thm:indecomposability-type-G} below; in particular, the decomposable
varifolds of Examples \ref{intro-example:touching-spheres} and
\ref{intro-example:three-line-segments} are indecomposable of type $\Gamma$.
Indecomposability of type $\Lambda$ is expedient for varifolds associated with
immersions and integral $G$ chains, see Theorems
\ref{Thm:connected-immersions} and \ref{Thm:Indecomposability-G-chains} below,
respectively.  Indecomposability of type $\mathscr E ( U, \mathbf R )$
suffices to guarantee connectedness of $\spt \| V \|$, see
\ref{thm:indecomposability_connected}.  Indecomposability of type $\mathscr D
( U, \mathbf R )$ will be employed in the final paper of our series (see
\cite{arXiv:1709.05504v3}) to establish---for varifolds satisfying a uniform
lower density bound, local $p$-th power summability of their mean curvature,
and a boundary condition---two types of varifold-geometric results: lower
density bounds $\mathscr H^{m-p}$ almost everywhere, provided $m-1 \leq p <
m$, and an a priori bound on the geodesic diameter of $\spt \| V \|$.  Theorem
\ref{Thm:Indecomposability-G-chains} and Corollaries \ref{Thm:E} and
\ref{Thm:Neumann} below then entail that these results are applicable to
various geometric variational problems including a variety of Plateau
problems.  The latter refers to the classical formulations of Reifenberg and
Federer-Fleming as well as to those solutions obtained using integral $G$
chains, min-max methods, and Brakke flow.  Finally, the notion of
indecomposability of type $\{ f \}$ serves as a tool to study the behaviour of
$V$ with regard to a fixed $f \in \mathbf T(V)$.

For varifolds $V$ associated with $m$ dimensional properly embedded
submanifolds $M$ of class $2$ of $U$, see \ref{example:indecomposability},
indecomposability of type $\mathscr E (U, \mathbf R )$ is equivalent to both
previously existing notions---connectedness of $M$ and indecomposability of
$V$; moreover, such $V$ is indecomposable of type $\mathscr D ( U, \mathbf R)$
if and only if either $M$ is connected or no connected component of $M$ is
compact.

Whereas indecomposability of types $\mathbf T (V)$, $\mathscr E ( U, \mathbf
R)$, $\mathscr D (U, \mathbf R )$, and $\{ f \}$ all differ by the examples
already discussed (and the fact that $V$ is always indecomposable of type $\{
0 \}$), we did not study whether the families $\Gamma$, $\Lambda$, and
$\mathscr E (U,\mathbf R )$ yield different types of indecomposability; in the
case that $V$ is rectifiable and every decomposition of $V$ is locally finite,
these three types of indecomposability are all equivalent to connectedness of
$\spt \| V \|$ by Corollary \ref{Corollary:equivalence} below.

Indecomposability of type $\Psi$ may be exploited by coarea type arguments.
The following theorem serves as an example of this method and is the
foundation for the varifold-geometric results in the final paper of our
series~\cite{arXiv:1709.05504v3}.

\begin{Thm} [see \protect{\ref{lemma:basic_indecomp}%
		\,\eqref{item:basic_indecomp:interval}%
		\,\eqref{item:basic_indecomp:constant} and \ref{thm:utility}}]
	\label{Thm:consequences-indecomposability-singleton}
	Suppose $f \in \mathbf T (V)$, $V$ is indecomposable of type $\{ f
	\}$, $Y \subset \mathbf R$, and $f(x) \in Y$ for $\| V \|$ almost all
	$x$.
	
	Then, $\spt f_\# \| V \|$ is a subinterval of $\mathbf R$ whose
	diameter is bounded by $\mathscr L^1 (Y)$.  Moreover, if $V \weakD f =
	0$, then $f$ is $\| V \| + \| \updelta V \|$ almost constant.
\end{Thm}

By way of contrast, in the case $V$ admits a countably infinite partition and
$Y$ is a countable dense subset of $\mathbf R$, there exists $f \in \mathbf T
( V)$ with $V \weakD f = 0$, $\im f = Y$, and $\spt f_\# \| V \| = \mathbf R$
by \cite[4.14]{MR3777387}.  This shows the substance of the assumed
indecomposability of type $\{ f \}$ and the diameter bound established.

Refining indecomposability in the opposite direction, a more stringent notion
for integral varifolds $V$ was formulated in \cite[2.15]{MR3148123} and
subsequently studied in \cite{MR4609162}: Such $V$ is
termed \emph{integrally indecomposable} if and only if there exists no $m$
dimensional integral varifold $W$ in $U$ such that $0 \neq W \neq V$,
\begin{equation*}
	\| V \| = \| W \| + \| V-W \|, \quad \| \updelta V \| = \| \updelta W
	\| + \| \updelta (V-W) \|.
\end{equation*}
The purpose of \cite{MR4609162}---a foundation to merge
key elements of \cite[Chapters 1 and 2]{MR1777737} and \cite{MR3528825}---is
quite different from ours.

\subsection*{Unique partition along a generalised weakly differentiable real
valued function}

We next introduce two concepts expedient in deriving criteria for
indecomposability of type $\Psi$.  Firstly, the situation described in Example
\ref{Example:no-adapt} led us to develop the notion of partition along a
member of $\mathbf T (V)$ as substitute.

\begin{intro-definition} [see \protect{\ref{definition:partition-along}}]
	For $V$ satisfying the \hyperlink{II-gen-hyp}{general hypotheses} and
	$f \in \mathbf T (V)$, a subfamily $\Pi$ of
	\begin{equation*}
		\big \{ V \restrict f^{-1} [I] \times \mathbf G (n,m) \with
		\text{$I$ subinterval of $\mathbf R$, $V \boundary f^{-1} [I]
		= 0$} \big \}
	\end{equation*}
	is termed \emph{partition of $V$ along $f$} if and only if $\Pi$ is a
	partition of $V$ and, whenever $W \in \Pi$, there exists no partition
	of $\mathbf R$ into subintervals $J_1$ and $J_2$ such that
	\begin{equation*}
		\big \{ W \restrict f^{-1} [J_1] \times \mathbf G (n,m), W
		\restrict f^{-1} [J_2] \times \mathbf G (n,m) \big \}
	\end{equation*}
	is a partition of $W$.
\end{intro-definition}

In contrast to the behaviour of general partitions exhibited in
Example \ref{Example:no-adapt}, we always have $f \in \mathbf T (W)$ for
members $W$ of a partition of $V$ along $f$ by \ref{remark:partition-along}.
Moreover, each member of a partition along $f$ is indecomposable of type $\{
f \}$, see \ref{remark:partition-along}, but a varifold $V$ indecomposable of
type $\{ f \}$ may still admit a nontrivial partition along $f$, see
\ref{example:indecomposability-along}. The different behaviour corresponds to
the absence of an exceptional set in the definition of partition along $f$.

\begin{Thm} [see \protect{\ref{thm:uniqueness-partition-along} and
	\ref{thm:decomposition_adapted_to_fct}}]
	\label{Thm:partition-along-fct}
	Under the \hyperlink{II-gen-hyp}{general hypotheses}, let $f \in
	\mathbf T (V)$.
	
	Then, there exists at most one partition of $V$ along $f$; if $V$ is
	rectifiable, then a partition of $V$ along $f$ exists.
\end{Thm}

The uniqueness established distinguishes the behaviour of partitions of $V$
along $f$ from that of decompositions, whereas the existence proof elaborates
on the machinery yielding existence of decompositions.  For non rectifiable
varifolds, similar to decompositions, existence may fail in simple examples,
see \ref{example:non-existence-partition-along}.

\subsection*{Criteria for local finiteness of decompositions}

Secondly, we discuss two sets of regularity
hypotheses---tailored to cover Dirichlet and Neumann boundary data,
respectively---which allow us to study $\spt \| V \|$ by providing criteria
for local finiteness of decompositions of $V$, see Theorems
\ref{Thm:criterion-local-finiteness} and
\ref{Thm:criterion-local-finiteness-II} below. For this purpose, we recall
the following definition.

\begin{intro-definition} [see \protect{\cite[2.4.12]{MR41:1976}}]
	Suppose $\phi$ measures $X$, $Y$ is a Banach space, $1 \leq p \leq
	\infty$, and $f$ is a $\phi$ measurable $Y$ valued function.
	
	Then, the quantity $\phi_{(p)}(f)$ is defined by
	\begin{align*}
		\phi_{(p)}(f) & = \big ( {\textstyle \int |f|^p \ud \phi}
		\big)^{1/p} \quad \text{if $p < \infty$}, \\
		\phi_{(p)}(f) & = \inf \big \{ s \with s \geq 0, \phi \, \{ x
		\with |f(x)|>s \} = 0 \big \} \quad \text{if $p = \infty$}.
	\end{align*}
\end{intro-definition}

\begin{reg-hyp} [see \ref{example:regularity-hypotheses}]
	\label{reg-hyp:Dirichlet}
	Suppose $m$ and $n$ are positive integers, $m \leq n$, $U$ is an open
	subset of $\mathbf R^n$, $B$ is a properly embedded $m-1$ dimensional
	submanifold of class $2$ of $U$, $X$ is an $m$ dimensional varifold in
	$U$,
	\begin{gather*}
		\text{$\beta = \infty$ if $m = 1$}, \quad \text{$\beta =
		m/(m-1)$ if $m>1$}, \\
		\sup \big \{ ( \updelta X ) ( \theta ) \with \theta \in
		\mathscr D ( U, \mathbf R^n ), \spt \theta \subset K \without
		B, \| X \|_{(\beta)} (\theta) \leq 1 \big \} < \infty
	\end{gather*}
	whenever $K$ is a compact subset of $U$, $\| X \| ( B ) = 0$, and
	\begin{equation*}
		\boldsymbol \Uptheta^m ( \| X \|, x ) \geq 1 \quad \text{for
		$\| X \|$ almost all $x$};
	\end{equation*}
	hence, \emph{$\| \updelta X \|$ is a Radon measure
	over $U$ and $X$ is rectifiable.}
\end{reg-hyp}

Using the inclusion map $i : U \without B \to U$, we may equivalently require
that, for some $m$ dimensional rectifiable varifold $W$ in $U \without B$ such
that
\begin{gather*}
	\boldsymbol \Uptheta^m ( \| W \|, x ) \geq 1 \quad \text{for $\| W \|$
	almost all $x$}, \\
	\text{$i_\# ( \| W \| + \| \updelta W \| )$ is a Radon measure over
	$U$}, \\
	\text{$\| \updelta W \|$ is absolutely continuous with respect to $\|
	W \|$ if $m > 1$}, \\
	\text{$i_\# ( \| W \| \restrict | \mathbf h ( W, \cdot ) |^m )$ is a
	Radon measure over $U$ if $m > 1$},
\end{gather*}
where $\mathbf h (W, \cdot )$ denotes the mean curvature vector of $W$, there
holds $i_\# W = X$.

\begin{Thm} [see \ref{thm:locally_finite_near_boundary}]
	\label{Thm:criterion-local-finiteness}
	Suppose $X$ satisfies the
	Regularity hypotheses \ref{reg-hyp:Dirichlet}.
	
	Then, every decomposition of $X$ is locally
	finite.
\end{Thm}

For $B = \varnothing$, this was obtained in \cite[6.11]{MR3528825} where also
the sharpness of the summability condition on the mean curvature vector was
noted.  Thus, the difficulty stems from the boundary $B$.  For this purpose,
we additionally make use of the concept of reach and isoperimetric
considerations near the boundary, see \cite[Section 4]{MR0110078} and
\cite[Section 3]{MR0397520}, respectively.

\begin{reg-hyp} [see \ref{remark:Neumann-rectifiability} and
	\ref{example:Neumann-Lm} with $p = m$]
	\label{reg-hyp:Neumann}
	Suppose $m$ and $n$ are integers, $2 \leq m \leq n$, $U$ is an open
	subset of $\mathbf R^n$, $M$ is a relatively closed subset of $U$ and
	an $n$ dimensional submanifold-with-boundary of class $2$
	of $\mathbf R^n$, $B = \partial M$, $X \in \mathbf
	V_m (U)$, $\spt \| X \| \subset M$, $\Theta$ consists of all $\theta :
	U \to \mathbf R^n$ of class $1$ with compact support and $\theta (b)
	\in \Tan (B,b)$ for $b \in B$, $\beta = m/(m-1)$,
	\begin{equation*}
		\sup \{ {\textstyle\int S_\natural \bullet \Der \theta (x) \ud
		X \, (x,S)} \with \theta \in \Theta, \spt \theta \subset K, \|
		X \|_{(\beta)} ( \theta ) \leq 1 \big \} < \infty
	\end{equation*}
	whenever $K$ is a compact subset of $U$, and
	\begin{equation*}
		\boldsymbol \Uptheta^m ( \| X \|, x ) \geq 1 \quad \text{for
		$\| X \|$ almost all $x$};
	\end{equation*}
	hence, \emph{$\| \updelta X \|$ is a Radon measure over $U$, $\langle
	\mathbf h (X,\cdot), \tau \rangle \in \mathbf L_m^{\textup{loc}} ( \|
	X \|, \mathbf R^n )$, and $X$ is rectifiable,} where $\tau : U \to
	\Hom ( \mathbf R^n, \mathbf R^n )$ is defined by
	\begin{equation*}
		\text{$\tau (x) = \mathbf 1_{\mathbf R^n}$ for $x \in U
		\without B$}, \quad \text{$\tau (x) = \Tan (B,x)_\natural$ for
		$x \in B$}.
	\end{equation*}
\end{reg-hyp}
 
This set of hypotheses for instance occurs in the study of
constrained free boundary problems (see
\ref{remark:free-boundary-constrained}), of minimal sets with sliding boundary
conditions (see \ref{remark:free-boundary-sliding}), of the Willmore
functional with free boundary (see \ref{remark:free-boundary-Willmore}), and
of Brakke, or level set mean curvature, flow with free boundary (see
\ref{remark:free-boundary-MCF}).

\begin{Thm} [see \ref{corollary:Neumann-locally-finite-decompositions}]
	\label{Thm:criterion-local-finiteness-II}
	Suppose $X$ satisfies the Regularity hypotheses \ref{reg-hyp:Neumann}.
	
	Then, every decomposition of $X$ is locally finite.
\end{Thm}

Similar to Theorem \ref{Thm:criterion-local-finiteness}, the preceding theorem
relies on the concept of reach and isoperimetric considerations near the
boundary.  However, in the present setting, an isoperimetric lower density
ratio bound analogous to \cite[3.4\,(1)]{MR0397520} was not available;
therefore, we establish it in \ref{thm:Neumann-lower-density-ratio-bound}.

\subsection*{Criteria for indecomposability with respect to a family of
generalised weakly differentiable real valued functions}

The following two theorems concern varifolds associated with connected
immersions and indecomposable $G$ chains, respectively.  For this purpose, we
recall that $\Lambda$ denotes the class of all locally Lipschitzian functions
$f : U \to \mathbf R$.

\begin{Thm} [see \protect{\ref{thm:indecomposability_immersions}}]
	\label{Thm:connected-immersions}
	Suppose $M$ is a connected $m$ dimensional manifold-with-boundary of
	class $2$, $F : M \to U$ is a proper immersion of class $2$, and $X$
	is the $m$ dimensional varifold in $U$ associated with $F$.
	
	Then, $X$ is indecomposable of type $\Lambda$.
\end{Thm}

\begin{Thm} [see \protect{\ref{thm:indecomposable_chains}}]
	\label{Thm:Indecomposability-G-chains}
	Suppose $G$ is a complete normed commutative group, $S$ is an
	indecomposable $m$ dimensional integral $G$ chain in $U$, $X$ is the
	$m$ dimensional varifold in $U$ associated with $S$, and $\updelta X$
	is representable by integration.
	
	Then, $X$ is indecomposable of type $\Lambda$.
\end{Thm}

Both theorems ultimately rely on the general coarea formulae obtained in
\cite[12.1]{MR3528825} and \cite[2.33, 2.34]{DOI_10.4171_RMI_1487}, which
entail that, whenever $f \in \Lambda$, we have
\begin{equation*}
	\| V \boundary \{ x \with f(x)>y \} \| = \big ( \mathscr H^{m-1}
	\restrict \{ x \with f(x)=y \} \big ) \restrict \boldsymbol \Uptheta^m
	( \| V \|, \cdot)
\end{equation*}
for $\mathscr L^1$ almost all $y$.  For Theorem
\ref{Thm:connected-immersions}, this is combined with a constancy theorem
based on \cite[4.5.11]{MR41:1976}, see \ref{thm:constancy}.  For Theorem
\ref{Thm:Indecomposability-G-chains}, we instead employ the basic theory of
slicing for integral $G$ chains as recorded in \cite[3.8,
4.13\,(8)]{DOI_10.4171_RMI_1487}.

Capitalising on the topological notion of connectedness of $\spt \| V \|$ is
more subtle.  For this purpose, we employ the family
\begin{equation*}
	\Gamma = \mathbf T (V) \cap \{ f \with \text{$\dmn f = \spt \| V \|$,
	$f$ continuous} \}
\end{equation*}
and assume all decompositions of the varifold in question to be locally
finite. This additional hypothesis distinguishes the purely varifold-geometric
setting of the next theorem and its corollaries from the more specific
Theorems \ref{Thm:connected-immersions} and
\ref{Thm:Indecomposability-G-chains}.

\begin{Thm} [see \protect{\ref{thm:indecomposable_type_Psi}}]
	\label{Thm:indecomposability-type-G}
	Suppose that $X$ is an $m$ dimensional rectifiable varifold in $U$,
	that $\updelta X$ is representable by integration, that every
	decomposition of $X$ is locally finite, and that $\spt \| X \|$ is
	connected.
	
	Then, $X$ is indecomposable of type $\Gamma$.
\end{Thm}

By \ref{lemma:spt-partitions}, the local finiteness of the decompositions of
$X$ ensures that
\begin{equation*}
	\spt \| X \| = \bigcup \{ \spt \| W \| \with W \in \Pi \} \quad
	\text{whenever $\Pi$ is a partition of $X$};
\end{equation*}
otherwise, $\spt \| X \|$ could be larger than this union.  For $f \in
\Gamma$, we then consider the partition $\Pi$ of $X$ along $f$ furnished by
Theorem \ref{Thm:partition-along-fct} and the subintervals $J(W) = \spt f_\#
\| W \|$ of $\mathbf R$ corresponding to $W \in \Pi$, see Theorem
\ref{Thm:consequences-indecomposability-singleton}.  Exploiting the equation
for the supports and the continuity of $f$, we finally prove that the set
\begin{equation*}
	B = \big \{ y \with \text{$\| X \| ( E(y)) > 0$, $\| X \| ( U \without
	E(y)) > 0$, $X \boundary E(y) = 0$} \big \},
\end{equation*}
where $E(y) = \{ x \with f(x)>y \}$, is in fact contained in the countable set
of boundary points of the intervals $J(W)$ corresponding to $W \in \Pi$;
hence, $\mathscr L^1 (B) = 0$.  Thus, the proof combines local finiteness of
decompositions with partitions along $f$.

Combining \ref{thm:indecomposability_connected} and Theorem
\ref{Thm:indecomposability-type-G} yields the following equivalence.
	
\begin{intro-corollary} [see
	\protect{\ref{corollary:indecomposability-equivalence}}]
	\label{Corollary:equivalence}
	Suppose that $V$ satisfies the \hyperlink{II-gen-hyp}{general
	hypotheses}, $V$ is rectifiable, and every decomposition of $V$ is
	locally finite.
		
	Then, connectedness of $\spt \| V \|$ is equivalent to
	indecomposability of $V$ of type $\Gamma$ and to indecomposability of
	$V$ of type $\mathscr E ( U, \mathbf R )$.
\end{intro-corollary}
	
Combining Theorems \ref{Thm:criterion-local-finiteness} and
\ref{Thm:indecomposability-type-G} implies the following corollary.

\begin{intro-corollary} [see
	\protect{\ref{corollary:indecomposable_type_Gamma}}]
	\label{Thm:E}
	Suppose $X$ satisfies the Regularity hypotheses
	\ref{reg-hyp:Dirichlet}, the set $\spt \| X \|$ is connected, and
	$\Gamma = \mathbf T (X) \cap \{ f \with \textup{$\dmn f = \spt \| X
	\|$, $f$ continuous} \}$.
	
	Then, $X$ is indecomposable of type $\Gamma$.
\end{intro-corollary}

The hypotheses of the preceding corollary for instance apply to the study of
the Willmore energy, see \ref{remark:clamped-Willmore}. In the final paper of
our series \cite{arXiv:1709.05504v3}, we will demonstrate the same for
solutions of Plateau's problem stemming from Reifenberg's formulation, Brakke
flow, and min-max methods.  For solutions based on integral currents or---more
generally---integral $G$ chains, we shall employ Theorem
\ref{Thm:Indecomposability-G-chains} instead.

Finally, combining Theorems
\ref{Thm:criterion-local-finiteness-II} and \ref{Thm:indecomposability-type-G}
entails the following corollary.
	
\begin{intro-corollary} [see
	\protect{\ref{corollary:Neumann-indecomposable}}]
	\label{Thm:Neumann}
	Suppose $X$ satisfies the Regularity hypotheses \ref{reg-hyp:Neumann},
	the set $\spt \| X \|$ is connected, and $\Gamma = \mathbf T (X) \cap
	\{ f \with \textup{$\dmn f = \spt \| X \|$, $f$ continuous} \}$.
	
	Then, $X$ is indecomposable of type $\Gamma$.
\end{intro-corollary}

\subsection*{Possible lines of further study}

A natural goal would be the development of a theory of real valued functions
possessing generalised bounded variation with respect to $V$ analogous to that
of generalised $V$ weakly differentiable functions laid out in \cite[Sections
8--13]{MR3528825}, \cite[Section 4]{MR3626845}, and \cite[Sections 4 and
5]{MR3777387}; ideally, such a theory would include a definition of
generalised $V$ bounded variation for functions with values in a
finite-dimensional normed space.  In this process, the concept of partitions
of $V$ along $f$ and indecomposability of type $\{ f \}$ could be defined---in
the obvious way---and investigated for real valued $\| V \| + \| \updelta V
\|$ measurable functions $f$ having generalised bounded variation with respect
to $V$.

\subsection*{Acknowledgements}

U.M.\ is grateful to Professor Felix Schulze for discussing
differential-geometric aspects of this note and to Professors Michael Goldman
and Benoît Merlet for discussing set-indecomposability.  One question answered
in the present paper (see \ref{remark:indecomposability_immersions}) first
appeared in the MSc thesis of the second author (see \cite[Section
A]{scharrer:MSc}).  Parts of the research described in this paper were carried
out while U.M.\ worked at the Max Planck Institutes for Gravitational Physics
(Albert Einstein Institute) and for Mathematics in the Sciences as well as the
Universities of Potsdam and Leipzig; further parts were compiled during
extended research visits at the Universities of Zurich and Bonn for which
U.M.\ is grateful to Professors Camillo De Lellis and Stefan Müller,
respectively; in Taiwan (R.O.C.), yet further parts were carried out while
U.M.\ was supported by the grants with nos.\ MOST 108-2115-M-003-016-MY3, MOST
110-2115-M-003-017, MOST 111-2115-M-003-014, 
NSTC 112-2115-M-003-001, and NSTC 113-2115-M-003-018-MY3
by the National Science and Technology Council (formerly termed Ministry of
Science and Technology) and as Center Scientist at the National Center for
Theoretical Sciences.

Parts of this paper were written while C.S.\ worked at the Max Planck
Institute for Gravitational Physics (Albert Einstein Institute) and the Max
Planck Institute for Mathematics. Further parts were written while C.S.\ was
supported by the EPSRC as part of the MASDOC DTC at the University of Warwick,
Grant No.\ EP/HO23364/1. 

\section{Notation}

\paragraph{Basic sources} As in the first paper of our
series~\cite{DOI_10.4171_RMI_1487}, our notation follows \cite{MR3528825} and
is thus largely consistent with H.\ Federer's terminology for geometric
measure theory listed in \cite[pp.\,669--676]{MR41:1976} and W.\ Allard's
notation for varifolds introduced in \cite{MR0307015}.  This includes, for
certain distributions $S$ and sets $A$, the definition of the variation
measure $\| S \|$ and of the restriction $S \restrict A$ of $S$ to $A$, see
\cite[2.18, 2.20]{MR3528825}, respectively; and, for certain varifolds $V$ and
sets $E$, the notions of \emph{distributional boundary}, $V \boundary E$, of
$E$ with respect to $V$, \emph{indecomposability}, \emph{component}, and
\emph{decomposition}, see \cite[5.1, 6.2, 6.6, 6.9]{MR3528825}, respectively.

\paragraph{Review} Here, we list symbols not already reviewed in
\cite[Introduction, Section 1]{MR3528825}: The determinant, $\det f$, of an
endomorphism $f$, see \cite[1.3.4]{MR41:1976}; the Grassmann manifold,
$\mathbf G(n,m)$, of all $m$ dimensional subspaces of $\mathbf R^n$, see
\cite[1.6.2]{MR41:1976}; the norms associated with inner products, $| \cdot
|$, see \cite[1.7.1]{MR41:1976}; the adjoint, $f^*$, of a linear map $f$
between inner product spaces, see \cite[1.7.4]{MR41:1976}; the seminorm, $\|
\cdot \|$, on $\Hom (V,W)$ associated with normed spaces $V$ and $W$, see
\cite[1.10.5]{MR41:1976}; the infimum and supremum, $\inf S$ and $\sup S$, of
a subset $S$ of the extended real numbers and the $\sign$ function, denoted
$\sign$, see \cite[2.1.1]{MR41:1976}; the class $\mathbf 2^X$ of all subsets
of $X$, see \cite[2.1.2]{MR41:1976}; the restriction, $\phi \restrict A$, of a
measure $\phi$ to a set $A$, see \cite[2.1.2]{MR41:1976}; the support,
abbreviated $\spt \phi$, of a measure $\phi$, see \cite[2.2.1]{MR41:1976}; the
Lebesgue spaces, $\mathbf L_p ( \phi, Y )$ and $\mathbf L_p ( \phi )$, and the
dual, $Y^\ast$, of a Banach space $Y$, see \cite[2.4.12]{MR41:1976}; the Dirac
measure, $\boldsymbol \updelta_a$, at $a$, see \cite[2.5.19]{MR41:1976}; the
$n$ dimensional Lebesgue measure, $\mathscr L^n$, see \cite[2.6.5]{MR41:1976};
the $\mathscr L^n$ measure of the unit ball, $\boldsymbol \upalpha (n)$, see
\cite[2.7.16]{MR41:1976}; the diameter of $S$, $\diam S$, see
\cite[2.8.8]{MR41:1976}; the limit, $(V)\,\lim_{S \to x} f(S)$, of $f$ at $x$
with respect to a Vitali relation $V$, see \cite[2.8.16]{MR41:1976}; the
absolutely continuous part $\psi_\phi$ of a measure $\psi$ with respect to
$\phi$, see \cite[2.9.1]{MR41:1976};%
\begin{footnote}%
	{As was stipulated in \cite{MR3528825} for similar notions,
	$\psi_\phi$ will also be employed in the case where one of the
	measures $\psi$ or $\phi$ on $X$ fails to be finite on bounded sets
	but there exist open sets $U_1, U_2, U_3, \ldots$ covering $X$ at
	which $\phi$ and $\psi$ are finite.}
\end{footnote}%
the derivative, $g'$, of a function on the real line, see
\cite[2.9.19]{MR41:1976}; the $m$ dimensional Hausdorff measure, $\mathscr
H^m$, see \cite[2.10.2]{MR41:1976}; the integralgeometric measure with
exponent $1$, denoted $\mathscr I_1^m$, see \cite[2.10.5]{MR41:1976}; the
multiplicity function, $\Number (f,\cdot)$, see \cite[2.10.9]{MR41:1976}; the
$m$ dimensional densities, $\boldsymbol \Uptheta^{\ast m} ( \phi, a)$ and
$\boldsymbol \Uptheta^m (\phi, a )$, see \cite[2.10.19]{MR41:1976}; the
differential, $\Der f$, and the gradient, $\grad f$, see
\cite[3.1.1]{MR41:1976}; the closed cones of tangent and normal vectors, $\Tan
(S,a)$ and $\Nor (S,a)$, see \cite[3.1.21]{MR41:1976}; the
normal bundle, $\Nor (B)$, of a submanifold $B$ of class $1$ of $\mathbf R^n$,
see \cite[3.1.21]{MR41:1976}; the unit sphere, $\mathbf S^{n-1}$, in $\mathbf
R^n$, see \cite[3.2.13]{MR41:1976}; the space, $\mathscr E (U,Y)$, of
functions of class $\infty$, the support, $\spt \gamma$, of a member $\gamma$
of $\mathscr E (U,Y)$, and the space, $\mathscr D'(U,Y)$, of distributions in
$U$ of type $Y$,%
\begin{footnote}%
	{Our topology on $\mathscr D(U,Y)$ differs from that employed in
	\cite[4.1.1]{MR41:1976} but yields the same dual $\mathscr D'(U,Y)$ of
	continuous linear functionals $S : \mathscr D (U,Y) \to \mathbf R$,
	see \cite[2.13, 2.17\,(1)]{MR3528825}.}
\end{footnote}%
see \cite[4.1.1]{MR41:1976};
the boundary, denoted $\boundary T$, for currents $T$ and the Euclidean
currents, $\mathbf E^n$, see \cite[4.1.7]{MR41:1976}; the one-dimensional
current $\boldsymbol [ u,v \boldsymbol ]$ in $\mathbf R^n$ associated with the
line segment from $u$ to $v$, see \cite[4.1.8]{MR41:1976}; the group, $\mathbf
I_m ( \mathbf R^n )$, of $m$ dimensional integral currents in $\mathbf R^n$,
see \cite[4.1.24]{MR41:1976}; the exterior normal, $\mathbf n ( A, \cdot)$, of
$A$, see \cite[4.5.5]{MR41:1976}; the weight, $\| V \|$, of a varifold $V$,
see \cite[3.1]{MR0307015}; and, the first variation, $\updelta V$, of a
varifold $V$, see \cite[4.2]{MR0307015}.

\paragraph{Modifications} For the push forward, $f_\# \phi$, of a measure
$\phi$ by a function $f$, we use the definition 
\cite[2.9]{DOI_10.4171_RMI_1487}
which extends \cite[2.1.2]{MR41:1976}. Whenever $M$ is an
$m$ dimensional submanifold of class $1$ of $\mathbf R^n$,
\begin{equation*}
	\mathbf G_m ( M ) = ( M \times \mathbf G(n,m) ) \cap \{ (a,S) \with S
	\subset \Tan (M,a) \};
\end{equation*}
this extends \cite[2.5]{MR0307015}. Whenever $M$ is a submanifold of class
$2$ of $\mathbf R^n$, the \emph{second fundamental form of $M$ at $a$}, for $a
\in M$, and the \emph{mean curvature vector of $M$ at $(a,S)$}, for $(a,S) \in
\mathbf G_m (M)$, are defined to be the unique (symmetric)
bilinear form $\mathbf b (M,a) : \Tan (M,a) \times \Tan (M,a) \to \Nor (M,a)$
and the unique vector $\mathbf h (M,a,S) \in \Nor (M,a)$ such that
\begin{gather*}
	v \bullet \langle w, \Der g (a) \rangle = - \mathbf b (M,a) (v,w)
	\bullet g(a) \quad \text{for $v,w \in \Tan (M,a)$}, \\
	S_\natural \bullet ( \Der g(a) \circ \Tan (M,a)_\natural ) = - \mathbf
	h ( M,a,S) \bullet g(a)
\end{gather*}
whenever $g : M \to \mathbf R^n$ is of class $1$ relative to $M$ and satisfies
$g(x) \in \Nor (M,x)$ for $x \in M$; these notions extend
\cite[2.5]{MR0307015}. Whenever $m$ and $n$ are positive integers, $m \leq n$,
$U$ is an open subset of $\mathbf R^n$, and $E$ is a $\mathscr H^m$ measurable
set which meets every compact subset of $U$ in a $(\mathscr H^m,m)$
rectifiable set, we define the $m$ dimensional varifold associated with $E$,
denoted $\mathbf v_m (E)$, as the unique member of $\mathbf {RV}_m ( U )$
satisfying $\| \mathbf v_m ( E ) \| = \mathscr H^m \restrict E$; this modifies
the notation of \cite[3.5]{MR0307015}.  The concepts relating to generalised
weak differentiability on varifolds---that is, \emph{generalised $V$ weak
derivatives}, $V \weakD f$, of \emph{generalised $V$ weakly differentiable
functions $f$} and the resulting function spaces, $\mathbf T(V,Y)$ and
$\mathbf T(V)$,---are employed in accordance with the generalisation
\cite[4.2]{MR3777387} of \cite[8.3]{MR3528825}.

\paragraph{Amendments} Whenever $n \geq 2$, the trace of an $\mathbf R^n$
valued bilinear form $B$ on $\mathbf R^m$ is defined so that $\alpha ( \trace
B ) = \trace ( \alpha \circ B)$ whenever $\alpha \in \Hom ( \mathbf R^n,
\mathbf R)$.  Modelled on \cite[4.1.7]{MR41:1976}, we employ the restriction
notation, $\phi \restrict f$, for the measure $\phi$ weighted by $f$
introduced in \cite[2.6]{DOI_10.4171_RMI_1487}; this weighted measure was
discussed---without name---in \cite[2.4.10]{MR41:1976}.  For subsets $A$ of
Euclidean space, we adopt the concepts of unique nearest point and reach, and
the symbols $\Unp (A)$, $\reach (A, \cdot )$, and $\reach (A)$ from
\cite[4.1]{MR0110078} as well as those of \emph{pointwise} and
\emph{approximate differentiability of order $k$} with the corresponding
\emph{pointwise} and \emph{approximate differentials of order $k$}, denoted by
$\pt \Der^k A$ and $\ap \Der^k A$, respectively, from \cite[3.3,
3.12]{MR3936235} and \cite[3.8, 3.20]{MR3978264}.  The terms \emph{immersion}
and \emph{embedding} are employed in accordance with \cite[p.\,21]{MR1336822}.
Whenever $k$ is a positive integer or $k = \infty$, we mean by a
[\emph{sub}]\emph{manifold-with-boundary of class~$k$} a Hausdorff topological
space with a countable base of its topology that is, in the terminology
of~\cite[pp.\,29--30]{MR1336822}, a $C^k$~[sub]manifold.  For
manifolds-with-boundary~$M$ of class~$k$, we similarly adapt the notion of
\emph{chart of class~$k$} and \emph{Riemannian metric of class~$k-1$}
from~\cite[p.\,29, p.\,95]{MR1336822} and denote by $\partial M$ its
\emph{boundary} as in~\cite[p.\,30]{MR1336822}.  The image, $f_\# V$, of a
varifold $V$ under suitable locally Lipschitzian maps $f$ is used in
accordance with \cite[2.28]{DOI_10.4171_RMI_1487}.  Finally, whenever $G$ is a
complete normed commutative group as defined in
\cite[2.1]{DOI_10.4171_RMI_1487}, we employ the following
notation regarding the group $\mathscr R_m^{\textup{loc}} ( U, G )$ of $m$
dimensional \emph{locally rectifiable $G$ chains} $S$ in an open subset $U$ of
$\mathbf R^n$, see \cite[3.5]{DOI_10.4171_RMI_1487}: the notions of weight
measure, $\| S \|$, and restriction, $S \restrict A$, see
\cite[3.5]{DOI_10.4171_RMI_1487}; the slice, $\langle S,f, \cdot \rangle$, of
$S$ by $f$, see \cite[3.8]{DOI_10.4171_RMI_1487}; the isomorphism $\iota_{U,m}
: \mathscr R_m^{\textup{loc}} (U) \to \mathscr R_m^{\textup{loc}} ( U, \mathbf
Z )$, see \cite[4.1]{DOI_10.4171_RMI_1487}; the bilinear operation $\cdot :
\mathscr R_m^{\textup{loc}} ( U, \mathbf Z ) \times G \to \mathscr
R_m^{\textup{loc}} ( U, G )$, see \cite[4.5]{DOI_10.4171_RMI_1487}; and, the
chain complex of \emph{integral $G$ chains} in $U$ with $m$-th chain group
$\mathbf I_m ( U, G)$---identified with a subgroup of $\mathscr
R_m^{\textup{loc}} ( U, G)$---and boundary operator~$\boundary_G$, see
\cite[4.11, 4.13, 4.17]{DOI_10.4171_RMI_1487}.

\paragraph{Definitions in the text} The \emph{$Y$ topology} on the dual space
$Y^\ast$ of a Banach space $Y$ is specified in \ref{def:Y-topology}.  The
\emph{flow} associated with a vector field occurs in \ref{example:flow}. For
varifolds, the function $\boldsymbol \upeta (V, \cdot)$ representing the first
variation $\updelta V$ with respect to $\| \updelta V \|$ is defined in
\ref{def:upeta}; the concepts of \emph{partition}, \emph{local finiteness} of
families thereof, \emph{indecomposability of type $\Psi$}, and \emph{partition
along $f$} are introduced in \ref{def:partition},
\ref{def:locally-finiteness-decompositions},
\ref{definition:indecomposability}, and \ref{definition:partition-along},
respectively.  For integral $G$ chains, the definition of
\emph{indecomposability} is recorded in
\ref{definition:indecomposable_chains}.  Based on \cite{MR3936235} and
\cite{MR3978264}, the notions of \emph{pointwise} and \emph{approximate mean
curvature vector} $\pt \mathbf h (A, \cdot )$ and $\ap \mathbf h (A, \cdot)$
of subsets $A$ of Euclidean space are introduced in
\ref{def:mean-curvature-of-sets}.  For immersions $F$, we lay down the
concepts of \emph{tangent cone} $\Tan (F,\cdot)$, \emph{normal cone} $\Nor
(F,\cdot)$, \emph{mean curvature vector} $\mathbf h (F, \cdot )$, and
\emph{exterior normal} $\mathbf n (F,\cdot)$ in \ref{def:mean_curvature},
\emph{associated varifold} in \ref{def:immersed_varifold}, and
\emph{Riemannian distance} in \ref{def:Riemannian_distance}.

\section{Preliminaries}

We compile properties of distributions representable by integration in
\ref{def:Y-topology}--\ref{example:distribution-submanifolds}, approximate a
retraction in \ref{lemma:almost-retraction-to-ball}, study the reach of the
graph of functions and submanifolds of class $2$ in
\ref{miniremark:p-q}--\ref{remark:bounded-domains} and
\ref{thm:reach-non-proper-submanifolds}--\ref{remark:flow-contained},
respectively, consider the flow associated to certain vector fields in
\ref{example:flow}, and discuss the intersection of sets of locally finite
perimeter in \ref{example:perimeter}.  Proceeding to varifolds, we treat
notions related to the first variation in
\ref{def:upeta}--\ref{remark:polar_decomposition} and observe some new basic
properties of the distributional boundary of superlevel sets of generalised
weakly differentiable functions in
\ref{thm:area-formular-representation-slice}%
--\ref{remark:area-formular-representation-slice}.

\begin{definition} [see \protect{\cite[p.\,420]{MR0117523}}]
	\label{def:Y-topology}
	Suppose $Y$ is a Banach space.
	
	Then, the topology on $Y^\ast$ inherited from $\mathbf R^Y$ is termed
	the \emph{$Y$ topology}.
\end{definition}

\begin{theorem} \label{thm:split-distribution-rep-by-int}
	Suppose $U$ is an open subset of $\mathbf R^n$, $Y$ is a separable
	Banach space, $\phi$ is a Radon measure over $U$, $T \in \mathscr D'
	(U,Y)$ is representable by integration, $g$ is a $Y^\ast$ valued
	function that is $\| T \|$ measurable with respect to the $Y$
	topology,
	\begin{gather*}
		\| g(x) \| = 1 \quad \text{for $\| T \|$ almost all $x$}, \\
		T ( \theta ) = {\textstyle\int \langle \theta, g \rangle \ud
		\| T \|} \quad \text{for $\theta \in \mathscr D (U,Y)$},
	\end{gather*}
	the function $h$ mapping a subset of $U$ into
	$\mathbf R^Y$ is defined by
	\begin{gather*}
		\langle y,h(a) \rangle = (V) \lim_{S \to a} T ( b_S \cdot y )
		\big / \phi (S) \in \mathbf R \ \text{for $y \in Y$}, \quad
		\text{whenever $a \in U$}, \\
		\text{where $V = \{ (a,\mathbf B (a,r)) \with a \in U, 0 < r <
		\infty, \mathbf B(a,r) \subset U \}$}
	\end{gather*}
	and $b_S$ is the characteristic function of $S$ on $U$.
	
	Then, $\im h \subset Y^\ast$, $h$ is a Borel function
	with respect to the $Y$ topology on $Y^\ast$, $\dmn h$ is a Borel set,
	and, for $\phi$ almost all $a$, there holds, for $y
	\in Y$,
	\begin{alignat*}{2}
		\langle y,h(a) \rangle & = \mathbf D ( \| T \|, \phi, V, a )
		\langle y, g(a) \rangle && \quad \text{if $a \in \dmn g$} \\
		\langle y,h(a) \rangle & = 0 && \quad \text{if $a \notin \dmn
		g$};
	\end{alignat*}
	in particular, $\| h(a) \| = \mathbf D ( \| T \|, \phi, V, a )$ for
	$\phi$ almost all $a$, $\| T \|_\phi = \phi \restrict \| h \|$, and
	\begin{equation*}
		T ( \theta ) = {\textstyle\int} \langle \theta, h \rangle \ud
		\phi + {\textstyle\int} \langle \theta, g \rangle \ud ( \| T
		\| - \| T \|_\phi ) \quad \text{for $\theta \in \mathbf L_1 (
		\| T \|, Y )$}.
	\end{equation*}
\end{theorem}

\begin{proof}
	The first three conclusions may be verified by
	employing the uniform boundedness principle (see
	\cite[II.1.11]{MR0117523}) and a countable dense subset of $Y$. In
	view of \cite[2.8.18]{MR41:1976}, $V$ is a Vitali relation with
	respect to $\phi$ and $\| T \|$.  If either $\mathbf
	D ( \| T \|, \phi, V, a) = 0$ or $0 < \mathbf D ( \| T \|, \phi, V , a
	) < \infty$ and, whenever $y \in Y$, the point $a$ belongs to the $(\|
	T \|, V )$ Lebesgue set of the function mapping $x \in \dmn g$ onto
	$\langle y, g(x) \rangle$, then the fourth
	conclusion holds at $a$.  Employing a countable dense subset of $Y$
	and \cite[2.9.5, 2.9.7, 2.9.9]{MR41:1976}, we see that $\phi$ almost
	all $a$ satisfy these conditions.  Since $\| T \|_\phi = \phi
	\restrict \mathbf D ( \| T \|, \phi, V, \cdot )$ by
	\cite[2.9.7]{MR41:1976} and \cite[2.7]{DOI_10.4171_RMI_1487} and $T (
	\theta ) = \int \langle \theta, g \rangle \ud \| T \|$ for $\theta \in
	\mathbf L_1 ( \| T \|, Y )$, the postscript follows from \cite[2.4.10,
	2.9.2, 2.9.6]{MR41:1976}.
\end{proof}

\begin{example} \label{example:distributions_representable_by_integration}
	Suppose $U$ is an open subset of $\mathbf R^n$, $Y$ is a separable
	Banach space, $\phi$ is a Radon measure over $U$, $k$ is a $Y^\ast$
	valued function that is $\phi$ measurable with respect to the $Y$
	topology, $\| k \| \in \mathbf L_1^{\textup{loc}} ( \phi )$, and $T
	\in \mathscr D' ( U, Y )$ satisfies
	\begin{equation*}
		T ( \theta ) = {\textstyle\int} \langle \theta, k \rangle \ud
		\phi \quad \text{for $\theta \in \mathscr D ( U, Y )$}.
	\end{equation*}
	Then, $\| T \|$ is a Radon measure absolutely continuous with respect
	to $\phi$ and
	\begin{equation*}
		\text{$\langle y, k(a) \rangle = (V) \lim_{S \to a} T ( b_S
		\cdot y) \big / \phi (S)$ for $y \in Y$}, \quad \text{for
		$\phi$ almost all $a$},
	\end{equation*}
	by \cite[2.8.18, 2.9.9]{MR41:1976}, where the notation of
	\ref{thm:split-distribution-rep-by-int} is employed; hence,
	\cite[2.9.2, 4.1.5]{MR41:1976} and
	\ref{thm:split-distribution-rep-by-int} yield $\| T \| = \phi
	\restrict \|k\|$ and
	\begin{equation*}
		T ( \theta ) = {\textstyle\int} \langle \theta, k \rangle \ud
		\phi \quad \text{for $\theta \in \mathbf L_1 ( \| T \|, Y )$}.
	\end{equation*}
	In particular, whenever $S \in \mathscr D' ( U, Y)$ is representable
	by integration and $A$ is $\| S \|$ measurable, we have $\| S
	\restrict A \| = \| S \| \restrict A$; in fact, we may take $T = S
	\restrict A$.
\end{example}

\begin{theorem} \label{thm:distributions_Lp_representable}
	Suppose $1 \leq p < \infty$, $q = \infty$ if $p = 1$ and $q = p/(p-1)$
	if $p>1$, $U$ is an open subset of $\mathbf R^n$, $\phi$ is a Radon
	measure over $U$, $Y$ is a separable Banach space, $T \in \mathscr D'
	( U, Y )$, and
	\begin{equation*}
		M = \sup \{ T ( \theta ) \with \theta \in \mathscr D ( U,Y ),
		\phi_{(p)} ( \theta ) \leq 1 \} < \infty.
	\end{equation*}
	
	Then, there exists a $\phi$ almost unique function $k : U \to Y^\ast$
	that is $\phi$ measurable with respect to the $Y$ topology such that
	$\|k\| \in \mathbf L_1^{\textup{loc}} ( \phi )$ and
	\begin{equation*}
		T ( \theta ) = {\textstyle \int} \langle \theta, k \rangle \ud
		\phi \quad \text{for $\theta \in \mathscr D (U,Y)$}.
	\end{equation*}
	Moreover, we have $M = \phi_{(q)} (\|k\|)$.
\end{theorem}

\begin{proof}
	In view of \cite[2.1]{MR3626845}, we may define the linear map
	\begin{equation*}
		S : \mathbf L_p ( \phi, Y ) \cap \{ \theta \with \dmn \theta =
		U \} \to \mathbf R
	\end{equation*}
	to be the unique $\phi_{(p)}$ continuous extension of $T$; hence,
	\begin{equation*}
		M = \sup \{ S ( \theta ) \with \theta \in \mathbf L_p ( \phi,
		Y ), \dmn \theta = U, \phi_{(p)} ( \theta ) \leq 1 \}.
	\end{equation*}
	Using \cite[Chapter 7, Sections 4 and 5]{MR0276438}, we obtain a
	function $k : U \to Y^\ast$ that is $\phi$ measurable with respect to
	the $Y$ topology such that
	\begin{equation*}
		S ( \theta ) = {\textstyle \int} \langle \theta, k \rangle \ud
		\phi \quad \text{whenever $\theta \in \mathbf L_p ( \phi, Y )$
		and $\dmn \theta = U$,}
	\end{equation*}
	and $M = \phi_{(q)}(\|k\|)$.%
	\begin{footnote}%
		{More elementary, one may pass from the case $Y = \mathbf R$
		treated in \cite[2.5.7\,(i)\,(ii)]{MR41:1976} to the general
		case adapting the method of \cite[2.5.12]{MR41:1976}; this is
		carried out in
		\cite[4.6.3\,(1)\,(2)]{snulmenn:diploma_thesis}.}
	\end{footnote}%
	Finally, the uniqueness statement follows from
	\ref{example:distributions_representable_by_integration}.
\end{proof}

\begin{remark} \label{remark:distributions_Lp_representable}
	We recall that, for general $Y$, the functions $g$, $h$, and $k$
	occurring in \ref{thm:split-distribution-rep-by-int}%
	--\ref{thm:distributions_Lp_representable} may be $\phi$ nonmeasurable
	with respect to the norm topology on $Y^\ast$, see
	\cite[2.32]{MR4241804}.  In the present paper, we only consider the
	case $Y = \mathbf R^n$.
\end{remark}

\begin{example} \label{example:distribution-submanifolds}
	Suppose $1 \leq p < \infty$, $q = \infty$ if $p = 1$ and $q = p/(p-1)$
	if $p > 1$, $k$ is positive integer or $k = \infty$, $U$ is an open
	subset of $\mathbf R^n$, $\phi$ is a Radon measure over $U$, $B$ is a
	relatively closed subset and a submanifold of class $k$ of $U$, the
	vector space $\Theta$ consists of all vector fields $\theta : U \to
	\mathbf R^n$ of class $k-1$ with compact support satisfying $\theta
	(b) \in \Tan (B,b)$ for $b \in B$, and
	\begin{equation*}
		X_r = \mathbf L_r ( \phi, \mathbf R^n ) \cap \{ \theta \with
		\text{$\dmn \theta = U$, $\theta (b) \in \Tan(B,b)$ for $b \in
		B$} \}
	\end{equation*}
	for $1 \leq r \leq \infty$; hence, $\Theta$ is $\phi_{(p)}$
	dense in $X_p$ as may be verified combining \cite[2.2.17, 2.4.18\,(1),
	3.1.19\,(1)]{MR41:1976} and \cite[2.1]{MR3626845}.  Next, suppose $R
	\in \Hom ( \Theta, \mathbf R )$ and
	\begin{equation*}
		M = \sup \{ R ( \theta ) \with \theta \in \Theta, \phi_{(p)} (
		\theta ) \leq 1 \} < \infty.
	\end{equation*}
	Let $S \in \Hom ( X_p, \mathbf R )$ be the unique $\phi_{(p)}$
	continuous extension of $R$.  Defining the endomorphism $E$ of
	$\mathbf L_p ( \phi, \mathbf R^n ) \cap \{ \theta \with \dmn \theta =
	U \}$ with $E \circ E = E$ by
	\begin{equation*}
		\text{$E(\theta)(x) = \theta(x)$ if $x \in U \without B$},
		\quad \text{$E(\theta)(x) = \langle \theta (x),
		\Tan(B,x)_\natural \rangle$ if $x \in B$}
	\end{equation*}
	for $\theta \in \mathbf L_p ( \phi, \mathbf R^n )$ with $\dmn \theta =
	U$ and noting $\phi_{(p)} ( E(\theta)) \leq \phi_{(p)} ( \theta )$ for
	such $\theta$, we apply \ref{thm:distributions_Lp_representable}
	with $T = S \circ E | \mathscr D ( U, \mathbf R^n )$ to furnish $k \in
	X_q$ satisfying
	\begin{equation*}
		T ( \theta ) = {\textstyle\int k \bullet \theta \ud \phi}
		\quad \text{for $\theta \in \mathbf L_p ( \phi, \mathbf R^n
		)$}
	\end{equation*}
	and $\phi_{(q)} (k) = M$.  Finally, we deduce from
	\ref{example:distributions_representable_by_integration} that, for
	$\phi$ almost all $x$, there holds
	\begin{equation*}
		k(x) \bullet \theta(x) = \lim_{r \to 0+} \frac{S ( b_{x,r}
		\cdot \theta )}{\phi \, \mathbf B (x,r)} \quad \text{whenever
		$\theta \in \Theta$},
	\end{equation*}
	where $b_{x,r}$ is the characteristic function of $\mathbf B(x,r)$ on
	$U$ since $\theta$ is continuous and $\lim_{r \to 0+} \int_{\mathbf
	B(x,r)} |k| \ud \phi \big / \phi \, \mathbf B (x,r) < \infty$ for
	$\phi$ almost all $x$ by \cite[2.8.18, 2.9.5]{MR41:1976}.
\end{example}

\begin{lemma} \label{lemma:almost-retraction-to-ball}
	Suppose $n$ is a positive integer and $\epsilon > 0$.

	Then, there exists a function $\rho : \mathbf R^n \to \mathbf R^n$ of
	class $\infty$ such that $\Lip \rho = 1$, $\im \rho \subset \mathbf
	B (0,1+\epsilon)$,
	\begin{equation*}
		\rho (b) = b \quad \text{for $b \in \mathbf R^n \cap \mathbf B
		(0,1)$},
	\end{equation*}
	and $\rho (b) \in \{ t b \with 0 \leq t \leq 1 \}$ whenever $b \in
	\mathbf R^n$.
\end{lemma}

\begin{proof}
	Choosing a concave nondecreasing function $f : \mathbf R \to \mathbf
	R$ of class $\infty$ such that $f(t) = t$ for $- \infty < t \leq 1$
	and $f(t) \leq 1 + \epsilon$ for $t \in \mathbf R$, we define $\rho :
	\mathbf R^n \to \mathbf R^n$ of class $\infty$ by requiring $\rho(y)=
	(f(|y|)/|y|) y$ for $y \in \mathbf R^n \without \{ 0 \}$.  Since $\Der
	\rho (0) = \mathbf 1_{\mathbf R^n}$ and
	\begin{equation*}
		\| \Der \rho (y) \| = \sup \{ f'(|y|), f(|y|)/|y| \} \leq 1
		\quad \text{for $y \in \mathbf R^n \without \{ 0 \}$},
	\end{equation*}
	we conclude $\Lip \rho \leq 1$ from \cite[2.2.7, 3.1.1]{MR41:1976}.
\end{proof}

\begin{miniremark} \label{miniremark:p-q}
	Suppose $m$ and $n$ are positive integers and $m < n$.  Following
	\cite[5.1.9]{MR41:1976}, we define orthogonal projections $\mathbf p :
	\mathbf R^n \to \mathbf R^m$ and $\mathbf q : \mathbf R^n \to \mathbf
	R^{n-m}$ by
	\begin{equation*}
		\mathbf p (z) = (z_1, \ldots, z_m ), \quad \mathbf q (z) =
		(z_{m+1}, \ldots, z_n)
	\end{equation*}
	whenever $z = (z_1, \ldots, z_n) \in \mathbf R^n$.
\end{miniremark}

\begin{theorem} \label{thm:reach-real-valued-graphs}
	In the situation of \ref{miniremark:p-q}, suppose $n-m=1$, $f :
	\mathbf R^{n-1} \to \mathbf R$ is of class $2$, $F = \mathbf p^\ast +
	\mathbf q^\ast \circ f$, $C = \im F$, $0 < r \leq \infty$, and
	\begin{equation*}
		\| \mathbf b ( C,c ) \| \leq r^{-1} \quad \text{for $c \in
		C$}.
	\end{equation*}
	
	Then, whenever $\dist (z,C) < r$, there exists a unique point $c$ in
	$C$ satisfying $|z-c| = \dist (z,C)$; moreover, we have $\sign \mathbf
	q ( z - c ) = \sign ( \mathbf q (z) - f ( \mathbf p (z)) )$.
\end{theorem}

\begin{proof}
	Abbreviating $A = \{ s \with -r < s < r \}$, we define functions $\nu
	: C \to \mathbf S^{n-1}$, $\phi_s : C \to \mathbf R^n$, and $\psi : A
	\times C \to \mathbf R^n$ of class $1$ by
	\begin{equation*}
		\begin{gathered}
			\nu (c) = ( 1 + | \Der f ( \mathbf p (c)) |^2)^{-1/2}
			\big ( \mathbf q^\ast (1) - \mathbf p^\ast ( \grad f (
			\mathbf p (c ) ) ) \big ), \\
			\phi_s (c) = \psi (s,c) = c + s \, \nu (c)
		\end{gathered}
	\end{equation*}
	whenever $(s,c) \in A \times C$; hence, we have
	\begin{gather*}
		| \phi_s(c)-c | \leq |s|, \quad \text{$\phi_s$ is proper}, \\
		\Der \phi_s (c) = \mathbf 1_{\Tan(C,c)} + s \Der \nu (c),
		\quad \Der \phi_s(c) \in \Hom ( \Tan(C,c), \Tan (C,c) ), \\
		\| \Der \phi_s (c) - \mathbf 1_{\Tan (C,c)} \| \leq |s| \|
		\mathbf b (C,c) \| \leq |s|r^{-1} < 1, \quad \det \Der\phi_s
		(c) > 0.
	\end{gather*}
	Defining functions $h = \mathbf p \circ \psi \circ ( \mathbf 1_A
	\times F )$ and $g_s = \mathbf p \circ \phi_s \circ F$ of class $1$,
	we see that, whenever $(s,x) \in A \times \mathbf R^{n-1}$, we have
	\begin{gather*}
		\mathbf p ( F (x) ) = x, \quad g_s(x) = x + s \, \mathbf p (
		\nu ( F(x))), \quad | g_s (x)-x| \leq |s|, \\
		\mathbf p \circ \Der F (x) = \mathbf 1_{\mathbf R^{n-1}},
		\quad \Der g_s(x) = \mathbf p \circ \Der \phi_s(F(x)) \circ
		\Der F(x), \quad \det \Der g_s (x) > 0.
	\end{gather*}
	Thus, $h | (I \times \mathbf R^{n-1})$ is proper whenever $I$ is a
	compact subinterval of $A$ and, for $s \in A$, $\Number (g_s,\cdot)$
	is a constant $\mathbf Z$ valued function because $g_s$ is a proper
	map which locally is a diffeomorphism of class $1$ by
	\cite[3.1.1]{MR41:1976}.  In view of \cite[4.1]{DOI_10.4171_RMI_1487},
	we conclude
	\begin{equation*}
		g_{s\,\#} \mathbf E^{n-1} = g_{0\,\#} \mathbf E^{n-1} \quad
		\text{for $s \in A$}
	\end{equation*}
	from \cite[4.16]{DOI_10.4171_RMI_1487} applied with $m$, $n$, $U$,
	$G$, $\nu$, $V$, $t$, and $S$ replaced by $n-1$, $n-1$, $\mathbf
	R^{n-1}$, $\mathbf Z$, $n-1$, $\mathbf R^{n-1}$, $s$, and
	$\iota_{\mathbf R^{n-1},n-1} ( \mathbf E^{n-1} )$.  Noting $g_0 =
	\mathbf 1_{\mathbf R^{n-1}}$ and computing $g_{s\,\#} \mathbf E^{n-1}$
	by means of \cite[3.6, 4.1]{DOI_10.4171_RMI_1487}, we obtain that
	\begin{equation*}
		\Number (g_s,\chi) = 1 \quad \text{for $(s,\chi) \in A \times
		\mathbf R^{n-1}$}.
	\end{equation*}
	Therefore, $\phi_s$ is univalent for $s \in A$.  Moreover, we have
	\begin{equation*}
		\sign \mathbf q ( \phi_s (c) - c ) = \sign s \quad \text{for
		$(s,c) \in A \times C$}.
	\end{equation*}
	
	Next, suppose $0 < d = \dist (z,C) < r$ and $\epsilon = \sign (
	\mathbf q (z) - f ( \mathbf p (z)) )$.  Observing
	\begin{equation*}
		\mathbf q (z) \notin \mathbf q \big [ C \cap \mathbf p^{-1} [
		\mathbf B ( \mathbf p (z),d) ] \big ]
	\end{equation*}
	because $\mathbf p [ \Tan (C,c) ] = \mathbf R^{n-1}$ for $c \in C$, we
	conclude
	\begin{equation*}
		\sign ( \mathbf q (z-F(x)) ) = \epsilon \quad \text{for $x \in
		\mathbf B ( \mathbf p (z),d )$}.
	\end{equation*}
	Finally, for $c \in C \cap \mathbf B ( z,d )$, we shall verify $z =
	\phi_{\epsilon d} (c)$ and $\sign \mathbf q (z-c) = \epsilon$; in
	fact, noting $z-c \in \Nor (C,c)$, there exists $s \in A$ such that
	$|s|=d$ and $\phi_s(c)=z$, whence we infer $\sign s = \sign \mathbf q
	( z-c ) = \epsilon$ as $C \cap \mathbf B (z,d) \subset F [ \mathbf B (
	\mathbf p (z), d ) ]$.
\end{proof}

\begin{remark} \label{remark:reach-real-valued-graphs}
	The principal conclusion is equivalent to $\reach (C) \geq r$; for
	this lower bound to hold, the hypothesis $n-m=1$ is essential, see
	\ref{example:spiral} below.
\end{remark}

\begin{example} \label{example:spiral}
	Suppose $m$ and $n$ are positive integers satisfying $n-m \geq 2$,
	$\pi = \mathscr H^1 ( \mathbf S^1 ) / 2$, $0 < \epsilon \leq 1$, $f :
	\mathbf R^m \to \mathbf R^{n-m}$ is defined by
	\begin{equation*}
		f (x) = ( \cos ( \pi x_1/\epsilon ), \sin ( \pi x_1/\epsilon),
		0, \ldots, 0) \in \mathbf R^{n-m} \quad \text{for $x=(x_1,
		\ldots, x_m) \in \mathbf R^m$},
	\end{equation*}
	and $C$ is associated with $f$ as in
	\ref{thm:reach-real-valued-graphs}.  Then, there holds
	\begin{equation*}
		\text{$\| \mathbf b(C,c) \| \leq 1$ for $c \in C$} \quad
		\text{but} \quad \reach (C) \leq \epsilon;
	\end{equation*}
	in fact, reducing to $(m,n)=(1,3)$ and defining $F$ as in
	\ref{thm:reach-real-valued-graphs}, we compute
	\begin{equation*}
		|F'(x)|^2=1+\epsilon^{-2}\pi^2, \quad |F''(x)| =
		\epsilon^{-2}\pi^2, \quad \| \mathbf b (C,F(x)) \| =
		\frac{\epsilon^{-2} \pi^2}{1+\epsilon^{-2}\pi^2} \leq 1
	\end{equation*}
	for $x \in \mathbf R$, take $z=(0,-1,0)$, $d = \dist (z,C)$, and note
	\begin{equation*}
		|z-F(x)|=|z-F(-x)| \quad \text{for $x \in \mathbf R$},
	\end{equation*}
	so that $|z-F(0)|=2>\epsilon = |z-F(\epsilon)| \geq d$ implies $\reach
	(C) \leq \epsilon$.
\end{example}

\begin{corollary} \label{corollary:reach-real-valued-graphs}
	Suppose additionally $U = \{ z \with \dist ( z,C ) < r \}$ and the
	functions $\xi : U \to C$ and $\delta : U \to \mathbf R$ satisfy
	\begin{equation*}
		\{ \xi (z) \} = C \cap \{ c \with |z-c| = \dist(z,C) \}, \quad
		\delta(z) = \dist (z,C) \sign \mathbf q ( z-\xi(z))
	\end{equation*}
	for $z \in U$, $f$ is of class $k \geq 2$, and $E = \{ z \with f(
	\mathbf p (z)) > \mathbf q (z) \}$.
	
	Then, $\xi$ is of class $k-1$, $\delta$ is of class $k$, and
	\begin{equation*}
		\grad \delta (z) = \mathbf n ( E,\xi(z)) \quad \text{for $z
		\in U$}.
	\end{equation*}
\end{corollary}

\begin{proof}
	Defining $A$, $\nu$, and $\psi$ as in the proof of
	\ref{thm:reach-real-valued-graphs}, we note that $\im \psi = U$ and
	that $\Der \psi (s,c)$ is univalent for $(s,c) \in A \times C$ because
	$\langle (1,0), \Der \psi (s,c) \rangle = \nu(c)$.  From
	\cite[4.8\,(13)]{MR0110078} and \ref{thm:reach-real-valued-graphs}, we
	obtain
	\begin{equation*}
		\xi ( \psi (s,c) ) = c, \quad \delta ( \psi (s,c)) = s
	\end{equation*}
	for $(s,c) \in A \times C$.  Thus, $\psi$ is a diffeomorphism of class
	$k-1$ of $A \times C$ onto $U$ by \cite[3.1.1, 3.1.11,
	3.1.19]{MR41:1976}; in particular, $\xi$ and $\delta$ are of class
	$k-1$.  Finally, we deduce
	\begin{equation*}
		\grad \delta (z) = \frac{z-\xi(z)}{\delta(z)} = \nu ( \xi (z))
		\quad \text{for $z \in U \without C$}
	\end{equation*}
	from \cite[4.8\,(3)\,(5)]{MR0110078} and the fact that
	$\frac{\psi(s,c)-c}s = \nu (c)$ whenever $(s,c) \in A \times C$ and $s
	\neq 0$; hence, $\grad \delta = \nu \circ \xi$ and $\delta$ is of
	class $k$.
\end{proof}

\begin{remark} \label{remark:bounded-domains}
	In the context of \emph{bounded domains}, a similar deduction of
	properties of the distance function is carried out in \cite[Lemma
	14.6]{MR1814364}.
\end{remark}

\begin{theorem} \label{thm:reach-non-proper-submanifolds}
	Suppose $B$ is a submanifold of class $2$ of $\mathbf R^n$, $A = \Clos
	B$,
	\begin{gather*}
		\xi = ( \mathbf R^n \times A ) \cap \{ (x,a) \with \{ a \} = A
		\cap \mathbf B (x,|x-a|) \}, \\
		X = \xi^{-1} [ B ] \cap \{ x \with \reach (A,\xi (x)) > \dist
		(x,A) \},
	\end{gather*}
	and $\psi : X \to \mathbf R^n \times \mathbf R^n$ satisfies $\psi (x)
	= ( \xi(x),x-\xi(x))$ for $x \in X$.

	Then, $X$ is an open set containing $B$, $\psi$ is a univalent
	function of class $1$ whose image is contained in $\Nor (B)$, and
	$\Der \psi (x)$ is univalent for $x \in X$.
\end{theorem}

\begin{proof}
	Assume $A \neq \varnothing$.  For $b \in B$, there holds $\reach (A,b)
	> 0$; in fact, \cite[4.12]{MR0110078} and \cite[3.1.19]{MR41:1976}
	yield a neighbourhood $U$ of $b$ and a set $C$ satisfying $\reach
	(C,b)>0$ and $U \cap A = U \cap C$.  It follows that $B \subset X$.
	From \cite[4.2, 4.8\,(4)]{MR0110078}, we recall that $\reach
	(A,\cdot)$ and $\xi$ are continuous functions; in particular,
	observing that $X \subset \Int \Unp ( A )$, we infer that $X$ is open,
	as $B$ is open relative to $A$.  Moreover, combining the continuity
	with \cite[4.8\,(13)]{MR0110078}, we conclude that $\psi$ is locally
	Lipschitzian with $\im \psi \subset \Nor (B)$, that $\psi$ is
	univalent, and that $\psi^{-1} (b,v) = b + v$ for $(b,v) \in \im
	\psi$.  Recalling \cite[3.1.19]{MR41:1976} and that
	$\Nor (B)$ is an $n$ dimensional submanifold of class $1$, we deduce
	firstly that $\Der \psi (x)$ is univalent for $x \in \dmn \Der \psi$
	by \cite[3.1.1]{MR41:1976}, secondly, that $\im \psi$ is open relative
	to $\Nor (B)$ by \cite[4.1.26]{MR41:1976}, and finally that $\psi$ is
	of class $1$ by \cite[3.1.1]{MR41:1976}.
\end{proof}

\begin{remark}
	Instead of applying degree theory for Lipschitzian maps via
	\cite[4.1.26]{MR41:1976}, the general principle of \emph{invariance of
	domain} could have been employed.
\end{remark}

\begin{remark} \label{remark:nearest-point-projection}
	From \cite[4.8\,(4)\,(12)]{MR0110078}, we deduce that $B = X \cap \{ x
	\with x = \xi (x) \}$ is closed relative to $X$ and that we have $\xi
	(b+v) = b$ whenever $(b,v) \in \Nor (B)$ and $|v| < \reach (A,b)$.
	Defining $Q : X \to \Hom ( \mathbf R^n, \mathbf R^n )$ by
	\begin{align*}
		v \bullet \langle u, Q(x) \rangle & = v \bullet \langle \Tan
		(B,\xi(x))_\natural (u), \Der \Nor ( B, \cdot)_\natural ( \xi
		(x)) \rangle (x-\xi(x)) \\
		& = - \mathbf b ( B, \xi(x) ) ( \langle u, \Tan ( B, \xi
		(x))_\natural \rangle, \langle v, \Tan (B,\xi(x))_\natural
		\rangle ) \bullet ( x- \xi(x))
	\end{align*}
	whenever $x \in X$ and $u,v \in \mathbf R^n$ with the help of
	\cite[3.1\,(2)]{MR3625810}, we thus obtain
	\begin{equation*}
		\| \mathbf b ( B, \xi(x)) \| \leq \reach (A,\xi(x))^{-1},
		\quad \Der \xi(x) = \Tan (B,\xi(x))_\natural +
		\sum_{i=1}^\infty (-1)^i Q(x)^i
	\end{equation*}
	whenever $x \in X$ by generalising \cite[2.2\,(3)]{MR0397520} to
	submanifolds of class $2$ and applying the result with $B$ replaced by
	$B \cap \{ b \with \reach (A,b)>r \}$ and $\kappa = r^{-1}$ for $0 < r
	< \infty$.  Abbreviating $\delta = \dist ( \cdot, B)$, we have $\grad
	\delta (x) = \delta(x)^{-1}(x-\xi(x))$ for $x \in X \without B$ by
	\cite[4.8\,(3)]{MR0110078}, $\delta | (X \without B)$ of class $2$,
	and, in case $\dim B = n-1$,
	\begin{equation*}
		\delta(x) \| \Der^2 \delta (x) \| \leq \frac{\| Q (x) \|}{1-\|
		Q(x)\|} \quad \text{for $x \in X \without B$}.
	\end{equation*}
\end{remark}

\begin{remark} \label{remark:flow-contained}
	If $\theta : \mathbf R^n \to \mathbf R^n$ is of class $1$, $\theta (b)
	\in \Tan (B,b)$ for $b \in B$, $I \subset \mathbf R$, $f : I \to
	\mathbf R^n$ is of class $1$, and $f' = \theta \circ f$, then $f^{-1}
	[B]$ is open; in fact, whenever $U$ is an open subset of $\mathbf
	R^n$, $\kappa = \Lip ( \theta | U ) < \infty$, and $B \neq
	\varnothing$, we recall \cite[2.10.19\,(4), 3.1.5]{MR41:1976} and
	\ref{remark:nearest-point-projection}, abbreviate $J = f^{-1} \big [ U
	\cap X \cap \xi^{-1} [ U ] \big ]$, let $g = \delta \circ f | J$,
	define $h_\pm : J \to \mathbf R \cap \{ y \with y \geq 0 \}$ by
	$h_\pm(t) = g(t) \exp ( \pm \kappa t )$ for $t \in J$, and estimate
	\begin{equation*}
		|g'(t)| \leq \kappa g (t), \quad h_+'(t) \geq 0, \quad h_-'(t)
		\leq 0 \quad \text{for $\mathscr L^1$ almost all $t \in J$},
	\end{equation*}
	because we have $( \delta \circ f )'(t) = \delta (f(t))^{-1} (
	f(t)-\xi(f(t))) \bullet ( \theta (f(t))- \theta ( \xi (f(t)))$ for $t
	\in f^{-1} [ U \cap X \without B ]$, whence, using
	\cite[2.9.20]{MR41:1976}, we infer $\{ t \with g(t) = 0 \}$ is open.
	Thus, if $I$ is an interval, $B \cap \spt \theta$ is closed, and $B
	\cap \im f \neq \varnothing$, then $\im f \subset B$ as either $B \cap
	\im f \subset \spt \theta$ or $f$ is constant by \cite[Chapter 1,
	Theorem 5.2]{MR0273082}.
\end{remark}

\begin{example} \label{example:flow}
	Suppose $U$ is an open subset of $\mathbf R^n$, $\theta : U \to
	\mathbf R^n$ is of class $1$, and $\spt \theta$ is compact.  Then, by
	\cite[Chapter 1, Theorems 3.4, 5.2, 5.7, and
	10.5]{MR0273082}, there exists $\phi : \mathbf R \times U \to U$ of
	class $1$ with $\phi (0, \cdot ) = \mathbf 1_U$ such that
	\begin{equation*}
		\phi ( \cdot, x )' (t) = \theta ( \phi (t,x)) \quad \text{for
		$(t,x) \in \mathbf R \times U$}
	\end{equation*}
	and such that, if $I$ is an interval, $0 \in I$, $f : I \to U$ is of
	class $1$, and $f' = \theta \circ f$, then $f = \phi (\cdot,f(0)) |
	I$; hence, $\phi (t,\cdot)^{-1} = \phi ( -t, \cdot )$ and $\spt ( \phi
	(t,\cdot)- \mathbf 1_U) \subset \spt \theta$ for $t \in \mathbf R$.
	The unique such $\phi$ is termed the \emph{flow associated with
	$\theta$}.  By the method of \cite[4.1]{MR0307015} it follows that, for
	$V \in \mathbf V_m ( U)$, the derivative, at $0$, of the function
	mapping $t \in \mathbf R$ onto $\| \phi ( t, \cdot )_\# V \| ( \spt
	\theta )$ equals $\int S_\natural \bullet \Der \theta (x) \ud V \,
	(x,S)$.  Finally, if $M$ is a relatively closed subset of $U$ and an
	$n$ dimensional submanifold-with-boundary of class $2$, $B = \partial
	M$, and $\theta (b) \in \Tan (B,b)$ for $b \in B$, then
	\ref{remark:flow-contained} implies
	\begin{equation*}
		\phi (t, \cdot) [B] \subset B, \quad \phi (t, \cdot) [M]
		\subset M, \quad \text{for $t \in \mathbf R$}.
	\end{equation*}
\end{example}

\begin{example} \label{example:perimeter}
	Suppose $E$ and $F$ are subsets of $\mathbf R^n$ with locally finite
	perimeter, and
	\begin{equation*}
		\begin{gathered}
			P = \{ x \with \boldsymbol \Uptheta^n ( \mathscr L^n
			\restrict E, x ) = 1 \}, \quad Q = \{ x \with
			\boldsymbol \Uptheta^n ( \mathscr L^n \restrict F, x )
			= 1 \}, \\
			A = \big \{ x \with \mathbf n(E,x) \in \mathbf S^{n-1}
			\big \}, \quad B = \big \{ x \with \mathbf n(F,x) \in
			\mathbf S^{n-1} \big \}.
		\end{gathered}
	\end{equation*}
	Then, $E \cap F$ is a set with locally finite perimeter by
	\cite[4.5.9\,(13)]{MR41:1976} and, recalling \cite[2.10.19\,(4),
	3.2.19, 4.1.28\,(1)\,(4)]{MR41:1976}, one readily verifies that
	\begin{gather*}
		\text{$\mathbf n(E\cap F,x) = \mathbf n(E,x)$ if $x \in A \cap
		Q$}, \quad \text{$\mathbf n(E\cap F,x) = \mathbf n(F,x)$ if $x
		\in B \cap P$}, \\
		\text{$\mathbf n (E \cap F, x) = {\textstyle \frac{\mathbf n
		(E,x) + \mathbf n(F,x)} 2}$ if $x \in A \cap B$}, \quad \quad
		\text{$\mathbf n (E \cap F, x) = 0$ else}
	\end{gather*}
	for $\mathscr H^{n-1}$ almost all $x \in \mathbf R^n$ by means of
	\cite[4.5.5, 4.5.6\,(1)\,(2)\,(3)]{MR41:1976}; a related fact is
	recorded in \cite[Proposition 1]{MR1812124}.
\end{example}

\begin{definition} \label{def:upeta}
	Suppose $m$ and $n$ are positive integers, $m \leq n$, $U$ is an open
	subset of $\mathbf R^n$, $V \in \mathbf V_m ( U )$, and $\| \updelta V
	\|$ is a Radon measure.
	
	Then, the $\mathbf R^n$ valued function $\boldsymbol \upeta (V,\cdot)$
	is defined on a subset of $U$, by the requirement that, for $x \in U$,
	\begin{equation*}
		\boldsymbol \upeta (V,x) \bullet u = \lim_{r \to 0+} \frac{(
		\updelta V ) ( b_{x,r} \cdot u )}{\| \updelta V \| \, \mathbf
		B (x,r)} \quad \text{whenever $u \in \mathbf R^n$},
	\end{equation*}
	where $b_{x,r}$ denotes the characteristic function of $\mathbf
	B(x,r)$ on $U$; hence, $x \in U$ belongs to the domain of $\boldsymbol
	\upeta (V,\cdot)$ if and only if the above limit exists for $u \in
	\mathbf R^n$.
\end{definition}

\begin{remark} \label{remark:polar_decomposition}
	Our definitions of $\mathbf h(V,\cdot)$, see
	\cite[p.\,992]{MR3528825}, and $\boldsymbol \upeta (V,\cdot)$ adapt
	\cite[4.3]{MR0307015} in the spirit of \cite[4.1.5]{MR41:1976}; in
	particular, \ref{thm:split-distribution-rep-by-int} and
	\ref{example:distributions_representable_by_integration} yield
	\begin{equation*}
		\mathbf h ( V, \cdot ) \in \mathbf L_1^{\textup{loc}} ( \| V
		\|, \mathbf R^n ), \quad \| \updelta V \|_{\| V \|} = \| V \|
		\restrict | \mathbf h (V,\cdot) |,
	\end{equation*}
	$\mathbf h (V,\cdot)$ and $\boldsymbol \upeta ( V,
	\cdot )$ are Borel functions, $\dmn \mathbf h (V,\cdot)$ and $\dmn
	\boldsymbol \upeta (V,\cdot)$ are Borel sets, $|\boldsymbol \upeta
	(V,x) | = 1$ for $\| \updelta V \|$ almost all $x$, 
	\begin{alignat*}{2}
		\mathbf h(V,x) & = | \mathbf h(V,x)| \, \boldsymbol \upeta (V,x)
		&& \quad \text{if $x \in \dmn \boldsymbol \upeta ( V, \cdot
		)$} \\
		\mathbf h (V,x) & = 0 && \quad \text{if $x \notin \dmn
		\boldsymbol \upeta (V,\cdot )$}
	\end{alignat*}
	for $\| V \|$ almost all $x$, and
	\begin{align*}
		( \updelta V ) ( \theta ) & = {\textstyle\int \boldsymbol
		\upeta (V,x) \bullet \theta (x) \ud \| \updelta V \| \, x} \\
		& = - {\textstyle\int} \mathbf h (V,x) \bullet \theta (x) \ud
		\| V \| \, x + {\textstyle\int \boldsymbol \upeta (V,x)
		\bullet \theta (x) \ud ( \| \updelta V \| - \| \updelta V
		\|_{\| V \|}) \, x}
	\end{align*}
	for $\theta \in \mathbf L_1 ( \| \updelta V \|, \mathbf R^n )$.
	Applying \cite[4.1, 4.5, 4.6]{MR3777387} and
	convolution shows
	\begin{equation*}
		\theta \in \mathbf T ( V, \mathbf R^n ) \quad \text{and} \quad
		( \updelta V ) ( \theta ) = {\textstyle\int \trace ( V \weakD
		\theta ) \ud \| V \|}
	\end{equation*}
	whenever $\theta : U \to \mathbf R^n$ is Lipschitzian with compact
	support.
\end{remark}

\begin{theorem} \label{thm:area-formular-representation-slice}
	Suppose $m$ and $n$ are positive integers, $m \leq n$, $U$ is an open
	subset of $\mathbf R^n$, $V \in \mathbf V_m (U)$, $\| \updelta V \|$
	is a Radon measure, $f \in \mathbf T (V)$, and
	\begin{equation*}
		E(y) = \{ x \with f(x) > y \} \quad \text{for $y \in \mathbf
		R$}.
	\end{equation*}

	Then, for $\mathscr L^1$ almost all $y$, $V \boundary E (y)$ is
	representable by integration and
	\begin{equation*}
		V \boundary E (y) ( \theta ) = {\textstyle\int \langle \theta
		(x), | V \weakD f (x) |^{-1} V \weakD f(x) \rangle \ud \| V
		\boundary E(y) \| \, x}.
	\end{equation*}
	for $\theta \in \mathbf L_1 ( \| V \boundary E(y) \|, \mathbf R^n )$.
\end{theorem}

\begin{proof}
	By \cite[8.30]{MR3528825} and \cite[4.11]{MR3777387}, we have
	\begin{equation*}
		{\textstyle\int_{-s}^s \| V \boundary E(y) \| (K) \ud \mathscr
		L^1 \, y < \infty}
	\end{equation*}
	whenever $K$ is a compact subset of $U$ and $0 \leq s < \infty$.  In
	view of \cite[2.2, 2.24]{MR3528825} and
	\ref{example:distributions_representable_by_integration}, it thus
	suffices to prove
	\begin{align*}
		& {\textstyle\int \omega (y) V \boundary E (y) ( \theta ) \ud
		\mathscr L^1 \, y } \\
		& \qquad = {\textstyle\int \omega(y) \int \langle
		\theta (x), | V \weakD f (x) |^{-1} V \weakD f(x) \rangle \ud
		\| V \boundary E(y) \| \, x \ud \mathscr L^1 \, y }
	\end{align*}
	for $\theta \in \mathscr D ( U, \mathbf R^n )$ and $\omega \in
	\mathscr D ( \mathbf R, \mathbf R )$.  The latter equation follows by
	applying \cite[8.5, 8.30]{MR3528825} and \cite[4.11]{MR3777387} with
	$\phi$ and $g$ satisfying $\phi(x,y) = \omega (y)\theta(x)$ and
	$g(x,y) = \omega (y) \langle \theta (x), | V \weakD f(x)|^{-1} V
	\weakD f (x) \rangle$ for $x \in U$ and $y \in \mathbf R$.
\end{proof}

\begin{remark}
	In the special case that $\boldsymbol \Uptheta^m ( \| V \|, x ) \geq
	1$ for $\| V \|$ almost all $x$, the conclusion is immediate from
	\cite[12.2]{MR3528825} and
	\ref{example:distributions_representable_by_integration}.
\end{remark}

\begin{remark} \label{remark:area-formular-representation-slice}
	Under the hypotheses of \ref{thm:area-formular-representation-slice},
	we will show that, if $y \in \mathbf R$, $V \boundary E(y)$ is
	representable by integration, $\theta : U \to \mathbf R^n$ is locally
	Lipschitzian, and
	\begin{equation*}
		K = \spt \theta \cap \spt \| V \| \cap \Clos E(y)
	\end{equation*}
	is compact, then $\theta \in \mathbf T ( V, \mathbf R^n )$ and there
	holds
	\begin{equation*}
		V \boundary E(y) ( \theta ) = {\textstyle\int_{E(y)}
		\boldsymbol \upeta (V,x) \bullet \theta (x) \ud \| \updelta V
		\| \, x}  - {\textstyle\int_{E(y)} \trace (V \weakD \theta
		(x)) \ud \| V \| \, x}.
	\end{equation*}
	By \cite[4.6\,(1)]{MR3777387}, $\theta$ is generalised $V$ weakly
	differentiable.  For the equation, we reduce the problem: firstly, to
	the case that $\spt \theta$ is compact by expressing $\theta$ as sum
	of $\zeta \theta$ and $(1-\zeta) \theta$ for some $\zeta \in \mathscr
	D ( U, \mathbf R )$ such that $K \subset \Int \{ x \with \zeta (x) = 1
	\}$ and noting \cite[8.20\,(1)\,(3)]{MR3528825} and
	\cite[4.11]{MR3777387}; secondly, to the case that $\theta \in
	\mathscr D ( U, \mathbf R^n )$ by approximating $\theta$ using
	convolution and noting \cite[4.6\,(2)]{MR3777387}.  The latter case
	follows from \cite[4.1, 4.5]{MR3777387} and
	\ref{remark:polar_decomposition}.
\end{remark}

\section{Real valued functions of generalised bounded variation}

We prove Theorem \ref{Thm:GBV} of the introductory section in
\ref{thm:equiv_bv} and \ref{corollary:equiv_bv}, coarea formulae in
\ref{remark:equiv_bv}, and relations to the case of Lebesgue measure in
\ref{example:set-of-locally-finite-perimeter} and \ref{remark:GBV}.

\begin{example} \label{example:set-of-locally-finite-perimeter}
	Suppose $V \in \mathbf V_n ( \mathbf R^n )$ is characterised by $\| V
	\| = \mathscr L^n$ and $A$ is an $\mathscr L^n$ measurable set.  Then,
	$\| V \boundary A \|$ is a Radon measure if and only if $A$ is a set
	of locally finite perimeter by \cite[4.5, Remark]{MR0307015} applied
	to $V \restrict A \times \mathbf G(n,n)$.
\end{example}

\begin{theorem} \label{thm:equiv_bv}
	Suppose $m$ and $n$ are positive integers, $m \leq n$, $U$ is an open
	subset of $\mathbf R^n$, $V$ is an $m$ dimensional varifold in $U$,
	$\| \updelta V \|$ is a Radon measure, $f$ is a real valued $\| V \| +
	\| \updelta V \|$ measurable function whose domain is contained in
	$U$, $E = \{ (x,y) \with f(x)>y \}$, $T \in \mathscr D' ( U \times
	\mathbf R, \mathbf R^n )$ satisfies
	\begin{equation*}
		T(\phi) = {\textstyle \int} V \boundary \{ x \with (x,y) \in E
		\} ( \phi ( \cdot, y ) ) \ud \mathscr L^1 \, y \quad \text{for
		$\phi \in \mathscr D ( U \times \mathbf R, \mathbf R^n)$},
	\end{equation*}
	$p : \mathbf R^n \times \mathbf R \to \mathbf R^n$ and $q : \mathbf
	R^n \times \mathbf R \to \mathbf R$ are defined by
	\begin{equation*}
		p(x,y) = x \quad \text{and} \quad q(x,y) = y \quad \text{for
		$(x,y) \in \mathbf R^n \times \mathbf R$},
	\end{equation*}
	and $W \in \mathbf V_{m+1} ( U \times \mathbf R )$ is characterised by
	\begin{equation*}
		W(k) = {\textstyle \int} k((x,y), S \times \mathbf R ) \ud ( V
		\times \mathscr L^1 ) \, ((x,S),y)
	\end{equation*}
	whenever $k \in \mathscr K ( ( U \times \mathbf R ) \times \mathbf G (
	\mathbf R^n \times \mathbf R, m+1 ) )$.
	
	Then, there holds
	\begin{equation*}
		W \boundary E ( \theta ) = T ( p \circ \theta ) -
		{\textstyle\int} ( q \circ \theta ) (x,f(x)) \ud \| V \| \, x
		\quad \text{for $\theta \in \mathscr D ( U \times \mathbf R,
		\mathbf R^n \times \mathbf R )$}.
	\end{equation*}
	Thus, $\| W \boundary E \|$ is a Radon measure if and only if $\| T
	\|$ has this property; in this case, $f$ possesses \emph{generalised
	bounded variation} with respect to $V$.
\end{theorem}

\begin{proof}
	Firstly, defining $\iota : ( U \times \mathbf G (n,m) ) \times \mathbf
	R \to ( U \times \mathbf R ) \times \mathbf G ( \mathbf R^n \times
	\mathbf R, m+1 )$ by $\iota ((x,S),y) = ((x,y), S \times \mathbf R )$
	whenever $x \in U$, $y \in \mathbf R$, and $S \in \mathbf G (n,m)$, we
	conclude that $W = \iota_\# ( V \times \mathscr L^1 )$ from
	\cite[2.2.17, 2.4.18\,(1)]{MR41:1976}.  Moreover, we deduce $\| W \| =
	\| V \| \times \mathscr L^1$ and $\| \updelta W \| = \| \updelta V \|
	\times \mathscr L^1$ from \cite[3.6\,(1)\,(5)]{MR3625810} in
	conjunction with \cite[3.4\,(1)\,(2)]{MR3528825} applied with $T$
	replaced by $\updelta W$ and \cite[2.6.2\,(4)]{MR41:1976}.  The last
	two equations imply that $\| \updelta W \|$ is a Radon measure and
	that $E$ is a $\| W \| + \| \updelta W \|$ measurable set by
	\cite[2.6.2\,(2)]{MR41:1976}.
	
	Next, suppose $\theta \in \mathscr D ( U \times \mathbf R, \mathbf R^n
	\times \mathbf R )$ and let $\phi = p \circ \theta$ and $\psi = q
	\circ \theta$.  Noting
	\begin{equation*}
		( S \times \mathbf R)_\natural \bullet \Der \theta (x,y) =
		S_\natural \bullet \Der ( \phi ( \cdot, y )) (x) + ( \psi
		(x,\cdot))' (y)
	\end{equation*}
	whenever $(x,y) \in U \times \mathbf R$ and $S \in \mathbf G (n,m)$ by
	\cite[3.6\,(4)]{MR3625810}, we employ the preceding paragraph,
	\cite[2.4.18\,(1)\,(2), 2.6.2\,(2)\,(4), 2.9.20\,(1)]{MR41:1976}, and
	\cite[4.1]{MR3777387} to obtain
	\begin{align*}
		& \updelta ( W \restrict E \times \mathbf G ( \mathbf R^n
		\times \mathbf R, m+1 ) ) ( \theta ) \\
		& \qquad = {\textstyle\int_{E \times \mathbf G ( \mathbf R^n
		\times \mathbf R, m+1 )}} R_\natural \bullet \Der \theta (x,y)
		\ud \iota_\# ( V \times \mathscr L^1 ) ( (x,y),R ) \\
		& \qquad = {\textstyle\iint_{\{(x,S) \with (x,y) \in E, S \in
		\mathbf G (n,m) \}}} S_\natural \bullet \Der ( \phi ( \cdot,
		y) )(x) \ud V \, (x,S) \ud \mathscr L^1 \, y \\
		& \qquad \phantom{=} \ + {\textstyle\iint_{\{y \with (x,y) \in
		E \}}} ( \psi (x,\cdot))'(y) \ud \mathscr L^1 \, y \ud \| V \|
		\, x \\
		& \qquad = {\textstyle\int} \updelta ( V \restrict \{ x \with
		(x,y) \in E \} \times \mathbf G (n,m) ) ( \phi ( \cdot, y ) )
		\ud \mathscr L^1 \, y + {\textstyle\int} \psi (x,f(x)) \ud \|
		V \| \, x.
	\end{align*}
	Combining this with the equation (see \cite[3.6\,(5)\,(6)]{MR3625810})
	\begin{equation*}
		( (\updelta W) \restrict E ) ( \theta ) = {\textstyle\int} (
		(\updelta V) \restrict \{ x \with (x,y) \in E \} ) ( \phi (
		\cdot, y ) ) \ud \mathscr L^1 \, y,
	\end{equation*}
	the conclusion follows.
\end{proof}

\begin{remark} \label{remark:equiv_bv}
	If $f$ has generalised bounded variation with respect to $V$, then
	abbreviating $E(y) = \{ x \with (x,y) \in E \}$ for $y \in \mathbf R$
	and taking $J = \mathbf R$ and $Z = \mathbf R^n$ in \cite[3.2,
	3.4\,(2)]{MR3528825} shows the \emph{coarea formulae}
	\begin{equation*}
		\begin{aligned}
			T ( \phi ) & = {\textstyle\int} V \boundary E(y) (
			\phi ( \cdot, y ) ) \ud \mathscr L^1 \, y, \\
			{\textstyle\int} g \ud \| T \| & = {\textstyle\iint}
			g(x,y) \ud \| V \boundary E(y) \| \, x \ud \mathscr
			L^1 \, y
		\end{aligned}
	\end{equation*}
	whenever $\phi \in \mathbf L_1 ( \| T \|, \mathbf R^n )$ and $g$ is a
	$\| T \|$ integrable function.  Therefore, $f$ has generalised bounded
	variation with respect to $V$ if and only if
	\begin{equation*}
		{\textstyle\int_{-s}^s} \| V \boundary E (y) \| (K) \ud
		\mathscr L^1 \, y < \infty
	\end{equation*}
	whenever $K$ is a compact subset of $U$ and $0 \leq s < \infty$; in
	the case that $f$ is the characteristic function of some $\| V \| + \|
	\updelta V \|$ measurable set $A$, the latter condition is equivalent
	to $\| V \boundary A\| $ being a Radon measure.
\end{remark}

\begin{remark} \label{remark:GBV}
	In the case $n = m$ and $\| V \| = \mathscr L^n | \mathbf 2^U$, the
	function $f$ possesses generalised bounded variation with respect to
	$V$ if and only if $f$ is a \emph{generalized function of bounded
	variation} in the sense of \cite[4.26]{MR2003a:49002}; in fact, in
	view of \ref{remark:equiv_bv}, this readily follows from
	\cite[4.27]{MR2003a:49002} and \cite[Lemma 1]{MR1812124}.
\end{remark}

\begin{corollary} \label{corollary:equiv_bv}
	Suppose additionally $f$ possesses generalised bounded variation with
	respect to $V$ and $F : \dmn f \to U \times \mathbf R$ satisfies
	$F(x)=(x,f(x))$ for $x \in \dmn f$.
	
	Then, $f$ is generalised $V$ weakly differentiable if and only if $\|
	W \boundary E \|$, or equivalently $\| T \|$, is absolutely continuous
	with respect to $F_\# \| V \|$.
\end{corollary}

\begin{proof}
	There exists a Borel function $k : U \times \mathbf R \to \Hom (
	\mathbf R^n, \mathbf R ) \cap \{ h \with |h|=1 \}$ such that
	\begin{equation*}
		T ( \phi ) = {\textstyle \int \langle \phi, k \rangle \ud \| T
		\|} \quad \text{for $\phi \in \mathscr D ( U \times \mathbf R,
		\mathbf R^n )$}
	\end{equation*}
	by \cite[2.3.6, 4.1.5]{MR41:1976}.  Since $F_\# \| V \|$ is a Radon
	measure by \cite[2.11]{DOI_10.4171_RMI_1487} and we have $\| T \| \leq
	\| W \boundary E \| \leq \| T \| + F_\# \| V \|$ by
	\ref{thm:equiv_bv}, the two absolute continuity conditions are
	equivalent.  If $f$ is generalised $V$ weakly differentiable, they are
	satisfied by \cite[8.5]{MR3528825} and \cite[4.11]{MR3777387}.  If $\|
	T \|$ is absolutely continuous with respect to $F_\# \| V \|$, then,
	by \cite[2.8.18, 2.9.2, 2.9.7]{MR41:1976} and
	\cite[2.8]{DOI_10.4171_RMI_1487}, there exists a nonnegative Borel
	function $g \in \mathbf L_1^{\textup{loc}} ( F_\# \| V \| )$ with $\|
	T \| = ( F_\# \| V \| ) \restrict g$.  Employing \cite[2.4.10,
	2.4.18\,(1)]{MR41:1976} and \cite[8.1, 8.2]{MR3528825}, we therefore
	verify that $f$ is generalised $V$ weakly differentiable with $V
	\weakD f(x) = (gk)(F(x))$ for $\| V \|$ almost all $x$.
\end{proof}

\section{Partitions}

Properties of sets with vanishing distributional boundary follow in
\ref{lemma:zero_distributional_boundary}.  The distributional boundary of an
intersection of certain sets then appears in \ref{lemma:little_more_general}
and \ref{remark:perimeter}; in \ref{example:perimeter-continued}, this leads
in particular to an example relating indecomposability of varifolds to the
notion of \cite[p.\,52]{MR1812124} for sets of finite perimeter bearing the
same name.  With these preparations at hand, we obtain basic properties of
partitions in \ref{def:partition}--\ref{remark:no-adapted-decomposition}
including Example \ref{Example:no-adapt} of the
introductory section in \ref{example:no-adapted-decomposition}.

\begin{lemma} \label{lemma:zero_distributional_boundary}
	Suppose $m$ and $n$ are positive integers, $m \leq n$, $U$ is an open
	subset of $\mathbf R^n$, $V \in \mathbf V_m ( U )$, $\| \updelta V \|$
	is a Radon measure, $E$ is $\| V \| + \| \updelta V \|$ measurable, $V
	\boundary E = 0$, and $W = V \restrict E \times \mathbf G (n,m)$.
	
	Then, there holds $\| \updelta W \| = \| \updelta V \| \restrict E$,
	$\| \updelta W \|_{\| W \|} = \| \updelta V \|_{\| V \|} \restrict E$,
	and
	\begin{align*}
		\boldsymbol \upeta (W,x) & = \boldsymbol \upeta (V,x) \quad
		\text{for $\| \updelta W \|$ almost all $x$}, \\
		\mathbf h (W,x) & = \mathbf h (V,x) \quad \text{for $\| W \|$
		almost all $x$}.
	\end{align*}
\end{lemma}

\begin{proof}
	We have $\| W \| = \| V \| \restrict E$ and $\updelta W = ( \updelta
	V) \restrict E$.  Taking $S = \updelta W$ in
	\ref{example:distributions_representable_by_integration} yields $\|
	\updelta W \| = \| \updelta V \| \restrict E$. Next, we verify the
	last conclusion; in fact, for $x \in U$, the condition
	\begin{equation*}
		\lim_{r \to 0+} \frac{( \| V \| + \| \updelta V \| ) ( \mathbf
		B (x,r) \without E )}{\| V \| \, \mathbf B (x,r)} = 0,
	\end{equation*}
	which is valid for $\| V \|$ almost all $x \in E$ as follows from
	\cite[2.8.18, 2.9.7]{MR41:1976} applied with $\psi = ( \| V \| + \|
	\updelta V \| ) \restrict ( U \without E )$, $\phi = \| V \|$, and $A
	= E$, implies
	\begin{equation*}
		\mathbf h (W,x) \bullet v = - \lim_{r \to 0+} \frac{( \updelta
		W) ( b_{x,r} \cdot v )}{\| W \| \, \mathbf B(x,r)} = - \lim_{r
		\to 0+} \frac{( \updelta V ) ( b_{x,r} \cdot v )}{\| V \| \,
		\mathbf B (x,r)} = \mathbf h(V,x) \bullet v
	\end{equation*}
	whenever $v \in \mathbf R^n$. Thus, we have $\| \updelta W \|_{\| W
	\|} = \| \updelta V \|_{\| V \|} \restrict E$ by
	\ref{remark:polar_decomposition}.  Finally, as
	\begin{align*}
		(\updelta W ) ( \theta ) & = ( ( \updelta V ) \restrict E ) (
		\theta ) \\
		& = {\textstyle\int_E} \boldsymbol \upeta (V,x) \bullet \theta
		(x) \ud \| \updelta V \| \, x = {\textstyle\int} \boldsymbol
		\upeta (V,x) \bullet \theta (x) \ud \| \updelta W \| \, x
	\end{align*}
	for $\theta \in \mathscr D ( U, \mathbf R^n )$ by
	\ref{remark:polar_decomposition}, the remaining conclusion follows
	from \ref{example:distributions_representable_by_integration}.
\end{proof}

\begin{lemma} \label{lemma:little_more_general}
	Suppose $m$ and $n$ are positive integers, $m \leq n$, $U$ is an open
	subset of $\mathbf R^n$, $V \in \mathbf V_m ( U )$, $\| \updelta V \|$
	is a Radon measure, $E$ is a $\| V \| + \| \updelta V \|$ measurable
	set, $\| V \boundary E \|$ is a Radon measure, $W = V \restrict E
	\times \mathbf G(n,m)$, and $F$ is a $\| V \| + \| \updelta V \| + \|
	V \boundary E \|$ measurable set.
	
	Then, there holds
	\begin{equation*}
		V \boundary ( E \cap F ) = W \boundary F + ( V \boundary E )
		\restrict F.
	\end{equation*}
\end{lemma}

\begin{proof}
	Approximation of $F$ reduces the assertion to \cite[5.4]{MR3528825}.
\end{proof}

\begin{remark} \label{remark:perimeter}
	If $m=n$, $U = \mathbf R^n$, $\| V \| = \mathscr L^n$, and $E$, $F$,
	$P$, $Q$, $A$, and $B$ are as in \ref{example:perimeter}, then we
	compute, for $\theta \in \mathscr D ( \mathbf R^n, \mathbf R^n )$,
	that
	\begin{align*}
		& W \boundary ( Q \cup B ) ( \theta ) \\
		& \qquad {\textstyle = \int_P \theta (x) \bullet \mathbf n
		(F,x) \ud \mathscr H^{n-1} \, x + \int_{A \cap B} \theta (x)
		\bullet \frac{\mathbf n (F,x) - \mathbf n (E,x)} 2 \ud
		\mathscr H^{n-1} \, x}
	\end{align*}
	by \cite[4.5.6\,(5)]{MR41:1976}; if additionally $F \subset E$, then
	the second integral is zero.
\end{remark}

\begin{example} \label{example:perimeter-continued}
	Suppose $V \in \mathbf V_n ( \mathbf R^n )$ such that $\| V \| =
	\mathscr L^n$, the set $E$ is $\| V \|$ measurable, and $\| V
	\boundary E \| ( \mathbf R^n ) < \infty$.  Then, $V \restrict E \times
	\mathbf G(n,n)$ is indecomposable if and only if $E$ is indecomposable
	in the sense of \cite[p.\,52]{MR1812124} as may be verified using
	\ref{lemma:little_more_general} and \ref{remark:perimeter}. $\big
	($This comparison was planned for inclusion in \cite{MR3528825} but
	was unintentionally omitted in that paper.$\big)$
\end{example}

\begin{definition} \label{def:partition}
	Suppose $m$ and $n$ are positive integers, $m \leq n$, $U$ is an open
	subset of $\mathbf R^n$, $V \in \mathbf V_m ( U )$, $\| \updelta V \|$
	is a Radon measure, and $\Pi \subset \mathbf V_m ( U )$.
	
	Then, $\Pi$ is termed a \emph{partition of $V$} if and only if the
	following three conditions are satisfied.
	\begin{enumerate}
		\item Whenever $W \in \Pi$, there exists a $\| V \| + \|
		\updelta V \|$ measurable set $E$ such that $W = V \restrict E
		\times \mathbf G (n,m)$, $\| V \| ( E ) > 0$, and $V \boundary
		E = 0$.
		\item \label{item:partition:weight} If $k \in \mathscr K ( U
		\times \mathbf G (n,m) )$, then $V (k) = \sum_{W \in \Pi} W
		(k)$.
		\item \label{item:partition:variation} If $f \in \mathscr K (
		U )$, then $\| \updelta V \| (f) = \sum_{W \in \Pi} \|
		\updelta W \| (f)$.
	\end{enumerate}
\end{definition}

\begin{remark} \label{remark:partition}
	The family $\Pi$ is countable and the equations in
	\eqref{item:partition:weight} and \eqref{item:partition:variation}
	also hold for $V$ integrable functions $k$ and $\| \updelta V \|$
	integrable functions $f$, respectively.  Thus, whenever $F$ is a $\| V
	\| + \| \updelta V \|$ measurable set,
	\ref{remark:polar_decomposition} and
	\ref{lemma:zero_distributional_boundary} yield
	\begin{equation*}
		V \boundary F ( \theta ) = \sum_{W \in \Xi} W \boundary F (
		\theta ) \quad \text{for $\theta \in \mathscr D ( U, \mathbf
		R^n )$}.
	\end{equation*}
\end{remark}

\begin{remark} \label{remark:partition-vs-decomposition}
	Every decomposition of $V$ forms a partition of $V$.  Conversely, if
	$\xi : \Pi \to \mathbf 2^{\mathbf V_m ( U )}$ such that $\xi ( W )$ is
	a decomposition of $W$ for $W \in \Pi$, then the map $\xi$ is
	univalent, $\im \xi$ is disjointed, and $\bigcup \im \xi$ is a
	decomposition of $V$; in fact, for $X \in \bigcup \im \xi$, we
	construct a $\| V \| + \| \updelta V \|$ measurable set $E$ with $\| V
	\| ( E ) > 0$, $X = V \restrict E \times \mathbf G (n,m)$, and $V
	\boundary E = 0$ using \cite[5.4]{MR3528825} or
	\ref{lemma:little_more_general}.  In case $V$ is rectifiable, such
	$\xi$ exists by \cite[6.12]{MR3528825}.
\end{remark}

\begin{remark} \label{remark:char-partition}
	Employing \ref{lemma:zero_distributional_boundary} and
	\ref{remark:partition}, we may verify that a subfamily $\Pi$ of
	$\mathbf V_m (U)$ is a partition of $V$ if and only if there exists a
	map $\pi : \Pi \to \mathbf 2^U$ mapping distinct members of $\Pi$ to
	disjoint Borel subsets of $U$ of positive $\| V \|$ measure satisfying
	\begin{equation*}
		{\textstyle ( \| V \| + \| \updelta V \| ) \big ( U \without
		\bigcup \im \pi \big ) = 0}
	\end{equation*}
	and $W = V \restrict \pi (W) \times \mathbf G (n,m)$ with $V \boundary
	\pi (W) = 0$ for $W \in \Pi$.
\end{remark}

\begin{remark} \label{remark:partition-earlier}
	If $m$, $n$, $U$, $V$, and $\Xi$ satisfy the hypotheses of
	\cite[4.14]{MR3777387}, then $\Xi \without \{ 0 \}$ is a partition of
	$V$ by \ref{lemma:zero_distributional_boundary} and
	\ref{remark:char-partition}.
\end{remark}

\begin{example} \label{example:no-adapted-decomposition}
	There exist $L$, $L_1, \ldots, L_4$, $M$, $V$, $\kappa$, and $f$ such
	that
	\begin{gather*}
		\text{$L$ is a submanifold of class $\infty$ of $\mathbf
		R^2$}, \quad \dim L = 1, \quad ( \Clos L ) \without L = \{ 0
		\}, \\
		\text{$L_1, \ldots, L_4$ enumerate the connected components of
		$L$}, \\
		\Tan (L_1,0) = \Tan (L_2,0) = \{ (t,0) \with t \geq 0 \}, \\
		\Tan (L_3,0) = \Tan (L_4,0) = \{ (t,0) \with t \leq 0 \}, \\
		\text{$\Clos (L_i \cup L_j)$ is a submanifold of class
		$\infty$ for $i \in \{ 1,2 \}$ and $j \in \{ 3,4 \}$}, \\
		M = \mathbf R^3 \cap \{ (x_1,x_2,x_3) \with (x_1,x_2) \in L
		\}, \quad V \in \mathbf{RV}_2 ( \mathbf R^3 ), \\
		\| V \| = \mathscr H^2 \restrict M, \quad 0 \leq \kappa <
		\infty, \quad \| \updelta V \| \leq \kappa \| V \|, \\
		\text{$f : M \to \{ y \with 0 \leq y \leq 1 \}$ is of class
		$\infty$ relative to $M$}, \\
		\text{$f \in \mathbf T(V)$}, \quad \text{but} \quad \text{$f
		\notin \mathbf T(W)$ whenever $W$ is a component of $V$}.
	\end{gather*}
\end{example}

\begin{proof}
	We choose $\gamma : \mathbf R \to \{ s \with 0 \leq s \leq 1 \}$ of
	class $\infty$ such that
	\begin{equation*}
		\{ r \with \gamma (r) = 0 \} = \{ r \with - \infty < r \leq 0
		\} \quad \text{and} \quad \{ r \with \gamma (r) = 1 \} = \{ r
		\with 1 \leq r < \infty \}.
	\end{equation*}
	For $i \in \{ 1, 2, 3, 4 \}$, we define functions $h_i : \mathbf R \to
	\mathbf R$ of class $\infty$ by
	\begin{equation*}
		h_1(r) = 0, \quad h_2(r) = \gamma(r), \quad h_3(r) =
		\gamma(|r|), \quad h_4 (r) = \gamma (-r)
	\end{equation*}
	whenever $r \in \mathbf R$ and associate with them $M_i$ and $V_i \in
	\mathbf{RV}_2 ( \mathbf R^3 )$ such that
	\begin{gather*}
		M_i = \mathbf R^3 \cap \{ (x_1,x_2,x_3) \with x_1 \neq 0, x_2
		= h_i(x_1) \}, \quad \| V_i \| = \mathscr H^2 \restrict M_i.
	\end{gather*}
	Defining $V = V_1 + V_3 = V_2 + V_4$, there exists $0 \leq \kappa <
	\infty$ with $\| \updelta V \| \leq \kappa \| V \|$ and $V_1, \ldots,
	V_4$ enumerate the components of $V$.  Abbreviating
	\begin{equation*}
		C = \mathbf R^2 \cap \{ (u_1,u_2) \with 0 < u_1 < u_2 \},
	\end{equation*}
	we define $g : \mathbf R^2 \cap \{ (u_1,u_2) \with u_1 \neq 0 \} \to
	\{ y \with 0 < y \leq 1 \}$ by
	\begin{equation*}
		\text{$g(u)=\gamma(u_1/u_2)$ if $u_1 >0<u_2$}, \quad
		\text{$g(u)=1$ else}
	\end{equation*}
	whenever $u=(u_1,u_2) \in \mathbf R^2$ and $u_1 \neq 0$.  We see that
	$g$ is of class $\infty$ with
	\begin{equation*}
		{\textstyle\int_{\mathbf B(0,r)} | \Der g | \ud \mathscr L^2}
		< \infty \quad \text{for $0 \leq r < \infty$},
	\end{equation*}
	since $C = \{ u \with g(u) < 1 \}$ and $| \Der g (u) | \leq ( 2 \sup
	\im |\gamma'| )/ u_2$ for $u=(u_1,u_2) \in C$.  We let $M = M_1 \cup
	M_3 = M_2 \cup M_4$, notice $\| V \| = \mathscr H^2 \restrict M$, and
	define the function $f : M \to \{ y \with 0 \leq y \leq 1 \}$ of class
	$\infty$ relative to $M$ by
	\begin{equation*}
		\text{$f(x) = g(x_1,x_3)$ if $x_2=0$}, \quad \text{$f(x) =
		1-g(x_1,x_3)$ if $x_2>0$}
	\end{equation*}
	whenever $x=(x_1,x_2,x_3) \in M$.  We take
	\begin{equation*}
		E = \{ (x,y) \with f(x)>y \}, \quad G = \{ (x,y) \with f(x) =
		y \},
	\end{equation*}
	and verify $\mathscr H^2 ( G \cap K ) < \infty$ whenever $K$ is a
	compact subset of $\mathbf R^3 \times \mathbf R$ by means of
	\cite[3.2.20]{MR41:1976}.  Associating $W_i$ with $V_i$ as in
	\ref{thm:equiv_bv} and noting $\| \updelta W_i \| \leq \kappa \| W_i
	\|$ by \cite[3.6\,(1)\,(6)]{MR3625810}, we define $T = ( \mathbf R^3
	\times \mathbf R ) \cap \{ ((x_1,x_2,x_3),y) \with x_1 = 0 \}$ and
	apply \cite[5.9\,(3)]{MR3528825} with $V$ replaced by $W_i$ to compute
	\begin{gather*}
		\| W_i \boundary E \| \restrict ( \mathbf R^3 \times \mathbf
		R) \without T = \mathscr H^2 \restrict G \cap ( M_i \times
		\mathbf R ) \quad \text{for $i \in \{ 1,2,3,4 \}$}, \\
		\| W_1 \boundary E \| \restrict T = \| W_3 \boundary E \|
		\restrict T = \mathscr H^2 \restrict T \cap \{ (x,y) \with
		x_2=0, x_3 \geq 0, 0 \leq y \leq 1 \}, \\
		\| W_2 \boundary E \| \restrict T = \| W_4 \boundary E \|
		\restrict T = \mathscr H^2 \restrict T \cap \{ (x,y) \with
		x_2=0, x_3 \leq 0, 0 \leq y \leq 1 \}, \\
		( W_1 \boundary E + W_3 \boundary E ) \restrict T = 0 = ( W_2
		\boundary E + W_4 \boundary E ) \restrict T.
	\end{gather*}
	Defining $W = W_1 + W_3 = W_2 + W_4$, we infer
	\begin{equation*}
		\| W \boundary E \| = \mathscr H^2 \restrict G
	\end{equation*}
	by \cite[5.2]{MR3528825}.  In view of \ref{corollary:equiv_bv}, the
	conclusion is now evident.
\end{proof}

\begin{remark} \label{remark:no-adapted-decomposition}
	Whereas, whenever $V$ is an $m$ dimensional rectifiable varifold in an
	open subset of $\mathbf R^n$ such that $\| \updelta V \|$ a Radon
	measure, a generalised $V$ weakly differentiable function may be
	defined (see \cite[6.10, 6.12, 8.24]{MR3528825}) by selecting a
	decomposition $\Xi$ of $V$ and, subject to the natural summability
	condition, a member of $\mathbf T(W)$ for each $W \in \Xi$, the
	preceding example shows that, in general, not all members of $\mathbf
	T(V)$ arise in this way.
\end{remark}

\section{Examples of decomposable varifolds}

After some preparations in \ref{lemma:criterion_indecomposable}--%
\ref{def:locally-finiteness-decompositions}, we construct Example
\ref{intro-example:touching-spheres} of the introductory section in
\ref{example:indecomposability_of_type_f} and
\ref{remark:indecomposability_of_type_f}.  In
\ref{definition:indecomposable_chains}--\ref{remark:indecomposable_chains}, we
introduce the terminology of indecomposability for integral chains with
coefficients in a complete normed commutative group to construct Example
\ref{intro-example:three-line-segments} of the introductory section in
\ref{example:four-arcs}.  In
\ref{def:mean-curvature-of-sets}--\ref{remark:immersion_varifold}, we discuss
basic properties of immersions and their associated varifolds to construct
Example \ref{Thm:immersion} of the introductory section in
\ref{example:inifinite_components}.

\begin{lemma} \label{lemma:criterion_indecomposable}
	Suppose $m$ and $n$ are positive integers, $m \leq n$, $U$ is an open
	subset of $\mathbf R^n$, $M$ is a connected $m$ dimensional
	submanifold of $U$ of class $2$ meeting every compact subset of $U$ in
	a set of finite $\mathscr H^m$ measure, $V \in \mathbf V_m ( U )$, $\|
	\updelta V \|$~is a Radon measure, $0 < c < \infty$,
	\begin{equation*}
		\spt ( V - c \mathbf v_m ( M ) ) \subset ( U \without M)
		\times \mathbf G (n,m),
	\end{equation*}
	$E$ is a $\| V \| + \| \updelta V \|$ measurable set, and $V \boundary
	E = 0$.
	
	Then, there holds $\| V \| ( M \cap E ) = 0$ or $\| V \| ( M \without
	E ) = 0$; in particular, if $V = c \mathbf v_m ( M )$, then $V$ is
	indecomposable.
\end{lemma}

\begin{proof}
	We use \cite[4.4, 4.6\,(3)]{MR0307015} to conclude that each $z \in M$
	admits a neighbourhood $X$ in $M$ such that $E \cap X$ is $\mathscr
	H^m$~almost equal to $\varnothing$ or $E \cap X$, whence the principal
	conclusion follows because $M$ is connected.
\end{proof}

\begin{definition} \label{def:locally-finiteness-decompositions}
	Suppose $m$ and $n$ are positive integers, $m \leq n$, $U$ is an open
	subset of $\mathbf R^n$, and $\Xi \subset \mathbf V_m (U)$.
	
	Then, $\Xi$ is termed \emph{locally finite} if and only if
	\begin{equation*}
		\card ( \Xi \cap \{ W \with K \cap \spt \| W \| \neq
		\varnothing \} ) < \infty
	\end{equation*}
	whenever $K$ is a compact subset of $U$.
\end{definition}

\begin{remark} \label{remark:criterion-locally-finite}
	If $V \in \mathbf V_m ( U )$ satisfies that $\|
	\updelta V \|$ is a Radon measure and that, for each $a \in U$, there
	exists a positive radius $r$ with $\mathbf B ( a, 2r ) \subset U$ and
	\begin{equation*}
		\inf \{ \| W \| \, \mathbf B ( a, 2r ) \with \text{$W$ is a
		component of $V$, $\mathbf B (a,r) \cap \spt \| W \| \neq
		\varnothing$} \} > 0,
	\end{equation*}
	then every decomposition of $V$ is locally finite.
\end{remark}

\begin{example} \label{example:indecomposability_of_type_f}
	Suppose $m$ is a positive integer, $n = m+1$, $\mathbf q$ is as in
	\ref{miniremark:p-q}, and
	\begin{equation*}
		A = \mathbf q^\ast [ 2 \mathbf Z ], \quad M = \mathbf R^n \cap
		\{ x \with \dist (x,A) = 1 \}, \quad V = \mathbf v_m ( M ) \in
		\mathbf V_m ( \mathbf R^n ),
	\end{equation*}
	Then, we have $\boldsymbol \Uptheta^m ( \| V \|, x )
	= 1$ for $\| V \|$ almost all $x$, $\spt \| V \|$ is connected, $\|
	\updelta V \| \leq m \| V \|$, and the following three statements
	hold.
	\begin{enumerate}
		\item Whenever $I$ is a subinterval of $\mathbf R$, there
		holds $V \boundary \mathbf q^{-1} [I] = 0$ if and only if
		$\Bdry I$ is contained in the set of odd integers.
		\item The family $\{ V \restrict \mathbf q^{-1} [ \mathbf U
		(2i,1) ] \times \mathbf G(n,m) \with i \in \mathbf Z \}$ is a
		decomposition of $V$; in particular, $V$ is decomposable.
		\item Every decomposition of $V$ is locally finite.
	\end{enumerate}
	Noting that $N = M \without \mathbf q^{-1} [ \{ z \with \text{$z$ odd
	integer} \} ]$ is an $m$ dimensional submanifold of class $\infty$ of
	$\mathbf R^n$, it suffices to apply
	\ref{lemma:criterion_indecomposable} to each connected component of
	$N$.
\end{example}

\begin{remark} \label{remark:indecomposability_of_type_f}
	Similar statements hold if $A$ is replaced by its subset $\mathbf
	q^\ast [ \{ 0,2 \} ]$ in the preceding example; in that case, the
	decomposition of $V$ is unique.
\end{remark}

\begin{definition} \label{definition:indecomposable_chains}
	Suppose $m$ and $n$ are nonnegative integers, $n \geq 1$, $U$ is an
	open subset of $\mathbf R^n$, $G$ is a complete normed commutative
	group, and $S \in \mathbf I_m ( U, G)$.
	
	Then, $S$ is called \emph{indecomposable} if and only if there exists
	no $T \in \mathbf I_m( U, G)$ with
	\begin{equation*}
		\begin{gathered}
			T \neq 0 \neq S - T, \quad \| S \| = \| T \| + \| S-T
			\|, \\
			\text{either $m=0$ or $\| \boundary_G S \| = \|
			\boundary_G T \| + \| \boundary_G (S-T) \|$}.
		\end{gathered}
	\end{equation*}
\end{definition}

\begin{remark} \label{remark:indecomposable_chains}
	By \cite[4.1]{DOI_10.4171_RMI_1487}, an integral current $R \in
	\mathbf I_m ( \mathbf R^n )$ is indecomposable if and only if
	$\iota_{\mathbf R^n,m} (R) \in \mathbf I_m ( \mathbf R^n, \mathbf Z )$
	is indecomposable.
\end{remark}

\begin{example} \label{example:four-arcs}
	This example is based on \cite[4.1, 4.5]{DOI_10.4171_RMI_1487} with $G
	= \mathbf Z/ 3 \mathbf Z$.  Abbreviating $b_1 = (1,0)$, $b_2 = (0,1)$,
	and $b_3 = (-1,0)$, we define $B = \{ b_1, b_2, b_3 \}$ and
	\begin{equation*}
		Q = \sum_{i=1}^3 R_i, \quad \text{where $R_i = \boldsymbol [
		0, b_i \boldsymbol ] \in \mathbf I_1 ( \mathbf R^2 )$}
	\end{equation*}
	so that we have $\boundary Q = \left ( \sum_{b \in B} \boldsymbol
	\updelta_b \right) - 3 \boldsymbol \updelta_0 \in \mathbf I_0 (
	\mathbf R^2 )$ and $N = (\spt Q) \without \spt \boundary Q$ is a
	one-dimensional submanifold of class $1$ of $\mathbf R^2$ satisfying
	\begin{equation*}
		( \Clos N ) \without N = B \cup \{ 0 \}.
	\end{equation*}
	Choosing $g \in G$ with $|g| = 1$, we see $g$ generates $G$.  We
	define $S$ and $V$ such that
	\begin{equation*}
		S = \iota_{\mathbf R^2,1} ( Q ) \cdot g \in \mathbf I_1 (
		\mathbf R^2, G), \quad V \in \mathbf{RV}_1 ( \mathbf R^2 ),
		\quad \| V \| = \| S \|;
	\end{equation*}
	hence, $\| V \| = \mathscr H^1 \restrict N$ and $\| \updelta V \| =
	\mathscr H^0 \restrict B \cup \{ 0 \}$.  Moreover, $V$ is decomposable
	because
	\begin{equation*}
		\{ \mathbf v_1 ( \spt R_1 \cup \spt R_3 ) , \mathbf v_1 ( \spt
		R_2 ) \}
	\end{equation*}
	is the only decomposition of $V$ by
	\ref{lemma:criterion_indecomposable}.  Since $\boundary_G S = \left (
	\sum_{b \in B} \iota_{\mathbf R^2,0} ( \boldsymbol \updelta_b )
	\right) \cdot g$, we have $\spt \| \boundary_G S \| = B$.  Finally, we
	show that $S$ is indecomposable.  Whenever
	\begin{equation*}
		\text{$T \in \mathbf I_1 ( \mathbf R^2, G)$ with $\spt \| T \|
		\subset \Clos N$ and $N \cap \spt \| \boundary_G T \| =
		\varnothing$},
	\end{equation*}
	we observe that \cite[6.1]{DOI_10.4171_RMI_1487} may be applied to
	each connected component $M$ of $N$ to represent $T = \sum_{i=1}^3
	\iota_{\mathbf R^2,1} ( R_i ) \cdot h_i$ for some $h_i \in G$;
	accordingly, the condition $\| T \| + \| S-T \| \leq \|S\|$ would
	imply $|h_i| + |g-h_i| \leq 1$, hence $h_i \in \{ 0,g \}$, for $i \in
	\{1,2,3\}$ and the additional requirement $0 \notin \spt \boundary_G
	T$ would then allow us to conclude $\sum_{i=1}^3 h_i = 0$, hence $h_1
	= h_2 = h_3$, and thus either $T = 0$ or $T = S$.
\end{example}

\begin{definition} \label{def:mean-curvature-of-sets}
	Suppose $A \subset \mathbf R^n$, $a \in \mathbf R^n$, and $A$ is
	pointwise [approximately] differentiable of order $2$ at $a$.  Then,
	the \emph{pointwise} [\emph{approximate}] \emph{mean curvature vector}
	of $A$ at $a$ is defined by
	\begin{align*}
		\pt \mathbf h ( A, a ) & = \trace \pt \Der^2 A (a,\Tan(A,a))
		\in \mathbf R^n \\
		\big [ \ap \mathbf h ( A, a ) & = \trace \ap \Der^2 A (a) \in
		\mathbf R^n \big ].
	\end{align*}
\end{definition}

\begin{definition} \label{def:mean_curvature}
	Suppose $M$ is a manifold-with-boundary of class $k$, the map $F : M
	\to \mathbf R^n$ is an immersion of class $k$, and $k \geq 1$.  Then,
	for $c \in M$, the \emph{tangent cone} $\Tan (F,c)$,
	the \emph{normal cone} $\Nor (F,c)$, and, in case $k \geq 2$ and $c
	\notin \partial M$, also the \emph{mean curvature vector $\mathbf h
	(F,c)$} along $F$ at $c$ are characterised by
	\begin{equation*}
		\begin{gathered}
			\Tan (F,c) = \Tan (F[W],F(c)), \quad \Nor (F,c) = \Nor
			( F[W], F(c)), \\
			\mathbf h (F,c) = \mathbf h (F[W],F(c))
		\end{gathered}
	\end{equation*}
	whenever $W$ is an open neighbourhood of $c$ in $M$ and $F|W$ is an
	embedding.  Moreover, we define the \emph{exterior normal $\mathbf n (
	F, c )$} along $F$ at $c$ by
	\begin{equation*}
		\{ - \mathbf n ( F,c ) \} = \mathbf S^{n-1} \cap \Tan ( F, c )
		\cap \Nor ( F | \partial M, c ) \quad \text{for $c \in
		\partial M$}.
	\end{equation*}
\end{definition}

\begin{remark} \label{remark:second_fundamental_form}
	Whenever $\phi$ is a chart of $M$ of class $2$, $c \in ( \dmn \phi )
	\without \partial M$, $\psi = \phi^{-1}$, and $e_1, \ldots, e_m \in
	\mathbf R^m$ are such that $\langle e_1, \Der (F \circ \psi) (\phi(c))
	\rangle, \ldots, \langle e_m, \Der (F \circ \psi) (\phi(c)) \rangle$
	form an orthonormal basis of $\Tan ( F, c )$, we have (cf.\
	\cite[pp.\,423--424]{MR0307015})
	\begin{equation*}
		\mathbf h (F,c) = \sum_{i=1}^m \left \langle (e_i,e_i), \Nor
		(F,c)_\natural \circ \Der^2 (F \circ \psi) (\phi(c)) \right
		\rangle.
	\end{equation*}
\end{remark}

\begin{example} \label{example:immersion-pt-approx-diff}
	Suppose $k$, $m$, and $n$ are positive integers, $m \leq n$, $M$ is an
	$m$ dimensional manifold-with-boundary of class~$k$, $U$~is an open
	subset of $\mathbf R^n$, $F : M \to U$ is a proper immersion of class
	$k$, and $A = \im F$.  Then, there holds
	\begin{gather*}
		\text{$\sup \Number ( F, \cdot ) [K] < \infty$ whenever $K$ is
		a compact subset of $U$}, \\
		\text{$A$ is pointwise and approximately differentiable of
		order $k$ at $a$}, \\
		\Tan ( A, a ) = \Tan^m ( \mathscr H^m \restrict A, a ) = \Tan
		(F, c ), \\
		\pt \Der^k A (a,\Tan(A,a)) = \ap \Der^k A (a), \\
		\text{if $k \geq 2$, then $\pt \mathbf h ( A, a ) = \ap
		\mathbf h ( A, a ) = \mathbf h (F,c)$},
	\end{gather*}
	whenever $F(c) = a$, for $\mathscr H^m$ almost all $a \in A$; in fact,
	recalling \cite[3.1, 3.3, 3.16, 3.22, 4.1, 4.3, 4.10]{MR3978264},
	\cite[3.7]{MR3936235}, and \cite[2.10.19\,(4)]{MR41:1976}, we apply
	first \cite[Footnote 19]{MR3936235} and then \cite[3.11]{MR3936235}
	regarding approximate and pointwise differentiation, respectively.
\end{example}

\begin{definition} \label{def:immersed_varifold}
	Suppose $m$ and $n$ are positive integers, $m \leq n$, $M$~is an
	$m$~dimensional manifold-with-boundary of class~$1$, $U$~is an open
	subset of~$\mathbf R^n$, and $F : M \to U$ is a proper immersion of
	class~$1$.
	
	Then, we define the \emph{varifold~$V$ associated with $(F,U)$} by
	\begin{equation*}
		V(k) = {\textstyle \int k (x,\Tan(\im F,x)) \Number(F,x) \ud
		\mathscr H^m \, x} \quad \text{for $k \in \mathscr K ( U
		\times \mathbf G (n,m))$}.
	\end{equation*}
\end{definition}

\begin{remark} \label{remark:immersion_varifold}
	We notice that $V$ is rectifiable, $\| V \| = \mathscr H^m \restrict
	\Number(F,\cdot)$, and
	\begin{equation*}
		\Tan^m ( \| V \|, x ) = \Tan ( \im F, x ), \quad \boldsymbol
		\Uptheta^m ( \| V \|, x ) = \Number (F,x)
	\end{equation*}
	for $\| V \|$ almost all $x$.  If $M$ and $F$ are of class $2$, then
	we employ \cite[4.4, 4.7]{MR0307015},
	\ref{remark:polar_decomposition}, and
	\ref{example:immersion-pt-approx-diff} to verify firstly that $\|
	\updelta V \|$ is a Radon measure satisfying
	\begin{align*}
		( \updelta V ) ( \theta ) & = {\textstyle - \int_{F[M \without
		\partial M]} \big ( \sum_{c \in F^{-1} [ \{ x \} ]} \mathbf h
		( F, c) \big ) \bullet \theta (x) \ud \mathscr H^m \, x} \\
		& \phantom = \ + {\textstyle\int_{F[ \partial M]} \big (
		\sum_{c \in F^{-1} [ \{ x \} ]} \mathbf n ( F, c ) \big )
		\bullet \theta (x) \ud \mathscr H^{m-1} \, x}
	\end{align*}
	for $\theta \in \mathscr D ( U, \mathbf R^n )$ and secondly that, for
	$\| V \|$ almost all $x$, there holds
	\begin{equation*}
		\mathbf h (V,x) = \mathbf h (F,c) \quad \text{whenever $F(c) =
		x$};
	\end{equation*}
	in particular, $\| \updelta V \| \leq \| V \| \restrict | \mathbf h
	(V,\cdot) | + \mathscr H^{m-1} \restrict \Number (F|\partial M,
	\cdot)$ with equality in case $F|\partial M$ is an embedding.
\end{remark}

\begin{example} \label{example:inifinite_components}
	There exists a two-dimensional manifold-with-boundary $M$ of class
	$\infty$ and a proper immersion $F : M \to \mathbf R^3$ of class
	$\infty$ such that the varifold associated with $( F, \mathbf R^3 )$
	has a unique decomposition $\Xi$ and $\Xi$ is not locally finite.
\end{example}

\begin{proof}
	We define a closed subset of $\mathbf R$ by
	\begin{equation*}
		X = \mathbf R \cap \{ x \with \text{$x \leq 0$, or $x \geq 1$,
		or $x=1/i$ for some positive integer $i$} \}
	\end{equation*}
	and use \cite[2.14]{DOI_10.4171_RMI_1487} to construct a nonpositive
	function $f : \mathbf R \to \mathbf R$ of class $\infty$ satisfying
	\begin{equation*}
		X = \{ x \with f(x) = 0 \} \quad \text{and} \quad \sup \im |
		f'' | \leq r^{-1},
	\end{equation*}
	where $r = 4 \sup \im |f|$.  We set $T = \mathbf R^2 \cap \{ (x,y)
	\with |y|<r/2 \}$ and define
	\begin{equation*}
		A = \{ (x,y) \with f(x)<y \}, \quad G = \{ (x,y) \with f(x)=y
		\}.
	\end{equation*}
	Since we have $\sup \im \| \mathbf b (G,\cdot) \| \leq \sup \im
	|f''|$, we conclude $\reach (G) \geq r$ from
	\ref{thm:reach-real-valued-graphs} and
	\ref{remark:reach-real-valued-graphs}; hence, taking $U$ and $\delta$
	as in \ref{thm:reach-real-valued-graphs} and
	\ref{corollary:reach-real-valued-graphs}, we have
	\begin{equation*}
		G \subset T \subset U, \quad A \cup T = \mathbf R^2 \cap \{
		(x,y) \with y > - r/2 \}.
	\end{equation*}
	Accordingly, we may construct a function $g : A \cup T \to \mathbf R$
	of class $\infty$ such that
	\begin{alignat*}{2}
		g(x,y) & = 3^{1/2} \delta(x,y) & \quad & \text{if $-r/2 < y
		\leq r/4$}, \\
		g(x,y) & > 0 && \text{if $r/4 < y < r/2$}, \\
		g(x,y) & = 0 && \text{if $y \geq r/2$}
	\end{alignat*}
	whenever $(x,y) \in \mathbf R^2$ and $y > -r/2$; hence, $T \cap \{
	(x,y) \with g(x,y)= 0 \} = G$ and
	\begin{equation*}
		\Der_2 g(x,y) > 0 \quad \text{whenever $x \in \mathbf R$ and
		$-r/2 < y \leq r/4$}.
	\end{equation*}
	Defining $M = \mathbf R^3 \cap \{ (x,y,z) \with g(x,y) = z \}$ as well
	as closed sets
	\begin{equation*}
		E = M \cap \{ (x,y,z) \with (x,y) \in A \cup G \}, \quad F = M
		\cap \{ (x,y,z) \with y \geq 0 \},
	\end{equation*}
	$E$ and $F$ are two-dimensional submanifolds-with-boundary of class
	$\infty$ of $\mathbf R^3$,
	\begin{alignat*}{2}
		\mathbf n ( M, E, (x,y,z) ) & = 2^{-1} \big ( \mathbf n
		(A,(x,y)), -3^{1/2} \big ) & \quad & \text{if $(x,y) \in G$
		and $z=0$} \\
		\mathbf n (M,E, (x,y,z)) & = 0 && \text{else}
	\end{alignat*}
	whenever $(x,y,z) \in \mathbf R^3$, where we have identified $\mathbf
	R^3 \simeq \mathbf R^2 \times \mathbf R$, and
	\begin{alignat*}{2}
		\mathbf n ( M, F, (x,y,z) ) & \in \{ (u,v,w) \with v<0,w<0 \}
		& \quad & \text{if $y = 0$, $g(x,y)=z$}, \\
		\mathbf n ( M,F, (x,y,z) ) & = 0 && \text{else}
	\end{alignat*}
	for $(x,y,z) \in \mathbf R^3$.  Defining
	\begin{align*}
		B_1 & = \mathbf R^3 \cap \{(x,y,z) \with (x,y) \in A \cup G,
		\, z = g(x,y)\}, \\
		B_2 & = \mathbf R^3 \cap \{(x,y,z) \with (x,-y) \in A \cup G,
		\, z = g(x,-y)\}, \\
		B_3 & = \mathbf R^3 \cap \{(x,y,z) \with (x,y) \in A \cup G,
		\, z = -g(x,y)\}, \\
		B_4 & = \mathbf R^3 \cap \{(x,y,z) \with (x,-y) \in A \cup G,
		\, z = -g(x,-y)\}, \\
		B_5 & = \mathbf R^3 \cap \{(x,y,z) \with z=0 \},
	\end{align*}
	we take $V = \sum_{j=1}^5 \mathbf v_2 ( B_j ) \in \mathbf V_2 (
	\mathbf R^3 )$ and compute
	\begin{equation*}
		( \updelta V ) ( \theta ) = - \Big ( \sum_{j=1}^4
		{\textstyle\int_{B_j}} \theta \bullet \mathbf h ( B_j, \cdot )
		\ud \mathscr H^2 \Big ) + {\textstyle\int} \mathbf \theta
		\bullet \mathbf n ( B_5, K, \cdot ) \ud \mathscr H^1
	\end{equation*}
	for $\theta \in \mathscr D ( \mathbf R^3, \mathbf R^3 )$, where $K =
	\mathbf R^3 \cap \{ (x,y,z) \with (x,y) \in A, (x,-y) \in A, z = 0
	\}$.  The connected components of $K$ are given by
	\begin{equation*}
		D_i = \mathbf R^3 \cap \big \{ (x,y,z) \with {\textstyle \frac
		1{i+1} < x < \frac 1i, |y| < |f(x)|, z=0} \big \}
	\end{equation*}
	corresponding to every positive integer $i$.
	
	Next, we consider the closed sets and one-dimensional submanifolds
	$L_1, \ldots, L_6$ of class $\infty$ of $\mathbf R^3$ defined by
	\begin{align*}
		L_1 & = \mathbf R^3 \cap \{ (x,y,z) \with y = r/2, z = 0 \},
		\\
		L_2 & = \mathbf R^3 \cap \{ (x,y,z) \with (x,-y) \in G, z = 0
		\}, \\
		L_3 & = \mathbf R^3 \cap \{ (x,y,z) \with y = 0, g(x,y)=z \},
		\\
		L_4 & = \mathbf R^3 \cap \{ (x,y,z) \with y = 0, -g(x,y)=z \},
		\\
		L_5 & = \mathbf R^3 \cap \{ (x,y,z) \with (x,y) \in G, z = 0
		\}, \\
		L_6 & = \mathbf R^3 \cap \{ (x,y,z) \with y = - r/2, z = 0 \},
	\end{align*}
	note $L_j \subset \spt \| V \|$ for $j = 1, \ldots, 6$, set $S =
	\bigcup_{j=1}^6 L_j$ and $R = ( \spt \| V \| ) \without S$, let
	\begin{align*}
		P_{i,1} & = \big \{ (x,y,z) \with {\textstyle \frac 1{i+1} < x
		< \frac 1i}, y<0, (x,y) \in A, g(x,y) = z \big \}, \\
		P_{i,2} & = \big \{ (x,y,z) \with {\textstyle \frac 1{i+1} < x
		< \frac 1i}, y>0, (x,-y) \in A, g(x,-y) = z \big \}, \\
		P_{i,3} & = \big \{ (x,y,z) \with {\textstyle \frac 1{i+1} < x
		< \frac 1i}, y<0, (x,y) \in A, -g(x,y) = z \big \}, \\
		P_{i,4} & = \big \{ (x,y,z) \with {\textstyle \frac 1{i+1} < x
		< \frac 1i}, y>0, (x,-y) \in A, -g(x,-y) = z \big \},
	\end{align*}
	for every positive integer $i$ and $Q_j = \bigcup_{i=1}^\infty
	P_{i,j}$ for $j \in \{ 1,2,3,4 \}$ and abbreviate 
	\begin{alignat*}{2}
		N_1 & = B_1 \cap \{ (x,y,z) \with y>0, z>0 \}, & \quad N_5 & =
		\mathbf R^3 \cap \{ (x,y,z) \with y>r/2,z=0 \}, \\
		N_2 & = B_2 \cap \{ (x,y,z) \with y<0, z>0 \}, & \quad N_6 & =
		\{ (x,y,z) \with -f(x) < y < r/2, z = 0 \}, \\
		N_3 & = B_3 \cap \{ (x,y,z) \with y>0, z<0 \}, & \quad N_7 & =
		\{ (x,y,z) \with -r/2 < y < f(x), z = 0 \}, \\
		N_4 & = B_4 \cap \{ (x,y,z) \with y<0, z<0 \}, & \quad N_8 & =
		\mathbf R^3 \cap \{ (x,y,z) \with y<-r/2,z=0 \};
	\end{alignat*}
	hence, $R$ is a two-dimensional submanifold of class $\infty$ of
	$\mathbf R^3$, the family $\Phi$ of its connected components consists
	of the sets $D_i$ and $P_{i,j}$ corresponding to positive integers $i$
	and $j \in \{ 1,2,3,4 \}$ together with the sets $N_1, \ldots, N_8$,
	the function $\boldsymbol \Uptheta^2 ( \| V \|, \cdot )|C$ is constant
	for $C \in \Phi$, we have $Q_j \subset B_j$ for $j \in \{ 1,2,3,4 \}$,
	\begin{alignat*}{2}
		B_1 \cap B_2 & = L_3, & \quad B_1 \cap B_3 & = L_1 \cup L_5
		\cup N_5 \\
		B_3 \cap B_4 & = L_4, & \quad B_2 \cap B_4 & = L_2 \cup L_6
		\cup N_8,
	\end{alignat*}
	and we let $X = (\bigcup \Phi ) \without K$.  Whenever $W$ is a member
	of a partition of $V$, we obtain
	\begin{equation*}
		\| W \| \restrict C = \| V \| \restrict C \quad \text{whenever
		$\| W \| ( C ) > 0$ and $C \in \Phi$}
	\end{equation*}
	by applying \ref{lemma:criterion_indecomposable}, for a suitable set
	$E$, with $m$, $n$, $U$, $M$, and $\{ c \}$ replaced by $2$, $3$,
	$\mathbf R^3$, $C$, and $\boldsymbol \Uptheta^2 ( \| V \|, \cdot )
	[C]$.  Thus, whenever $\Xi$ is a partition of $V$, we may partition
	$\Phi$ into nonempty subfamilies $Y_\Xi (W)$ corresponding to $W \in
	\Xi$ such that $W = V \restrict ( \bigcup Y_\Xi (W)) \times \mathbf
	G(3,2)$ for $W \in \Xi$.  We also record that
	\begin{equation*}
		\mathscr H^2 ( ( \Clos D_i) \without D_i ) = 0 \quad
		\text{and} \quad V \boundary \Clos D_i = 0
	\end{equation*}
	whenever $i$ is a positive integer, as $\mathbf n ( B_5, K, \cdot)$
	and $\mathbf n ( B_5, D_i, \cdot )$ agree on $\Clos D_i$ and applying
	\cite[4.4, 4.7]{MR0307015} with $M$ and $E$ replaced by $B_5$ and
	$D_i$ yields
	\begin{equation*}
		\updelta ( V \restrict D_i \times \mathbf G (3,2) ) ( \theta )
		= {\textstyle\int} \theta \bullet \mathbf n ( B_5, D_i, \cdot)
		\ud \mathscr H^1 \quad \text{for $\theta \in \mathscr D (
		\mathbf R^3, \mathbf R^3 )$}.
	\end{equation*}
	Accordingly, $\Clos D_1, \Clos D_2, \Clos D_3, \ldots$ form a $\| V \|
	+ \| \updelta V \|$ almost disjoint sequence of sets with vanishing
	distributional $V$ boundary whose union $\Clos K$ is $\| V \| + \|
	\updelta V \|$ equal to $\mathbf R^3 \without X$; in particular, we
	have $V \boundary X = 0$ and, defining
	\begin{gather*}
		Z_1 = V \restrict X \times \mathbf G (3,2), \quad Z_2 = V
		\restrict ( \Clos K ) \times \mathbf G (3,2), \\
		\text{and $\Pi = \{ V \restrict D_i \times \mathbf G (3,2)
		\with i = 1, 2, 3, \ldots \}$},
	\end{gather*}
	we see that $\{ Z_1, Z_2 \}$ and $\{ Z_1 \} \cup \Pi$ are partitions
	of $V$ by \ref{remark:char-partition}.  Employing
	\ref{lemma:little_more_general}, we also verify that $\Pi$ is the
	unique decomposition of $Z_2$.
	
	Since $V$ admits a decomposition by \cite[6.12]{MR3528825}, the proof
	may be concluded by showing that $Z_1$ belongs to every decomposition
	$\Xi$ of $V$ because in this case $\Xi \without \{ Z_1 \}$ is a
	decomposition of $Z_2$ by \cite[5.4, 6.7, 6.10]{MR3528825}.  For this
	purpose, computing $\updelta ( V \restrict C \times \mathbf G (3,2) )$
	whenever $C \in \Phi$, we shall verify:
	\begin{equation*}
		\begin{aligned}
			& \text{If $\| \updelta W \| ( L_1 ) = 0$ and $N_5 \in
			Y$, then $\{ N_1, N_3, N_6 \} \subset Y$}; \\
			& \text{if $\| \updelta W \| ( L_6 ) = 0$ and
			$\{N_2,N_4,N_7 \} \subset Y$, then $N_8 \in Y$}; \\
			& \text{if $\| \updelta W \| ( L_2 \cap L_5 ) = 0$ and
			$\{ N_1, N_3, N_6 \} \subset Y$, then $\{ N_2, N_4,
			N_7 \} \subset Y$;} \\
			& \text{if $\| \updelta W \| ( L_3 ) = 0$ and $N_1 \in
			Y$, then $Q_1 \subset {\textstyle \bigcup Y}$}; \\
			& \text{if $\| \updelta W \| ( L_3 ) = 0$ and $N_2 \in
			Y$, then $Q_2 \subset {\textstyle \bigcup Y}$}; \\
			& \text{if $\| \updelta W \| ( L_4 ) = 0$ and $N_3 \in
			Y$, then $Q_3 \subset {\textstyle\bigcup Y}$; and}, \\
			& \text{if $\| \updelta W \| ( L_4 ) = 0$ and $N_4 \in
			Y$, then $Q_4 \subset {\textstyle \bigcup Y}$}.
		\end{aligned}
	\end{equation*}
	whenever $Y \subset \Phi$ and $W = V \restrict ( \bigcup Y ) \times
	\mathbf G (3,2)$.  In fact, the first two implications are elementary;
	the third implication follows from the equations for $\mathbf n
	(M,E,\cdot)$; and the last four implications follow from the
	inclusions for $\mathbf n(M,F,\cdot)$.  Therefore, if $Y \subset \Phi$
	and $W = V \restrict ( \bigcup Y ) \times \mathbf G (3,2)$ is a member
	of a partition of $V$, then
	\begin{equation*}
		\text{either $X \subset {\textstyle \bigcup Y}$ or $X \cap
		{\textstyle\bigcup Y} = \varnothing$};
	\end{equation*}
	hence, there exists $W \in \Xi$ with $X \subset \bigcup Y_\Xi (W)$.
	Noting $\| \updelta V \| \restrict X$ is absolutely continuous with
	respect to $\| V \|$, we apply \ref{lemma:little_more_general}, for a
	suitable set $E$, with $m$, $n$, $U$, and $F$ replaced by $2$, $3$,
	$\mathbf R^3$, and $X$ to conclude $W \boundary X = 0$, whence we
	infer $\bigcup Y_\Xi (W) \subset X$.  Thus, $Z_1 \in \Xi$.
\end{proof}

\section{Properties of indecomposability with respect to a family of
generalised weakly differentiable real valued functions}
\label{section:indecomposability}

We introduce the main new concept of the present paper---indecomposability of
type $\Psi$---in
\ref{definition:indecomposability}--\ref{example:indecomposability-type-f}.
Its basic relations to topological connectedness are given in
\ref{thm:indecomposability_connected}--\ref{example:indecomposability} before
establishing first geometric consequences in
\ref{lemma:image-TVY}--\ref{remark:utility}.  The latter includes Theorem
\ref{Thm:consequences-indecomposability-singleton} of the introductory section
in \ref{lemma:basic_indecomp}\,\eqref{item:basic_indecomp:interval}%
\,\eqref{item:basic_indecomp:constant} and \ref{thm:utility}.

\begin{definition} \label{definition:indecomposability}
	Suppose $m$ and $n$ are positive integers, $m \leq n$, $U$~is an open
	subset of~$\mathbf{R}^n$, $V \in \mathbf{V}_m(U)$, $\| \updelta V\|$
	is a Radon measure, and $\Psi \subset \mathbf{T}(V)$. 
	
	Then, $V$ is called \emph{indecomposable of type~$\Psi$} if and only
	if, whenever $f \in \Psi$, the set of~$y \in \mathbf R$, such that
	$E(y) = \{x \with f(x)>y\}$ satisfies
	\begin{equation*}
		\|V\|( E(y) )>0, \quad \|V\|( U \without E(y) )>0, \quad V
		\boundary E(y) = 0,
	\end{equation*}
	has $\mathscr L^1$~measure zero.
\end{definition}

\begin{remark} \label{remark:indecomposability_implications}
	If $V$ is indecomposable, then $V$ is also indecomposable of type
	$\Psi$ whenever $\Psi \subset \mathbf T (V)$.  For $\Psi = \mathbf T (
	V)$, the converse implication holds; in fact, \cite[4.14]{MR3777387}
	readily yields that, whenever $E$ is a $\| V \| + \| \updelta
	V\|$~measurable set satisfying $V \boundary E = 0$, its characteristic
	function belongs to~$\mathbf T (V)$.
\end{remark}

\begin{remark} \label{remark:indecomposability_compactness}
	If $\spt \| V \|$ is compact, then indecomposability of
	types~$\mathscr E ( U, \mathbf R )$ and~$\mathscr D ( U, \mathbf R )$
	agree.  In general, these concepts differ as will be shown
	in~\ref{example:indecomposability}.
\end{remark}

\begin{remark} \label{remark:indecomposability_ball}
	If $V$ is indecomposable of type~$\mathscr D (U,\mathbf R )$, $a \in
	U$, $0 < r < \infty$, $\mathbf B (a,r) \subset U$, and $f : U \to
	\mathbf R$ satisfies $f(x) = \sup \{ r-|x-a|,0 \}$ for~$x \in U$, then
	$V$~is indecomposable of type~$\{ f \}$, as may be verified by
	approximation.
\end{remark}

\begin{remark} \label{remark:local_indecomposability}
	If $V$~is indecomposable of type~$\mathscr D ( U, \mathbf R )$ and
	$A$~is a relatively closed subset of~$U$, then $V | \mathbf 2^{(U
	\without A) \times \mathbf G (n,m)}$~is indecomposable of
	type~$\mathscr{D} ( U \without A, \mathbf R )$, as may be verified
	using the canonical extension map of~$\mathscr D ( U \without A,
	\mathbf R )$ into~$\mathscr D ( U, \mathbf R )$.
\end{remark}

\begin{example} \label{example:indecomposability-type-f}
	Taking $V$ and $\mathbf q$ as in
	\ref{example:indecomposability_of_type_f} and defining $E(y) = \{ x
	\with \mathbf q (x)>y \}$ whenever $y \in \mathbf R$, the set $\{ y
	\with \| V \| ( E(y) )> 0, \| V \| ( U \without E(y)) > 0, V \boundary
	E(y) = 0 \}$ is countably infinite and $V$ is indecomposable of type
	$\{ \mathbf q \}$.
\end{example}

\begin{theorem} \label{thm:indecomposability_connected}
	Suppose $m$ and $n$ are positive integers, $m \leq n$, $U$~is an open
	subset of~$\mathbf{R}^n$, $V \in \mathbf{V}_m(U)$, $\| \updelta V\|$
	is a Radon measure, and $V$~is indecomposable of type~$\mathscr E ( U,
	\mathbf R )$.
	
	Then, $\spt \| V \|$ is connected.
\end{theorem}

\begin{proof}
	If $\spt \| V \|$ were not connected, then there would exist nonempty
	disjoint relatively closed subsets $E_0$ and $E_1$ of~$U$ with $\spt
	\| V \| = E_0 \cup E_1$, and \cite[2.16]{DOI_10.4171_RMI_1487} would
	yield $f$ satisfying $\spt \| V \| \cap \{ x \with f(x) > y \} = E_1$
	for $0 \leq y < 1$, in contradiction to $\| V \| (E_1) > 0$, $\| V \|
	( U \without E_1 ) > 0$, and $V \boundary E_1 = 0$ by
	\cite[6.5]{MR3528825}.
\end{proof}

\begin{remark}
	In view of~\ref{remark:indecomposability_implications}, the preceding
	theorem extends~\cite[6.5]{MR3528825}.
\end{remark}

\begin{example} \label{example:indecomposability}
	Suppose $m$ and $n$ are integers, $1 \leq m \leq n$,
	$U$~is an open subset of~$\mathbf R^n$, $M$~is an
	$m$~dimensional submanifold-with-boundary of class
	$2$ of $U$,
	the inclusion map $F : M \to U$ is proper, $V$ is associated with
	$(F,U)$, and $\Phi$ is the family of connected components of~$M$.
	Then, $\| \updelta V \|$ is a Radon measure
	by~\ref{remark:immersion_varifold} and one verifies the equivalence of
	the following four conditions using
	\ref{lemma:criterion_indecomposable} and
	\ref{thm:indecomposability_connected}:
	\begin{enumerate}
		\item \label{item:indecomposability:connected} The
		submanifold-with-boundary~$M$ is connected.
		\item The submanifold~$M \without \partial M$ is connected.
		\item \label{item:indecomposability:dito} The varifold $V$ is
		indecomposable.
		\item \label{item:indecomposability:type-E} The varifold $V$
		is indecomposable of type $\mathscr E (U,\mathbf R)$.
		\setcounter{enumi_memory}{\value{enumi}}
	\end{enumerate}
	Hence, \emph{$V \restrict C \times \mathbf G (n,m)$ is indecomposable
	and $V \boundary C = 0$ for $C \in \Phi$}, as $C$~is relatively open
	in~$M$.  Next, we notice that \cite[5.2]{MR3528825} may be used to
	obtain \emph{that
	\begin{equation*}
		\begin{aligned}
			V \boundary E ( \theta ) & = \sum_{C \in \Phi} ( V
			\restrict C \times \mathbf G (n,m) ) \boundary E (
			\theta ), \\
			\| V \boundary E \| ( f ) & = \sum_{C \in \Phi} \| ( V
			\restrict C \times \mathbf G (n,m) ) \boundary E \| (
			f )
		\end{aligned}
	\end{equation*}
	whenever $E$~is $\| V \| + \| \updelta V \|$~measurable, $\theta \in
	\mathscr D ( U, \mathbf R^n )$, and $f \in \mathscr K ( U )$.} These
	equations are readily used to verify that \emph{$\{ V \restrict C
	\times \mathbf G (n,m) \with C \in \Phi \}$ is the unique
	decomposition of~$V$} and that the following two conditions are
	equivalent:
	\begin{enumerate}
		\setcounter{enumi}{\value{enumi_memory}}
		\item If $C$ is compact for some $C \in \Phi$, then $M$ is
		connected.
		\item The varifold~$V$ is indecomposable of type~$\mathscr D (
		U, \mathbf R )$.
	\end{enumerate}
\end{example}

\begin{lemma} \label{lemma:image-TVY}
	Suppose $m$ and $n$ are positive integers, $m \leq n$, $U$ is an open
	subset of $\mathbf R^n$, $V \in \mathbf V_m ( U )$, $\| \updelta V \|$
	is a Radon measure, $Y$ is a finite dimensional normed space, and $f
	\in \mathbf T ( V,Y )$.
	
	Then, there holds
	\begin{equation*}
		f(x) \in \spt f_\# \| V \| \quad \text{for $\| V \| + \|
		\updelta V \|$ almost all $x$}.
	\end{equation*}
\end{lemma}

\begin{proof}
	Noting \cite[4.11]{MR3777387}, we see that $\gamma \circ f \in \mathbf
	T (V)$ by \cite[8.15]{MR3528825} and obtain
	\begin{equation*}
		f(x) \in \{ y \with \gamma(y) = 0 \} \quad \text{for $\|
		\updelta V \|$ almost all $x$},
	\end{equation*}
	since we have $\| \updelta V \|_{(\infty)} ( \gamma \circ f ) \leq \|
	V \|_{(\infty)} ( \gamma \circ f ) = 0$ by \cite[8.33]{MR3528825},
	whenever $\gamma \in \mathscr D ( Y, \mathbf R )$ satisfies $\spt
	\gamma \cap \spt f_\# \| V \| = \varnothing$.
\end{proof}

\begin{lemma} \label{lemma:basic_indecomp}
	Suppose $m$ and $n$ are positive integers, $m \leq n$, $U$~is an open
	subset of~$\mathbf R^n$, $V \in \mathbf V_m ( U )$, $\| \updelta V \|$
	is a Radon measure, $f \in \mathbf T (V)$, $V$ is indecomposable of
	type~$\{ f \}$, and $J = \spt f_\# \| V \|$.
	
	Then, the following four statements hold.
	\begin{enumerate}
		\item \label{item:basic_indecomp:interval} The set $J$ is a
		subinterval of $\mathbf R$.
		\item \label{item:basic_indecomp:bdry} For $\mathscr L^1$
		almost all $y \in J$, we have $V \boundary \{ x \with f (x)>y
		\} \neq 0$.
		\item \label{item:basic_indecomp:abs} If $Y \subset \mathbf R$
		and $V \weakD f (x) = 0$ for $\| V \|$~almost all $x \in
		f^{-1} [Y]$, then we have $\mathscr L^1 ( Y \cap J ) = 0$.
		\item \label{item:basic_indecomp:constant} If $V \weakD f =
		0$, then $f$ is $\| V \| + \| \updelta V \|$~almost constant.
	\end{enumerate}
\end{lemma}

\begin{proof}
	Abbreviate $I = \mathbf R \cap \{ y \with \inf J \leq y \leq \sup J
	\}$, choose compact sets $K_i$ with $K_i \subset \Int K_{i+1}$ and
	$\bigcup_{i=1}^\infty K_i = U$, and pick $\varepsilon_i > 0$ such that
	(see \cite[2.8, 2.11]{DOI_10.4171_RMI_1487})
	\begin{equation*}
		\nu = \sum_{i=1}^\infty \varepsilon_i f_\# \big ( ( \| V \|
		\restrict K_i \cap \{ x \with |f(x)| \leq i \} ) \restrict | V
		\weakD f | \big )
	\end{equation*}
	satisfies $\nu ( \mathbf R ) < \infty$. Define $B = \mathbf R \cap \{
	y \with V \boundary \{ x \with f (x)>y \} = 0 \}$.  Clearly, we have
	$\mathscr L^1 ( B \cap I ) = 0$; in particular
	\eqref{item:basic_indecomp:bdry} holds.  \cite[4.11]{MR3777387} and
	\cite[8.29]{MR3528825} yield
	\begin{equation*}
		\{ y \with \boldsymbol \Uptheta^1 ( \nu, y ) = 0 \} \subset B.
	\end{equation*}
	If $Y$ satisfies the hypotheses of~\eqref{item:basic_indecomp:abs},
	then we have $\nu (Y)=0$ which entails firstly $\mathscr L^1 ( Y
	\without \{ y \with \boldsymbol \Uptheta^1 ( \nu, y ) = 0 \} ) = 0$ by
	\cite[2.10.19\,(4)]{MR41:1976}, and then $\mathscr L^1 ( Y \cap I ) =
	0$.  Now, \eqref{item:basic_indecomp:interval} and
	\eqref{item:basic_indecomp:abs} follow; in particular $I = J$.  Under
	the hypothesis of~\eqref{item:basic_indecomp:constant}, we have $\nu =
	0$ which implies $B = \mathbf R$, $\mathscr L^1 (I) = 0$, and that
	$f$~is $\| V \|$~almost constant; thence, \ref{lemma:image-TVY}
	implies \eqref{item:basic_indecomp:constant}.
\end{proof}

\begin{theorem} \label{thm:utility}
	Suppose $m$ and $n$ are positive integers, $m \leq n$, $U$ is an open
	subset of $\mathbf R^n$, $V \in \mathbf V_m (U)$, $\| \updelta V \|$
	is a Radon measure, $f \in \mathbf T(V)$, $V$ is indecomposable of
	type $\{ f \}$, $Y \subset \mathbf R$, and $f(x) \in Y$ for $\| V \|$
	almost all $x$.
	
	Then, there holds
	\begin{equation*}
		\diam \spt f_\# \| V \| \leq \mathscr L^1 (Y).
	\end{equation*}
\end{theorem}

\begin{proof}
	From \ref{lemma:basic_indecomp}\,\eqref{item:basic_indecomp:interval}
	and \ref{lemma:basic_indecomp}\,\eqref{item:basic_indecomp:abs}
	applied with $Y$ replaced by $\mathbf R \without Y$, we conclude that
	$\spt f_\# \| V \|$ is an interval which is $\mathscr L^1$ almost
	contained in $Y$.
\end{proof}

\begin{remark} \label{remark:utility}
	The utility of an estimate of $\diam \spt f_\# \| V \|$ is illustrated
	by the following two facts concerning Radon measures $\phi$ over $U$,
	see \cite[2.11]{DOI_10.4171_RMI_1487}.
	\begin{enumerate}
		\item \label{item:utility:ess-sup} Whenever $g$ is a
		nonnegative real valued $\phi$ measurable function satisfying
		$\phi \, \{ x \with g(x) \leq y \} > 0$ for $y>0$, we have
		$\phi_{(\infty)} (g) = \diam \spt g_\# \phi$.
		\item \label{item:utility:continuous} If $g : \spt \phi \to
		\mathbf R$ is continuous, then $\im g$ is a dense subset of
		$\spt g_\# \phi$.
	\end{enumerate}
	In the final paper of our series \cite{arXiv:1709.05504v3}, for
	certain $f$, such a set $Y$ satisfying a geometric estimate for
	$\mathscr L^1 (Y)$ will be constructed to obtain a novel type of
	Sobolev Poincaré inequality.
\end{remark}

\section{Unique partition along a generalised weakly
differentiable real valued function}

Following some preparations in
\ref{lemma:restriction}--\ref{remark:redundancy-derivative-agreement}, we
introduce the concept of partition along $f$ in
\ref{definition:partition-along}--\ref{remark:partition-along-pi}. Its
uniqueness and existence properties are proven in
\ref{thm:uniqueness-partition-along}--\ref{remark:constancy-theorem}; this
includes Theorem \ref{Thm:partition-along-fct} of the introductory section in
\ref{thm:uniqueness-partition-along} and
\ref{thm:decomposition_adapted_to_fct}.

\begin{lemma} \label{lemma:restriction}
	Suppose $m$ and $n$ are positive integers, $m \leq n$, $U$ is an open
	subset of $\mathbf R^n$, $V \in \mathbf V_m ( U )$, $\| \updelta V \|$
	is a Radon measure, $f \in \mathbf T ( V )$, $y \in \mathbf R$, and
	\begin{equation*}
		E = \{ x \with f(x) > y \}, \quad V \boundary E = 0, \quad W =
		V \restrict E \times \mathbf G (n,m).
	\end{equation*}
	
	Then, there holds $f \in \mathbf T (W)$ and
	\begin{equation*}
		W \weakD f(x) = V \weakD f(x) \quad \text{for $\| W \|$ almost
		all $x$}.
	\end{equation*}
\end{lemma}

\begin{proof}
	Employing \cite[4.11]{MR3777387} and \cite[8.12,
	8.13\,(4)]{MR3528825}, we define the function $g = \sup \{ f,y \} \in
	\mathbf T (V)$ and notice
	\begin{align*}
		V \weakD g(x) & = V \weakD f(x) \quad \text{for $\| V \|$
		almost all $x \in E$} \\
		V \weakD g(x) & = 0 \quad \text{for $\| V ||$ almost all $x
		\in U \without E$}.
	\end{align*}
	Splitting the left integrals along the partition $\{ E, U \without E
	\}$, we obtain
	\begin{gather*}
		( \updelta V ) ( ( \gamma \circ g ) \theta ) = ( \updelta W )
		( ( \gamma \circ f ) \theta ) + \gamma(y) \big ( (\updelta V)
		\restrict ( U \without E ) \big ) ( \theta ), \\
		\begin{aligned}
			& {\textstyle\int} \gamma ( g(x)) S_\natural \bullet
			\Der \theta (x) \ud V \, (x,S) \\
			& \qquad = {\textstyle\int} \gamma ( f(x) ) S_\natural
			\bullet \Der \theta (x) \ud W \, (x,S) + \gamma(y)
			\updelta (  V \restrict ( U \without E ) \times
			\mathbf G (n,m)) ( \theta ),
		\end{aligned} \\
		{\textstyle\int} \langle \theta (x), \Der \gamma ( g(x) )
		\circ V \weakD g (x) \rangle \ud \| V \| \, x =
		{\textstyle\int} \langle \theta (x), \Der \gamma (f(x)) \circ
		V \weakD f(x) \rangle \ud \| W \| \, x
	\end{gather*}
	whenever $\theta \in \mathscr D ( U, \mathbf R^n )$, $\gamma \in
	\mathscr E ( \mathbf R, \mathbf R )$, and $\spt \Der \gamma$ is
	compact.  Subtracting the last two equations from the first, the
	conclusion follows, as $V \boundary ( U \without E ) = 0$.
\end{proof}

\begin{remark}
	In view of \ref{example:no-adapted-decomposition} or
	\cite[8.25]{MR3528825}, the condition on $E$ to equal $\{ x \with f(x)
	> y \}$ for some $y \in \mathbf R$ may not weakened to $\| V \| + \|
	\updelta V \|$ measurability.
\end{remark}

\begin{lemma} \label{lemma:family_of_pieces}
	Suppose $m$ and $n$ are positive integers, $m \leq n$, $U$ is an open
	subset of $\mathbf R^n$, $V \in \mathbf V_m ( U )$, $\| \updelta V \|$
	is a Radon measure, $f \in \mathbf T (V)$, and $R$ is the family of
	all $\| V \| + \| \updelta V \|$ measurable sets $E$ such that $V
	\boundary E = 0$ and such that $W = V \restrict E \times \mathbf G
	(n,m)$ satisfies $f \in \mathbf T (W)$ and
	\begin{equation*}
		W \weakD f(x) = V \weakD f(x) \quad \text{for $\| W \|$ almost
		all $x$}.
	\end{equation*}
	
	Then, the following four statements hold:
	\begin{enumerate}
		\item The set $U$ belongs to $R$.
		\item If $E \in R$, $F \in R$, and $F \subset E$, then $E
		\without F \in R$.
		\item If $E_1, E_2, E_3, \ldots$ form a disjoint sequence in
		$R$, then $\bigcup_{i=1}^\infty E_i \in R$.
		\item If $E_1 \supset E_2 \supset E_3 \supset \cdots$ form a
		sequence in $R$, then $\bigcap_{i=1}^\infty E_i \in R$.
	\end{enumerate}
	In particular, if $I$ is a subinterval of $\mathbf R$, $E = f^{-1}
	[I]$, and $V \boundary E = 0$, then $E \in R$.
\end{lemma}

\begin{proof}
	Recalling \cite[5.3]{MR3528825}, we verify the principal conclusion
	using that $E \in R$ if and only if $E$ is $\| V \| + \| \updelta V
	\|$ measurable, $V \boundary E = 0$, and
	\begin{align*}
		& ((\updelta V) \restrict E ) ( ( \gamma \circ f ) \theta ) \\
		& \quad = {\textstyle\int_{E \times \mathbf G (n,m)} ( \gamma
		\circ f ) (x) S_\natural \bullet \Der \theta (x) + \langle
		\theta (x), \Der \gamma (f(x)) \circ V \weakD f(x) \rangle \ud
		V \, (x,S)}
	\end{align*}
	whenever $\theta \in \mathscr D ( U, \mathbf R^n )$, $\gamma \in
	\mathscr E ( \mathbf R, \mathbf R )$, and $\spt \Der \gamma$ is
	compact.  Finally, the postscript follows from the principal
	conclusion and \ref{lemma:restriction}.
\end{proof}

\begin{remark}
	The family $R$ need not to be a Borel family; in fact, it may happen
	that $E \in R$ and $F \in R$ but neither $E \cup F \in R$ nor $E \cap
	F \in R$ as is evident from considering the varifold constructed in
	\cite[6.13]{MR3528825} and taking $f=0$.
\end{remark}

\begin{remark} \label{remark:redundancy-derivative-agreement}
	If $\boldsymbol \Uptheta^m ( \| V \|, x ) \geq 1$ for $ \| V \|$
	almost all $x$, then the condition that ``$W \weakD f (x) = V \weakD
	f(x)$ for $\| W \|$ almost all $x$'' in the definition of $R$ is
	redundant, as may be verified by means of \cite[11.2]{MR3528825} in
	conjunction with \cite[2.8.18, 2.9.11]{MR41:1976},
	\cite[3.5\,(1b)]{MR0307015}, and \cite[2.26]{DOI_10.4171_RMI_1487}.
\end{remark}

\begin{definition} \label{definition:partition-along}
	Suppose $m$ and $n$ are positive integers, $m \leq n$, $U$ is an open
	subset of $\mathbf R^n$, $V \in \mathbf V_m ( U )$, $\| \updelta V \|$
	is a Radon measure, and $f \in \mathbf T ( V )$,
	
	Then, $\Pi$ is called a \emph{partition of $V$ along $f$} if and only
	if $\Pi$ is a partition of $V$ and, whenever $W \in \Pi$, the
	following two conditions hold.
	\begin{enumerate}
		\item \label{item:partition-along:representation} There exists
		a subinterval $I$ of $\mathbf R$ such that $W = V \restrict
		f^{-1}[I] \times \mathbf G (n,m)$ and $V \boundary f^{-1} [I]
		= 0$.
		\item \label{item:partition-along:indecomposable} There exists
		no partition of $\mathbf R$ into subintervals $J_1$ and $J_2$
		such that
		\begin{equation*}
			\text{$\| W \| ( f^{-1} [J_i] ) > 0$ for $i \in \{ 1,2
			\}$}, \quad W \boundary f^{-1} [J_1] = 0.
		\end{equation*}
	\end{enumerate}
\end{definition}

\begin{remark} \label{remark:partition-along}
	For $W \in \Pi$, we see $\| \updelta W \|$ is a Radon measure by
	\eqref{item:partition-along:representation}, $f \in \mathbf T ( W )$
	with
	\begin{equation*}
		W \weakD f(x) = V \weakD f(x) \quad \text{for $\| W \|$ almost
		all $x$}
	\end{equation*}
	by \ref{lemma:family_of_pieces}, $W$ is indecomposable of type $\{ f
	\}$ by \eqref{item:partition-along:indecomposable}, $J = \spt f_\# \|
	V \|$ is a subinterval of $\mathbf R$ by
	\ref{lemma:basic_indecomp}\,\eqref{item:basic_indecomp:interval}, and
	$W = V \restrict f^{-1} [I] \times \mathbf G (n,m)$ for some
	subinterval $I$ of $J$ with $V \boundary f^{-1} [I] = 0$ by
	\ref{lemma:image-TVY} and \eqref{item:partition-along:representation};
	in this case, $I$ is dense in $J$, whence we infer
	\begin{equation*}
		\Int J \subset I \subset J, \quad \Bdry I = \Bdry J.
	\end{equation*}
\end{remark}

\begin{remark} \label{remark:partition-along-pi}
	By \ref{lemma:zero_distributional_boundary} and
	\ref{remark:partition-along}, there exists a function $\pi : \Pi \to
	\mathbf 2^{\mathbf R}$ whose value at $W \in \Pi$ is characterised to
	be the dense subinterval $\pi(W)$ of $\spt f_\# \| W \|$ such that $W
	= V \restrict f^{-1} [ \pi (W)] \times \mathbf G (n,m)$ with $V
	\boundary f^{-1} [ \pi(W) ] = 0$ and
	\begin{equation*}
		b \in \pi(W) \quad \text{if and only if} \quad ( \| W \| + \|
		\updelta W \| ) \, \{ x \with f(x) = b \} > 0
	\end{equation*}
	whenever $b \in \Bdry \pi(W)$; hence, $\| \updelta W \| = \| \updelta
	V \| \restrict f^{-1} [ \pi(W) ]$ for $W \in \Pi$.  This entails that
	$\pi$ maps distinct members of $\Pi$ to disjoint subintervals of
	$\mathbf R$ and
	\begin{equation*}
		{\textstyle ( \| V \| + \| \updelta V \| ) \big ( U \without
		f^{-1} [ \bigcup \im \pi ] \big ) = 0}
	\end{equation*}
	because $\card \{ W \with x \in f^{-1} [\pi (W)] \} = 1$ for $\| V \|
	+ \| \updelta V \|$ almost all $x$ by \ref{remark:partition}.
\end{remark}

\begin{example} \label{example:indecomposability-along}
	Taking $V$ and $\mathbf q$ as in
	\ref{example:indecomposability_of_type_f}, we see that $V$ admits a
	nontrivial partition along $\mathbf q$ but is nevertheless
	indecomposable of type $\{ \mathbf q \}$.
\end{example}

\begin{theorem} \label{thm:uniqueness-partition-along}
	Suppose $m$ and $n$ are positive integers, $m \leq n$, $U$ is an open
	subset of $\mathbf R^n$, $V \in \mathbf V_m (U)$, $\| \updelta V \|$
	is a Radon measure, and $f \in \mathbf T (V)$.
	
	Then, there exists at most one partition of $V$ along $f$.
\end{theorem}

\begin{proof}
	We suppose $\Pi$ and $P$ are partitions of $V$ along $f$ and associate
	functions $\pi$ and $\rho$ to $\Pi$ and $P$, respectively, as in
	\ref{remark:partition-along-pi}.  Whenever $W \in \Pi$, there exists
	$X \in P$ satisfying
	\begin{equation*}
		f_\# \| V \| ( \pi ( W ) \cap \rho (X) ) > 0
	\end{equation*}
	by \ref{remark:partition}.  Accordingly, the proof may conducted by
	showing that, whenever $W \in \Pi$ and $X \in P$ are related by the
	preceding condition, there holds
	\begin{equation*}
		f_\# \| V \| ( \rho (X) \without \pi (W) ) = 0.
	\end{equation*}
	Clearly, we have $\pi (W ) \cap \rho (X) \neq \varnothing$.  Moreover,
	whenever $J$ is a subinterval of $\mathbf R$, $\mathbf R \without J$
	is a subinterval of $\mathbf R$, every member of $( \im \pi \cup \im
	\rho ) \without \{ \rho(X) \}$ is contained either in $J$ or in
	$\mathbf R \without J$, and $\pi(W) \subset \mathbf R \without J$, we
	apply \ref{remark:partition} twice to conclude
	\begin{equation*}
		X \boundary f^{-1} [ J ] = V \boundary f^{-1} [J] = 0
	\end{equation*}
	so that $\| X \| \big ( f^{-1} [ \mathbf R \without J ] \big ) \geq
	f_\# \| V \| ( \pi (W) \cap \rho (X) ) > 0$ implies
	\begin{equation*}
		f_\# \| V \| ( \rho (X) \cap J ) = \| X \| \big ( f^{-1} [ J ]
		\big ) = 0.
	\end{equation*}
	This allows us to deduce
	\begin{equation*}
		f_\# \| V \| \big ( \rho (X) \cap \{ y \with \sup \pi (W) \leq
		y \} \without \pi(W) \big ) = 0;
	\end{equation*}
	in fact, assuming $b = \sup \pi (W) \in \rho (X)$, we take $J = \{ y
	\with b \leq y < \infty \} \without \pi(W)$.  Similarly, we obtain
	\begin{equation*}
		f_\# \| V \| \big ( \rho (X) \cap \{ y \with y \leq \inf \pi
		(W) \} \without \pi(W) \big ) = 0
	\end{equation*}
	by assuming $b = \inf \pi(W) \in \rho (X)$ and taking $J = \{ y \with
	- \infty < y \leq b \} \without \pi(W)$.
\end{proof}

\begin{example} \label{example:non-existence-partition-along}
	Suppose $m$ and $n$ are positive integers, $m < n$, $T \in \mathbf
	G(n,m)$, $V = \mathscr L^n \times \boldsymbol{\updelta}_T \in \mathbf
	V_m ( \mathbf R^n )$, $f : \mathbf R^n \to \mathbf R$ is a nonzero
	linear map, and $T \subset \ker f$.  Then, $\updelta V = 0$ by
	\cite[4.8\,(2)]{MR0307015} and there exists no partition of $V$ along
	$f$; in fact, whenever $I$ is a subinterval of $\mathbf R$, we verify
	$V \boundary f^{-1} [ I ] = 0$ by means of
	\cite[4.12\,(1)]{MR3777387}, \ref{lemma:restriction}, and
	\ref{lemma:family_of_pieces}, whence, taking $W = V \restrict f^{-1} [
	I ] \times \mathbf G (n,m)$, we infer $W \boundary f^{-1} [ J ] = 0$
	whenever $J$ is a subinterval of $\mathbf R$ by
	\ref{lemma:little_more_general}.
\end{example}

\begin{lemma} \label{lemma:criterion-partition-along}
	Suppose $m$ and $n$ are positive integers with $m \leq n$, $U$ is an
	open subset of $\mathbf R^n$, $V \in \mathbf V_m ( U )$, $\| \updelta
	V \|$ is a Radon measure, $f \in \mathbf T (V)$, $P$ is the family of
	all subintervals $I$ of $\mathbf R$ such that $V \boundary f^{-1} [I]
	= 0$ and $\| V \| ( f^{-1} [I] ) > 0$, and $C$ is a countable
	disjointed subfamily of $P$ satisfying
	\begin{equation*}
		{\textstyle \| V \| \big ( U \without f^{-1} [ \bigcup C ]
		\big ) = 0}
	\end{equation*}
	such that, for $I \in C$, there exists no $J$ with
	\begin{equation*}
		J \subset I, \quad J \in P, \quad I \without J \in P.
	\end{equation*}
	
	Then, $\big \{ V \restrict f^{-1} [I] \times \mathbf G (n,m) \with I
	\in C \big \}$ is a partition of $V$ along $f$.
\end{lemma}

\begin{proof}
	Abbreviating $W_I = V \restrict f^{-1} [I] \times \mathbf G(n,m)$, we
	define $\Pi = \{ W_I \with I \in C \}$.  Firstly, we see that $\Pi$ is
	a partition of $V$ because
	\begin{equation*}
		\| \updelta V \| (f) \leq \sum_{I \in C} \| \updelta W_I \|
		(f) = \big ( \| \updelta V \| \restrict {\textstyle \bigcup C}
		\big ) (f) \quad \text{for $0 \leq f \in \mathscr K(U)$}
	\end{equation*}
	by \ref{lemma:zero_distributional_boundary}.  Then, in view of
	\cite[5.3]{MR3528825}, it suffices to note $V \boundary f^{-1} [ I
	\cap J ] = 0$ whenever $I \in C$, $J$ is a subinterval of $\mathbf R$,
	and $W_I \boundary f^{-1} [J] = 0$ by \ref{lemma:little_more_general}.
\end{proof}

\begin{theorem} \label{thm:decomposition_adapted_to_fct}
	Suppose $m$ and $n$ are positive integers, $m \leq n$, $U$ is an open
	subset of $\mathbf R^n$, $V \in \mathbf{RV}_m ( U )$, $\| \updelta V
	\|$ is a Radon measure, and $f \in \mathbf T (V)$.
	
	Then, there exists a partition of $V$ along $f$.
\end{theorem}

\begin{proof}
	Assuming $V \neq 0$, we will verify the conditions of
	\ref{lemma:criterion-partition-along}.
	
	First, for every positive integer $i$, we define
	\begin{equation*}
		\delta_i = \boldsymbol \upalpha (m) 2^{-m-1} i^{-1-2m}, \quad
		\epsilon_i = 2^{-1} i^{-2}
	\end{equation*}
	and let $A_i$ denote the Borel set of $a \in \mathbf R^n$ satisfying
	\begin{gather*}
		|a| \leq i, \quad \mathbf U ( a, 2\epsilon_i ) \subset U,
		\quad \boldsymbol \Uptheta^m ( \| V \|, a ) \geq 1/i, \\
		\| \updelta V \| \, \mathbf B (a,r) \leq \boldsymbol \upalpha
		(m) i r^m \quad \text{for $0 < r < \epsilon_i$}
	\end{gather*}
	Clearly, we have $A_i \subset A_{i+1}$ for every positive integer $i$
	and $\| V \| \big ( U \without \bigcup_{i=1}^\infty A_i \big ) = 0$ by
	\cite[3.5\,(1a)]{MR0307015} and \cite[2.8.18, 2.9.5]{MR41:1976}.
	Moreover, we define
	\begin{equation*}
		P_i = P \cap \big \{ I \with f_\# ( \| V \| \restrict A_i )
		(I) > 0 \big \}
	\end{equation*}
	and notice that $P_i \subset P_{i+1}$ for every positive integer $i$
	and $P = \bigcup_{i=1}^\infty P_i$.  Next, we observe the lower bound
	given by
	\begin{equation*}
		f_\# ( \| V \| \restrict \mathbf B(a,\epsilon_i) ) (I) \geq
		\delta_i
	\end{equation*}
	whenever $i$ is a positive integer, $I \in P$, $a \in A_i$, and
	$\boldsymbol \Uptheta^{\ast m} ( \| V \| \restrict f^{-1} [I], a )
	\geq 1/i$; in fact, noting
	\begin{equation*}
		{\textstyle\int_0^{\epsilon_i}} r^{-m} \| \updelta ( V
		\restrict f^{-1} [I] \times \mathbf G(n,m) ) \| \, \mathbf B (
		a,r ) \ud \mathscr L^1 \, r \leq \boldsymbol \upalpha (m) i
		\epsilon_i,
	\end{equation*}
	the inequality follows from \cite[4.5, 4.6]{MR3528825}.  Let $Q_i$
	denote the set of $I \in P$ such that there is no $J$ satisfying
	\begin{equation*}
		J \subset I, \quad J \in P_i, \quad I \without J \in P_i.
	\end{equation*}
	
	We denote by $\Omega$ the class of all disjointed subfamilies $H$ of
	$P$ with $\bigcup H = \mathbf R$ and let $G_0 = \{ \mathbf R \} \in
	\Omega$.  The previously observed lower bound implies
	\begin{equation*}
		\delta_i \card ( H \cap P_i ) \leq \| V \| ( U \cap \{ x \with
		\dist (x,A_i) \leq \epsilon_i \} ) < \infty
	\end{equation*}
	whenever $H$ is a disjointed subfamily of $P$ and $i$ is a positive
	integer; in fact, for each $I \in H \cap P_i$, there exists $a \in
	A_i$ with
	\begin{equation*}
		\boldsymbol \Uptheta^m ( \| V \| \restrict f^{-1} [I],a ) =
		\boldsymbol \Uptheta^m ( \| V \|, a ) \geq 1/i
	\end{equation*}
	by \cite[2.8.18, 2.9.11]{MR41:1976}, whence we infer
	\begin{equation*}
		\| V \| \big ( f^{-1} [I] \cap \{ x \with \dist (x,A_i) \leq
		\epsilon_i \} \big ) (I) \geq f_\# ( \| V \| \restrict \mathbf
		B (a,\epsilon_i) ) (I) \geq \delta_i.
	\end{equation*}
	In particular, such $H$ is countable.
	
	Next, we inductively (for every positive integer $i$) define
	$\Omega_i$ to be the class of all $H \in \Omega$ such that every $E
	\in G_{i-1}$ is the union of some subfamily of $H$, and choose $G_i
	\in \Omega_i$ such that
	\begin{equation*}
		\card ( G_i \cap P_i ) \geq \card ( H \cap P_i ) \quad
		\text{whenever $H \in \Omega_i$}.
	\end{equation*}
	The maximality of $G_i$ implies $G_i \subset Q_i$; in fact, if there
	existed $I \in G_i \without Q_i$, there would exist $J$ satisfying
	\begin{equation*}
		J \subset I, \quad J \in P_i, \quad I \without J \in P_i,
	\end{equation*}
	and $H = ( G_i \without \{ I \} ) \cup \{ J, I \without J \}$ would
	belong to $\Omega_i$ with
	\begin{equation*}
		\card ( H \cap P_i ) > \card ( G_i \cap P_i ).
	\end{equation*}
	Moreover, it is evident that, to each $x \in \dmn f$, there
	corresponds a sequence $I_1, I_2, I_3, \ldots$ characterised by the
	conditions $f(x) \in I_i \in G_i$, hence $I_{i+1} \subset I_i$, for
	every positive integer $i$.
	
	We define $G = \bigcup_{i=1}^\infty G_i$ and notice that $G$ is
	countable.  We let $C$ denote the collection of sets
	$\bigcap_{i=1}^\infty I_i$ with positive $f_\# \| V \|$ measure
	corresponding to some sequence $I_1, I_2, I_3, \ldots$ with $I_{i+1}
	\subset I_i \in G_i$ for every positive integer $i$.  Clearly, $C$ is
	a disjointed subfamily of $P$, and hence $C$ is countable.  We will
	show that
	\begin{equation*}
		{\textstyle f_\# \| V \| \big ( \mathbf R \without \bigcup C
		\big ) = 0}.
	\end{equation*}
	In view of \cite[2.8.18, 2.9.11]{MR41:1976}, it is sufficient to prove
	\begin{align*}
		& A_i \cap ( \dmn f ) \without {\textstyle\bigcup} C \\
		& \qquad \subset {\textstyle \bigcup} \big \{ f^{-1} [I] \cap
		\{ x \with \boldsymbol \Uptheta^{\ast m} ( \| V \| \restrict
		f^{-1} [I], x ) < \boldsymbol \Uptheta^{\ast m} ( \| V \|, x )
		\} \with I \in C \big \}
	\end{align*}
	for every positive integer $i$.  For this purpose, we consider $a \in
	A_i \cap ( \dmn f ) \without \bigcup C$ with corresponding sequence
	$I_1, I_2, I_3, \ldots$ as above.  It follows that
	\begin{equation*}
		f_\# \| V \| \left ( \bigcap_{j=1}^\infty I_j \right ) = 0.
	\end{equation*}
	Therefore, there exists $j$ with $f_\# (\| V \| \restrict \mathbf B
	(a,\epsilon_i)) (I_j) < \delta_i$, and the lower bound implies
	\begin{equation*}
		\boldsymbol \Uptheta^{\ast m} ( \| V \| \restrict f^{-1}
		[I_j], a) < 1/i \leq \boldsymbol \Uptheta^m ( \| V \|, a ).
	\end{equation*}
	
	If, for some $I \in C$, there existed $J$ with
	\begin{equation*}
		J \subset I, \quad J \in P, \quad I \without J \in P,
	\end{equation*}
	then there would exist intervals $I_i$ with $I = \bigcap_{i=1}^\infty
	I_i$ and $I_i \in G_i$ for every positive integer $i$.  We could
	choose $i$ such that $J \in P_i$ and $I \without J \in P_i$.  Defining
	$E = f^{-1} [ I_i \without I ]$ and $W = V \restrict E \times \mathbf
	G (n,m)$, we would notice
	\begin{equation*}
		V \boundary E = 0, \quad f \in \mathbf T (W)
	\end{equation*}
	by \ref{lemma:family_of_pieces}.  As $J$ and $I \without J$ would be
	nonempty, $I$ would have nonempty interior.  Accordingly, picking an
	interval $L$ in $\mathbf R$ such that $\mathbf R \without L$ would
	also be an interval in $\mathbf R$ and
	\begin{equation*}
		J = I \cap L,
	\end{equation*}
	we could employ \cite[8.5, 8.30]{MR3528825} to deduce $W \boundary
	f^{-1} [ L ] = 0$, whence it would follow
	\begin{equation*}
		V \boundary f^{-1} [ I_i \cap L \without I ] = 0
	\end{equation*}
	from \ref{lemma:little_more_general} with $F = f^{-1} [L]$.  Finally,
	noting that $I_i \cap L$ would equal the disjoint union of $J$ and
	$I_i \cap L \without I$ and that $I_i \without L = I_i \without ( I_i
	\cap L )$, we could establish
	\begin{gather*}
		V \boundary f^{-1} [ I_i \cap L ] = 0, \quad V \boundary
		f^{-1} [ I_i \without L ] = 0, \\
		J \subset I_i \cap L \in P_i, \quad I \without J \subset I_i
		\without L \in P_i
	\end{gather*}
	by \cite[5.3]{MR3528825}, in contradiction to $I_i \in Q_i$.
\end{proof}

\begin{remark}
	The novelty of the preceding proof is its last paragraph; otherwise,
	it rests on the machinery developed in \cite[6.12]{MR3528825}.  A
	refinement of that machinery in a different direction appears in
	\cite[5.12]{MR4609162}.
\end{remark}

\begin{remark}
	The rectifiability hypothesis may not be omitted by
	\ref{example:non-existence-partition-along}.
\end{remark}

\begin{remark} \label{remark:constancy-theorem}
	In view of \ref{remark:partition} and
	\ref{remark:partition-vs-decomposition}, one may obtain the case $Y =
	\mathbf R$ of the constancy theorem \cite[8.34]{MR3528825} from
	\ref{thm:decomposition_adapted_to_fct} by means of
	\ref{lemma:basic_indecomp}\,\eqref{item:basic_indecomp:constant} and
	\ref{remark:partition-along}.
\end{remark}

\section{Criteria for local finiteness of decompositions}

A first criterion is obtained in
\ref{example:regularity-hypotheses}--\ref{thm:locally_finite_near_boundary};
in particular, Theorem \ref{Thm:criterion-local-finiteness} of the
introductory section is provided in \ref{thm:locally_finite_near_boundary}.  A
second criterion is derived in \ref{lemma:zero-trace-test}%
--\ref{corollary:Neumann-locally-finite-decompositions}: Firstly, the material
of \cite[pp.\,2614--2625]{MR4359920} is localised and streamlined in
\ref{lemma:zero-trace-test}--\ref{remark:comparison-DM21}; secondly, an
isoperimetric lower density ratio bound is established in
\ref{example:Neumann-Lm}--\ref{remark:Neumann-lower-density-ratio-bound};
and, finally, Theorem \ref{Thm:criterion-local-finiteness-II} of the
introductory section is provided in
\ref{corollary:Neumann-locally-finite-decompositions}.

\begin{example} \label{example:regularity-hypotheses}
	Suppose $m$ and $n$ are positive integers, $m \leq n$, $U$ is an open
	subset of $\mathbf R^n$, $B$ is an $m-1$ dimensional submanifold of
	class $2$ of $U$ satisfying $U \cap ( \Clos B ) \without B =
	\varnothing$, $V \in \mathbf V_m (U)$, $\beta = \infty$ if $m = 1$,
	$\beta = m/(m-1)$ if $m>1$,
	\begin{equation*}
		\sup \big \{ ( \updelta V ) ( \theta ) \with \theta \in
		\mathscr D ( U, \mathbf R^n ), \spt \theta \subset K \without
		B, \| V \|_{(\beta)} (\theta) \leq 1 \big \} < \infty
	\end{equation*}
	whenever $K$ is a compact subset of $U$, $\| V \| ( B ) = 0$, and
	\begin{equation*}
		\boldsymbol \Uptheta^m ( \| V \|, x ) \geq 1 \quad \text{for
		$\| V \|$ almost all $x$}.
	\end{equation*}
	Then, $V$ is rectifiable, $\| \updelta V \|$ is a Radon measure,
	$\mathbf h (V,\cdot) \in \mathbf L_m^{\textup{loc}} ( \| V \|, \mathbf
	R^n )$, and, if $m>1$, then $\spt ( \| \updelta V \| - \| \updelta V
	\|_{\| V \|} ) \subset B$; in fact, noting that $\| \updelta V \|
	\restrict ( U \without B )$ is a Radon measure and that
	\cite[3.1\,(2)]{MR0397520} remains valid for submanifolds of class
	$2$, we conclude that $\| \updelta V \|$ is a Radon measure, hence $V$
	is rectifiable by \cite[5.5\,(1)]{MR0307015}, and the remaining
	assertions follow from \ref{thm:distributions_Lp_representable} and
	\ref{remark:polar_decomposition}.
\end{example}

\begin{theorem} \label{thm:locally_finite_near_boundary}
	Suppose $V$ is as in
	\ref{example:regularity-hypotheses}.
		
	Then, every decomposition of $V$ is locally finite.
\end{theorem}

\begin{proof}
	We take $m$, $n$, $U$, and $B$ as in
	\ref{example:regularity-hypotheses}.

	If $m=1$, then the assertion is a special case of
	\cite[6.11]{MR3528825}. To prove the conclusion in case $m > 1$, we
	will verify the hypotheses of \ref{remark:criterion-locally-finite}.
	If $a \in U \without B$, then, in view of
	\cite[2.5]{MR2537022} and \ref{lemma:zero_distributional_boundary}, we
	may take $r > 0$ with $\mathbf B (a,2r) \subset U \without B$ and
	$\int_{\mathbf B(a,2r)} | \mathbf h (V,\cdot) |^m \ud \| V \| \leq
	2^{-m} \boldsymbol \upgamma (m)^{-m}$.
	
	If $a \in B$, we firstly notice that, for instance by
	\ref{remark:nearest-point-projection}, there exists $R > 0$ with
	$\mathbf B (a,R) \subset U$ and
	\begin{equation*}
		\big | \Nor (B,b)_\natural (y-b) \big | \leq R^{-1} |y-b|^2/2
	\end{equation*}
	whenever $b,y \in B \cap \mathbf U (a,R)$. Then, we pick $0 < r \leq
	R/2$ such that
	\begin{equation*}
		{\textstyle\int_{\mathbf B (a,2r)}} | \mathbf h (V,\cdot) |^m
		\ud \| V \| \leq 8^{-m} \boldsymbol \upgamma(m)^{-m}.
	\end{equation*}
	We will conclude the proof by showing that
	\begin{equation*}
		\| W \| \, \mathbf B (a,2r) \geq (16m \boldsymbol
		\upgamma(m))^{-m} r^m.
	\end{equation*}
	whenever $W$ is a component of $V$ and $\mathbf B (a,r) \cap \spt \| W
	\| \neq \varnothing$.  For this purpose, we abbreviate $A = \spt ( \|
	\updelta W \| - \| \updelta W \|_{\| W \|} )$ and assume
	\begin{equation*}
		\| W \| \, \mathbf B (a,2r) < ( 4 m \boldsymbol
		\upgamma(m))^{-m} r^m.
	\end{equation*}
	Noting \ref{lemma:zero_distributional_boundary}, we infer $A \cap
	\mathbf B (a,3r/2) \neq \varnothing$; in fact, taking $x \in \mathbf B
	(a,r) \cap \spt \| W \|$, we have $A \cap \mathbf B (x,r/2) \neq
	\varnothing$ by \cite[2.5]{MR2537022}.  As $A \subset B \cap \spt \| W
	\|$ by \ref{lemma:zero_distributional_boundary}, we may hence select
	\begin{equation*}
		b \in B \cap \mathbf B (a,3r/2) \cap \spt \| W \|.
	\end{equation*}
	Noting that \cite[3.4\,(1)]{MR0397520} remains valid%
	\begin{footnote}%
		{Notice that ``$m(s)^{1/k}$'' should be replaced by ''$k
		m(s)^{1/k}$'' in the cited statement and that ``$+m(t)$'' in
		line six of its proof should read ``$+km(t)$''.}
	\end{footnote}%
	when $\mathbf c_1$ is replaced by $\boldsymbol \upgamma(k)$ and $B$ is
	required to be of class $2$, instead of class $\infty$, we apply
	\cite[3.4\,(1)]{MR0397520} with
	\begin{align*}
		& \text{$B$, $k$, $s$, $V$, $\alpha$, and $r$ replaced by} \\
		& \text{$B \cap \mathbf U (a,R)$, $m$,
		$r/4$, $W | \mathbf 2^{\mathbf U (b,r/4) \times \mathbf G
		(n,m)}$, $( 8 \boldsymbol \upgamma (m))^{-1}$, and $r/4$}
	\end{align*}
	to obtain
	\begin{equation*}
		\| W \| \, \mathbf B (b,r/4) \geq (16m \boldsymbol
		\upgamma(m))^{-m} r^m,
	\end{equation*}
	as $R/(R-r/4) \leq 2$ and $m \boldsymbol \upgamma(m) \| W \| ( \mathbf
	B(b,r/4))^{1/m} / (R-r/4) \leq 1/4$.
\end{proof}

\begin{lemma} \label{lemma:zero-trace-test}
	Suppose $m$ and $n$ are positive integers, $m \leq n$, $U$ is an open
	subset of $\mathbf R^n$, $B$ is a relatively closed subset of $U$, and
	$H$ is the set of Lipschitzian functions $\eta : U \to \mathbf R^n$
	with compact support satisfying $\eta (b) = 0$ for $b \in B$.

	Then, the following three statements hold.
	\begin{enumerate}
		\item \label{item:zero-trace-test:density} If $\eta \in H$ and
		$\epsilon > 0$, then there exists $\theta \in H$ such that
		\begin{equation*}
			\sup \im | \eta-\theta | \leq \epsilon, \quad |\theta|
			\leq | \eta |, \quad \Lip \theta \leq \Lip \eta, \quad
			B \cap \spt \theta = \varnothing,
		\end{equation*}
		and, if $\eta$ is of class $1$, then so is $\theta$.
		\item \label{item:zero-trace-test:Radon} If $W \in \mathbf V_m
		( U \without B )$, $\| \updelta W \|$ is a Radon measure,
		$\eta \in H$, $h = \eta | ( U \without B )$, and $\| W \| (
		\spt h ) + \| \updelta W \|_{(1)}(h) < \infty$,%
		\begin{footnote}%
			{For the present paper, the stronger condition $( \| W
			\| + \| \updelta W \| ) ( K \without B) < \infty$
			whenever $K$ is a compact subset of $U$ would be
			sufficient.  The additional generality is essential in
			\cite{arXiv:2512.19227v1}.}
		\end{footnote}%
		then $h \in \mathbf T ( W, \mathbf R^n )$ and
		\begin{equation*}
			{\textstyle\int \trace (
			W \weakD h ) \ud \| W \|} = {\textstyle\int
			\boldsymbol \upeta (W,\cdot) \bullet h \ud \| \updelta
			W \|}.
		\end{equation*}
		\item \label{item:zero-trace-test:boundedness} If $V \in
		\mathbf V_m ( U )$ and
		\begin{equation*}
			\sup \{ {\textstyle\int S_\natural \bullet \Der \eta
			(x) \ud V \, (x,S)} \with \textup{$\eta \in H$ is of
			class $1$, $\spt \eta \subset K$, $|\eta| \leq 1$} \}
			< \infty
		\end{equation*}
		whenever $K$ is a compact subset of $U$, then
		\begin{equation*}
			{\textstyle\int_{B \times \mathbf G(n,m)} S_\natural
			\bullet \Der \eta (x) \ud V \, (x,S)} = 0 \quad
			\text{whenever $\eta \in H$ is of class $1$}.
		\end{equation*}
	\end{enumerate}
\end{lemma}

\begin{proof}
	Choosing a nonnegative function $\beta : \mathbf R \to \mathbf R$ of
	class $1$ with $\inf \spt \beta > 0$ such that $t-\epsilon \leq \beta
	(t) \leq t$ and $0 \leq \beta'(t) \leq 1$ for $0 \leq t < \infty$, and
	defining $\gamma : \mathbf R^n \to \mathbf R^n$ of class $1$ such that
	$\gamma(y) = \beta (|y|) |y|^{-1} y$ for $0 \neq y \in \mathbf R^n$,
	we verify
	\begin{equation*}
		\sup \im | \mathbf 1_{\mathbf R^n} - \gamma | \leq \epsilon,
		\quad | \gamma | \leq | \mathbf 1_{\mathbf R^n} |, \quad \sup
		\im \| \Der \gamma \| \leq 1, \quad 0 \notin \spt \gamma
	\end{equation*}
	and may take $\theta = \gamma \circ \eta$ in
	\eqref{item:zero-trace-test:density} because $\spt \theta \subset
	\eta^{-1} [ \spt \gamma ] \subset U \without B$.  Combining
	\cite[4.6]{MR3777387} with \eqref{item:zero-trace-test:density}
	reduces the proof of \eqref{item:zero-trace-test:Radon} to the case
	that $\spt \eta \subset U \without B$ which follows from
	\ref{remark:polar_decomposition}.  Next, under the hypotheses of
	\eqref{item:zero-trace-test:boundedness}, we suppose $\eta \in H$ is
	of class $1$.  The special case $\spt \eta \subset U \without B$ is
	trivial.  To treat the general case, we define $W = V | \mathbf 2^{(U
	\without B) \times \mathbf G (n,m)} \in \mathbf V_m ( U \without B )$,
	note
	\begin{equation*}
		{\textstyle\int_{B \times \mathbf G (n,m)} S_\natural \bullet
		\Der \eta (x) \ud V \, (x,S)} = {\textstyle\int S_\natural
		\bullet \Der \eta (x) \ud V \, (x,S)} - {\textstyle\int
		\boldsymbol \upeta (W, \cdot ) \bullet \eta \ud \| \updelta W
		\|}
	\end{equation*}
	by \eqref{item:zero-trace-test:Radon} in conjunction with \cite[4.1,
	4.5]{MR3777387}, and use the fact that, by
	\eqref{item:zero-trace-test:density}, there exist $\theta_1, \theta_2,
	\theta_3, \ldots$ in $H$ of class $1$ such that $\spt \theta_i \subset
	(\spt \eta) \without B$ for every positive integer $i$ and $\lim_{i
	\to \infty} \sup \im | \eta - \theta_i | = 0$.
\end{proof}

\begin{remark} \label{remark:boundary-not-charged}
	If $\| V \| (B) = 0$ and $W = V | \mathbf 2^{(U \without B) \times
	\mathbf G (n,m)}$, then the hypotheses of
	\eqref{item:zero-trace-test:Radon} and
	\eqref{item:zero-trace-test:boundedness} are equivalent by
	\ref{remark:polar_decomposition}.  In general, they differ as examples
	with $\card B = 1$ and $\card \spt V = 1$ show.
\end{remark}

\begin{theorem} \label{thm:Allard-constancy}
	Suppose $m$ and $n$ are positive integers, $m \leq n$, $U$ is an open
	subset of $\mathbf R^n$, $B$ is a submanifold of class $2$ in $U$, the
	inclusion map $i : B \to U$ is proper, $V \in \mathbf V_m (U)$, and
	\begin{equation*}
		\sup \left \{ {\textstyle\int S_\natural \bullet \Der \eta (x)
		\ud V \, (x,S)} \with \textup{$\eta \in H$ is of class $1$,
		$\spt \eta \subset K$, $| \eta | \leq 1$} \right \} < \infty
	\end{equation*}
	whenever $K$ is a compact subset of $U$, where $H$ is as in
	\ref{lemma:zero-trace-test}.

	Then, there holds $V \restrict B \times \mathbf G (n,m) \in i_\# [
	\mathbf V_m ( B ) ]$ in the case that $m \leq \dim
	B$ and $V \restrict B \times \mathbf G(n,m) = 0$ in the case that $m >
	\dim B$.
\end{theorem}

\begin{proof}
	In view of
	\ref{lemma:zero-trace-test}\,\eqref{item:zero-trace-test:boundedness},
	we may proceed as in \cite[4.6\,(1)\,(2)]{MR0307015}.
\end{proof}

\begin{remark}
	The case $\dim B = n-1$ was treated, under slightly more restrictive
	hypotheses, in \cite[Lemma 3.1]{MR4359920}; the case $\dim B \neq n-1$
	is new.
\end{remark}

\begin{remark} \label{remark:rectifiability-boundary-part}
	If $\| \updelta V \|$ is a Radon measure (see
	\ref{example:Neumann-boundary} below) and
	$m = \dim B$, then $V \restrict B \times \mathbf G
	(n,m)$ is rectifiable by \cite[3.5\,(1b),
	5.1\,(4)]{MR0307015} and 
	\cite[2.25]{DOI_10.4171_RMI_1487}.
\end{remark}

\begin{remark} \label{remark:mean-curv}
	We record the following related basic fact: If $\theta : U \to \mathbf
	R^n$ is a function of class $1$ with compact
	support and $\theta (b) \in \Nor (B,b)$ for $b \in B$, then
	\begin{equation*}
		S_\natural \bullet \Der \theta (b) = - \mathbf h (B,b,S)
		\bullet \theta (b) \quad \text{for $(b,S) \in \mathbf G_m ( B
		)$}.
	\end{equation*}
\end{remark}

\begin{example} \label{example:Neumann-setting}
	Suppose $m$ and $n$ are positive integers, $m \leq n$, $U$ is an open
	subset of $\mathbf R^n$, $M$ is a relatively closed subset of $U$ and
	an $n$ dimensional submanifold-with-boundary of class $2$ in $U$, $B =
	\partial M$, the open subset $X$ of $\mathbf R^n$ is associated with
	$B$ as in \ref{thm:reach-non-proper-submanifolds}, $\delta : X \to
	\mathbf R$ satisfies
	\begin{equation*}
		\text{$\delta (x) = \dist (x,B)$ if $x \in M$}, \quad
		\text{$\delta (x) = - \dist (x,B)$ if $x \notin M$}
	\end{equation*}
	whenever $x \in X$; hence, $B \subset X$, $\delta$ is of class $2$,
	$|\grad \delta (x) | = 1$ for $x \in X$, and
	\begin{equation*}
		\grad \delta (b) = - \mathbf n ( M, b ) \quad \text{for $b \in
		B$}
	\end{equation*}
	by \cite[4.8\,(3)]{MR0110078}, \cite[3.1.19\,(5)]{MR41:1976},
	\ref{thm:reach-real-valued-graphs},
	\ref{corollary:reach-real-valued-graphs}, and
	\ref{thm:reach-non-proper-submanifolds}.

	Suppose $W \in \mathbf V_m ( U \cap X \without B)$, $\spt \| W \|
	\subset M$, and
	\begin{equation*}
		( \| W \| + \| \updelta W \| ) ( K \without B ) < \infty \quad
		\text{whenever $K$ is a compact subset of $U \cap X$}.
	\end{equation*}
	Let $f = \delta | (U \cap X \without B)$, note $f \in \mathbf T ( W )$
	and $| W \weakD f | \leq 1$ by \cite[4.6\,(1)]{MR3777387}, and
	abbreviate
	\begin{equation*}
		E(y) = \{ x \with f(x)>y \} \quad \text{for $y \in \mathbf
		R$}.
	\end{equation*}
	Whenever $\zeta : U \cap X \to \mathbf R$ is a Lipschitzian function
	with compact support, we define a right-continuous
	function $g_\zeta : \{ y \with 0 \leq y < \infty \} \to \mathbf R$ by
	\begin{equation*}
		g_\zeta (y) = {\textstyle\int_{E(y)} ( \boldsymbol \upeta
		(W,\cdot) \bullet \grad f ) \zeta \ud \| \updelta W \|} -
		{\textstyle\int_{E(y)} \trace ( W \weakD ( \zeta \grad f ) )
		\ud \| W \|}
	\end{equation*}
	for $0 \leq y < \infty$ with the help of \cite[4.6\,(1)]{MR3777387},
	hence
	\begin{equation*}
		g_\zeta(0) = {\textstyle\int ( \boldsymbol \upeta (W,\cdot)
		\bullet \grad f ) \zeta \ud \| \updelta W \|} -
		{\textstyle\int \trace ( W \weakD ( \zeta \grad f ) ) \ud \| W
		\|};
	\end{equation*}
	we also let $Q_\zeta \in \mathscr D' ( \mathbf R, \mathbf R )$ be
	defined by $Q_\zeta ( \omega ) = \int_0^\infty g_\zeta \omega \ud
	\mathscr L^1$ for $\omega \in \mathscr D ( \mathbf R, \mathbf R )$.
	Finally, we define $T \in \mathscr D' ( U \cap X, \mathbf R )$ by $T (
	\zeta ) = g_\zeta (0)$ for $\zeta \in \mathscr D ( U \cap X, \mathbf R
	)$.
\end{example}

\begin{remark} \label{remark:trace-estimate}
	We record that
	\begin{equation*}
		| \trace ( h \circ S_\natural ) | \leq m \| h \| \quad
		\text{for $h \in \Hom ( \mathbf R^n, \mathbf R^n )$ and $S \in
		\mathbf G (n,m)$}.
	\end{equation*}
\end{remark}

\begin{lemma} \label{lemma:cut-towards-Neumann}
	Suppose $m$, $U$, $B$, $X$, $W$, $f$, $g_\zeta$, $Q_\zeta$, and $T$
	are as in \ref{example:Neumann-setting}.

	Then, the following two statements hold.
	\begin{enumerate}
		\item \label{item:cut-towards-Neumann:monotone} The
		distribution $T$ is representable by integration, $\spt T
		\subset B$,
		\begin{equation*}
			T ( \zeta ) = {\textstyle\int} \zeta \ud \| T \| \quad
			\text{for $\zeta \in \mathbf L_1 ( \| T \| )$}.
		\end{equation*}
		\item \label{item:cut-towards-Neumann:estimate} If $\zeta : U
		\cap X \to \mathbf R$ is a nonnegative Lipschitzian function
		with compact support, then $\alpha = \zeta | ( U \cap X
		\without B ) \in \mathbf T ( W )$ and
		\begin{equation*}
			{\textstyle\int \alpha \ud \| T \|} \leq
			{\textstyle\int \alpha \ud \| \updelta W \|} +
			{\textstyle\int ( | W \weakD \alpha | + m \alpha \|
			\Der^2 f \| ) \ud \| W \|}.
		\end{equation*}
		\item \label{item:cut-towards-Neumann:BV} Suppose $\zeta : U
		\cap X \to \mathbf R$ is Lipschitzian with compact support.
		Then,
		\begin{gather*}
			\begin{aligned}
				Q_\zeta ( \omega ) & = {\textstyle\int \zeta
				\langle \grad f, W \weakD f \rangle (\omega
				\circ f) \ud \| W \|}, \\
				\Der_1 Q_\zeta ( \omega ) & = g_\zeta (0)
				\omega (0) - {\textstyle\int \zeta (
				\boldsymbol \upeta (W,\cdot) \bullet \grad f)
				(\omega \circ f) \ud \| \updelta W \|} \\
				& \phantom = \ + {\textstyle\int \langle \grad
				f, W \weakD \zeta \rangle (\omega \circ f) \ud
				\| W \|} \\
				& \phantom = \ + {\textstyle\int \zeta \trace
				( W \weakD ( \grad f ) ) ( \omega \circ f) \ud
				\| W \|}
			\end{aligned}
		\end{gather*}
		whenever $\omega \in \mathscr D ( \mathbf R, \mathbf R )$; in
		particular,
		\begin{align*}
			\| Q_\zeta \| & \leq f_\# \big ( \| W \| \restrict |
			\zeta \, W \weakD f | \big ), \\
			\| \Der_1 Q_\zeta \| & \leq | g_\zeta(0) |
			\boldsymbol \updelta_0 + f_\# \big ( \| \updelta W \|
			\restrict | \zeta | + \| W \| \restrict ( | W \weakD
			\zeta | + m \| \zeta \Der^2 f \| ) \big ),
		\end{align*}
		and, if $\zeta \geq 0$, then $g_\zeta \geq 0$.  Moreover,
		there holds
		\begin{equation*}
			{\textstyle\int \zeta \ud \| T \| = g_\zeta (0)}.
		\end{equation*}
	\end{enumerate}
\end{lemma}

\begin{proof}
	We suppose $\zeta : U \cap X \to \mathbf R$ is Lipschitzian with
	compact support.

	Noting $B \cap \Clos E(y) = \varnothing$ for $0 < y < \infty$, we
	apply \ref{thm:area-formular-representation-slice} and
	\ref{remark:area-formular-representation-slice} with $U$, $V$, and
	$\theta$ replaced by $U \cap X \without B$, $W$, and $\zeta \grad f$
	to conclude
	\begin{align*}
		g_\zeta (y) & = W \boundary E(y) ( \zeta \grad f ) \\
		& = {\textstyle\int \zeta (x) \langle \grad f (x), | W \weakD
		f (x) |^{-1} W \weakD f (x) \rangle \ud \| W \boundary E(y)
		\| \, x}
	\end{align*}
	for $\mathscr L^1$ almost all $0 \leq y < \infty$; in particular,
	\begin{equation*}
		g_\zeta (0) = 0 \quad \text{in case $B \cap \spt \zeta =
		\varnothing$}
	\end{equation*}
	because $\spt W \boundary E(y) \subset \{ x \with f(x)=y \}$ for $y
	\in \mathbf R$.  Noting \cite[4.11]{MR3777387}, we apply \cite[8.5,
	8.30]{MR3528825} for each $\omega \in \mathscr D ( \mathbf R, \mathbf
	R )$ with $V$ and $g(x,y)$ replaced by $W$ and
	\begin{equation*}
		i(y) \omega (y) \zeta (x) \langle \grad f(x), | W \weakD f
		(x)|^{-1} W \weakD f (x) \rangle,
	\end{equation*}
	where $i$ is the characteristic function of $\{ y \with 0 \leq y <
	\infty \}$ on $\mathbf R$, to infer the first equation in
	\eqref{item:cut-towards-Neumann:BV}.  In case $\zeta \geq 0$, we
	deduce $Q_\zeta ( \omega ) \geq 0$ whenever $0 \leq \omega \in
	\mathscr D ( \mathbf R, \mathbf R )$, as
	\begin{equation*}
		\langle \grad f(x), W \weakD f (x) \rangle = {\textstyle\int |
		S_\natural ( \grad f(x)) |^2 \ud W^{(x)} \, S} \geq 0
	\end{equation*}
	for $\| W \|$ almost all $x$ by \cite[2.3\,(1)]{MR0307015} and
	\cite[4.1, 4.5]{MR3777387}.  By
	\ref{example:distributions_representable_by_integration} with $U$,
	$Y$, $\phi$, $k$, and $T$ replaced by $\mathbf R$, $\mathbf R$,
	$\mathscr L^1$, $g_\zeta$, and $Q_\zeta$, where $\mathbf R^\ast \simeq
	\mathbf R$, it follows that
	\begin{equation*}
		g_\zeta \geq 0  \quad \text{in case $\zeta \geq 0$}.
	\end{equation*}
	By means of \cite[4.1.5]{MR41:1976}, the preceding considerations
	readily yield \eqref{item:cut-towards-Neumann:monotone} and the final
	equation in \eqref{item:cut-towards-Neumann:BV} which in turn imply
	\eqref{item:cut-towards-Neumann:estimate} by
	\cite[8.20\,(4)]{MR3528825}, \cite[4.6\,(1), 4.11]{MR3777387}, and
	\ref{remark:trace-estimate}.

	To prove the second equation of \eqref{item:cut-towards-Neumann:BV},
	we additionally suppose $\omega \in \mathscr D ( \mathbf R, \mathbf
	R)$ and take $\theta = \zeta ( \omega \circ f ) \grad f$.  Noting
	\cite[4.6\,(1), 4.11]{MR3777387}, we compute
	\begin{align*}
		\trace ( W \weakD \theta(x)) & = \langle \grad f(x), W \weakD
		\zeta (x) \rangle \omega (f(x)) \\
		& \phantom = \ + \zeta (x) \langle \grad f(x), W \weakD f (x)
		\rangle \omega'(f(x)) \\
		& \phantom = \ + \zeta (x) \trace ( W \weakD ( \grad f ) (x) )
		\omega (f(x))
	\end{align*}
	for $\| W \|$ almost all $x$ by \cite[8.20\,(4)]{MR3528825}.
	Thus, the first equation of \eqref{item:cut-towards-Neumann:BV}
	yields
	\begin{align*}
		- Q_\zeta (\omega') & = - {\textstyle\int \trace ( W \weakD
		\theta ) \ud \| W \|} + {\textstyle\int \langle \grad f, W
		\weakD \zeta \rangle ( \omega \circ f ) \ud \| W \|} \\
		& \phantom = \ + {\textstyle\int \zeta \trace ( W \weakD (
		\grad f)) ( \omega \circ f ) \ud \| W \|}.
	\end{align*}
	Taking $\delta$ as in \ref{example:Neumann-setting}, hence $\delta [
	\spt T ] \subset \{ 0 \}$ by
	\eqref{item:cut-towards-Neumann:monotone}, the final equation in
	\eqref{item:cut-towards-Neumann:BV} shows
	\begin{equation*}
		{\textstyle\int \boldsymbol \upeta (W,\cdot) \bullet \theta
		\ud \| \updelta W \|} - {\textstyle\int \trace ( W \weakD
		\theta ) \ud \| W \|} = g_{\zeta (\omega \circ \delta)} (0) =
		g_\zeta (0) \omega (0)
	\end{equation*}
	and the second equation of \eqref{item:cut-towards-Neumann:BV}
	follows.  Taking \cite[4.1, 4.5]{MR3777387} and
	\ref{remark:trace-estimate} into account, the remaining assertions of
	\eqref{item:cut-towards-Neumann:BV} may now be verified by means of
	\cite[2.4.18\,(1)]{MR41:1976} and \cite[2.8,
	2.11]{DOI_10.4171_RMI_1487}.
\end{proof}

\begin{theorem} \label{thm:Neumann-normal-variations}
	Suppose that $n$, $U$, $M$, $B$, $X$, $W$, and $T$ are as in
	\ref{example:Neumann-setting}, that $\theta : U \cap X \to \mathbf
	R^n$ is a Lipschitzian function with compact support satisfying
	$\theta (b) \in \Nor ( B,b)$ for $b \in B$, and that $h = \theta | (U
	\without B )$.
	
	Then, $h \in \mathbf T ( W, \mathbf R^n )$ and there holds
	\begin{equation*}
		{\textstyle\int \trace ( W \weakD h ) \ud \| W \|} =
		{\textstyle\int \boldsymbol \upeta (W,\cdot) \bullet \theta
		\ud \| \updelta W \|} + {\textstyle\int \mathbf n (M,\cdot)
		\bullet \theta \ud \| T \|}.
	\end{equation*}
\end{theorem}

\begin{proof}
	We take $\delta$ and $g_\zeta$ as in \ref{example:Neumann-setting} and
	recall \cite[4.6\,(1)]{MR3777387}.  By additivity, see
	\cite[8.20\,(3)]{MR3528825} and \cite[4.11]{MR3777387}, it suffices to
	consider two cases, namely $\theta = ( \theta \bullet \grad \delta)
	\grad \delta$ and $\theta \bullet \grad \delta = 0$.  In the first
	case, we take $\zeta = \theta \bullet \grad \delta$ and note that
	\begin{equation*}
		g_\zeta (0) = {\textstyle \int \zeta \ud \| T \|} = -
		{\textstyle\int \mathbf n (M,\cdot) \bullet \theta \ud \| T
		\|}
	\end{equation*}
	by \ref{lemma:cut-towards-Neumann} which yields the conclusion in that
	case.  The second case implies $\theta (b) = 0$ for $b \in B$ and thus
	follows from
	\ref{lemma:zero-trace-test}\,\eqref{item:zero-trace-test:Radon}
	applied with $U$ and $\eta$ replaced by $U \cap X$ and $\theta$ in
	conjunction with \ref{lemma:cut-towards-Neumann}%
	\,\eqref{item:cut-towards-Neumann:monotone}.
\end{proof}

\begin{example} \label{example:Neumann-boundary}
	Suppose $m$ and $n$ are positive integers, $m \leq n$, $U$ is an open
	subset of $\mathbf R^n$, $M$ is a relatively closed subset of $U$ and
	an $n$ dimensional submanifold-with-boundary of class $2$, $B =
	\partial M$, $X$ is associated with $B$ as in
	\ref{thm:reach-non-proper-submanifolds}, and $V \in \mathbf V_m (U)$
	satisfies $\spt \| V \| \subset M$ and
	\begin{equation*}
		\sup \{ {\textstyle\int S_\natural \bullet \Der \theta (x) \ud
		V \, (x,S)} \with \text{$\theta \in \Theta$, $\spt \theta
		\subset K$, $| \theta | \leq 1$} \} < \infty
	\end{equation*}
	whenever $K$ is a compact subset of $U$, where $\Theta$ is the vector
	space of functions $\theta : U \to \mathbf R^n$ of class $1$ with
	compact support and $\theta (b) \in \Tan (B,b)$ for $b \in B$. Then,
	\begin{equation*}
		W = V | \mathbf 2^{(U \cap X \without B ) \times \mathbf G
		(n,m)} \in \mathbf V_m ( U \cap X \without B )
	\end{equation*}
	satisfies the conditions of \ref{example:Neumann-setting} and,
	defining $T \in \mathscr D' ( U \cap X, \mathbf R )$ by
	\begin{equation*}
		T ( \zeta ) = {\textstyle\int ( \boldsymbol \upeta (W,\cdot)
		\bullet \grad f ) \zeta \ud \| \updelta W \|} -
		{\textstyle\int \trace ( W \weakD ( \zeta \grad f ) ) \ud \| W
		\|}
	\end{equation*}
	for $\zeta \in \mathscr D ( U \cap X, \mathbf R)$, where $f : U \cap X
	\without B \to \mathbf R$ is given by
	\begin{equation*}
		\text{$f(x) = \dist ( x, B )$ if $x \in M$}, \quad \text{$f(x)
		= - \dist (x,B )$ if $x \notin M$}
	\end{equation*}
	for $x \in U \cap X \without B$, we conclude that $\| \updelta V \|$
	is a Radon measure and
	\begin{align*}
		& {\textstyle\int_B \boldsymbol \upeta ( V, b ) \bullet \Nor
		(B,b)_\natural ( \theta (b))  \ud \| \updelta V \| \, b } \\
		& \qquad = {\textstyle\int \mathbf n ( M, x ) \bullet \theta
		(x) \ud \| T \| \, x} - {\textstyle\int_{\mathbf G_m ( B )}
		\mathbf h ( B,b,S) \bullet \theta (b) \ud V \, (b,S)}
	\end{align*}
	whenever $\theta : U \to \mathbf R^n$ is of class $1$ with compact
	support; in fact, recalling \ref{lemma:cut-towards-Neumann}%
	\,\eqref{item:cut-towards-Neumann:monotone}, it suffices to note that
	if such $\theta$ satisfies $\spt \theta \subset U \cap X$ and $\theta
	(b) \in \Nor (B,b)$ for $b \in B$, then combining
	\ref{thm:Allard-constancy}, \ref{remark:mean-curv},
	\ref{thm:Neumann-normal-variations}, and \cite[4.1, 4.5]{MR3777387}
	yields
	\begin{align*}
		& {\textstyle\int S_\natural \bullet \Der \theta (x) \ud V \,
		(x,S)} = - {\textstyle \int_{\mathbf G_m ( B)} \mathbf h ( B,
		b, S ) \bullet \theta (b) \ud V \, (b,S)} \\
		& \qquad + {\textstyle\int \boldsymbol \upeta ( W, x ) \bullet
		\theta (x) \ud \| \updelta W \| \, x} + {\textstyle\int
		\mathbf n ( M,x ) \bullet \theta (x) \ud \| T \| \, x}.
	\end{align*}

	With the techniques of \cite[2.5.14]{MR41:1976}, we construct a Radon
	measure $\psi$ over $U$ such that
	\begin{equation*}
		\psi (G) = \sup \{ ( \updelta V ) ( \theta ) \with \theta \in
		\Theta, \spt \theta \subset G, | \theta | \leq 1 \}
	\end{equation*}
	whenever $G$ is an open subset of $U$.  Since the vector subspace
	$\Theta$ is $\| \updelta V \|_{(1)}$ dense in $\mathbf L_1 ( \|
	\updelta V \|, \mathbf R^n ) \cap \{ \theta \with \text{$\dmn \theta =
	U$, $\theta (b) \in \Tan (B,b)$ for $b \in B$} \}$ by
	\ref{example:distribution-submanifolds}, we conclude
	\begin{equation*}
		\psi = \| \updelta V \| \restrict | \langle \boldsymbol \upeta
		( V, \cdot ), \tau \rangle |
	\end{equation*}
	by means of \cite[2.8]{DOI_10.4171_RMI_1487},
	\ref{lemma:almost-retraction-to-ball}, and
	\ref{remark:polar_decomposition}, where $\tau : U \to \Hom ( \mathbf
	R^n, \mathbf R^n )$ is defined by
	\begin{equation*}
		\text{$\tau (x) = \mathbf 1_{\mathbf R^n}$ for $x \in U
		\without B$}, \quad \text{$\tau (x) = \Tan(B,x)_\natural$ for
		$x \in B$}.
	\end{equation*}
	Whenever $a \in \mathbf R^n$, $0 < r < \infty$, and $\mathbf B (a,2r)
	\subset U \cap X$, we will estimate
	\begin{align*}
		\| \updelta V \| \, \mathbf B (a,r) & \leq \big ( \psi + \| T
		\| + r^{-1} m \lambda \| V \| \big ) \, \mathbf B (a,r) \\
		& \leq \big ( 2 \psi + r^{-1} ( 1 + 2m \lambda )  \| V \| \big
		) \, \mathbf B (a,2r),
	\end{align*}
	where $\lambda = r \sup \| \Der^2 f \| [ \mathbf B (a,2r) ]$; in fact,
	for the first inequality, it suffices to note $\| \mathbf b (B,b) \| =
	\| \Der^2 f (b) \|$ for $b \in B$ and recall
	\ref{lemma:cut-towards-Neumann}%
	\,\eqref{item:cut-towards-Neumann:monotone}, whereas, for the second
	inequality, in view of \cite[4.6\,(1)]{MR3777387}, we may apply
	\ref{lemma:cut-towards-Neumann}%
	\,\eqref{item:cut-towards-Neumann:estimate} with
	\begin{equation*}
		\alpha (x) = \sup \{ 0, 1 - \dist (x, \mathbf B (a,r))/r \}
		\quad \text{for $x \in U \cap X \without B$}.
	\end{equation*}
	Finally, if $E$ is $\| V \| + \| \updelta V \|$ measurable, $V
	\boundary E = 0$, and $V' = V \restrict E \times \mathbf G (n,m)$,
	then $\updelta V' = ( \updelta V ) \restrict E$ yields that, whenever
	$K$ is a compact subset of $U$, we have
	\begin{equation*}
		\sup \{ {\textstyle\int S_\natural \bullet \Der \theta (x) \ud
		V' \, (x,S)} \with \text{$\theta \in \Theta$, $\spt \theta
		\subset K$, $| \theta | \leq 1$} \} < \infty
	\end{equation*}
	and $\psi' = \psi \restrict E$ by \ref{remark:polar_decomposition} and
	\ref{lemma:zero_distributional_boundary}, where $\psi'$ is associated
	with $V'$ as $\psi$ with $V$.
\end{example}

\begin{remark} \label{remark:Neumann-rectifiability}
	If $\boldsymbol \Uptheta^m ( \| V \|, x ) \geq 1$ for $\| V \|$ almost
	all $x$, then $V$ is rectifiable by \cite[5.5\,(1)]{MR0307015}.
\end{remark}

\begin{remark} \label{remark:comparison-DM21}
	Based on \ref{thm:reach-non-proper-submanifolds} and
	\ref{remark:nearest-point-projection}, the development in
	\ref{lemma:zero-trace-test}--\ref{remark:Neumann-rectifiability}
	localises the results of \cite[Section 3, Subsections
	4.1--4.3]{MR4359920}.  By fully employing the machinery of
	\cite{MR0110078}, \cite{MR0307015}, \cite{MR0397520},
	\cite{MR3528825}, and \cite{MR3777387}---in particular,
	\cite[4.8]{MR0110078}, \cite[2.5, 4.6]{MR0307015},
	\cite[2.2]{MR0397520}, \cite[8.5, 8.30]{MR3528825}, and \cite[4.6,
	4.11]{MR3777387}---and re-organising the material, we intend to
	provide a deeper understanding of the long computations of
	\cite[pp.\,2614--2625]{MR4359920}: Namely, \cite[Lemma 3.1]{MR4359920}
	is a special case of \ref{thm:Allard-constancy}; the content of
	\cite[Theorem 1.1, Theorem 4.1, and Corollary 4.2]{MR4359920} was
	split and, recalling \cite[4.8\,(3)]{MR0110078}, it is implied by
	\ref{thm:Allard-constancy}, \ref{remark:mean-curv},
	\ref{lemma:cut-towards-Neumann}%
	\,\eqref{item:cut-towards-Neumann:monotone}%
	\,\eqref{item:cut-towards-Neumann:estimate}, and
	\ref{thm:Neumann-normal-variations};%
	\begin{footnote}%
		{In \cite[Theorems 1.1 and 4.1]{MR4359920}, to obtain
		assertions supported by the proofs provided, \emph{second
		fundamental form} should be replaced by \emph{reach} regarding
		the dependence of $c (\mathscr M)$.}
	\end{footnote}%
	for \cite[Corollary 4.3]{MR4359920}, the same holds when
	\ref{remark:rectifiability-boundary-part} and
	\ref{example:Neumann-boundary} are taken into account; for
	\cite[Corollaries 4.4 and 4.6]{MR4359920}, one similarly takes note of
	\ref{remark:boundary-not-charged}; finally, \cite[Corollary
	4.7]{MR4359920} is contained in \ref{lemma:cut-towards-Neumann}%
	\,\eqref{item:cut-towards-Neumann:monotone}%
	\,\eqref{item:cut-towards-Neumann:estimate} and
	\ref{example:Neumann-boundary}.  The statements on the distributional
	derivative of $Q_\zeta$ in
	\ref{lemma:cut-towards-Neumann}\,\eqref{item:cut-towards-Neumann:BV}
	are new.
\end{remark}

\begin{example} \label{example:Neumann-Lm}
	Suppose $m$ and $n$ are integers, $1 \leq m \leq
	n$, $U$ is an open subset of $\mathbf R^n$, $M$ is a relatively closed
	subset of $U$ and an $n$ dimensional submanifold-with-boundary of
	class $2$, $B = \partial M$,
	$V \in \mathbf V_m ( U )$, $\spt \| V \| \subset M$, $\Theta$ consists
	of all $\theta : U \to \mathbf R^n$ of class $1$ with compact support
	and $\theta (b) \in \Tan (B,b)$ for $b \in B$, $1 <
	p \leq \infty$, $q = 1$ if $p = \infty$, and $q=p/(p-1)$ if $p <
	\infty$. Then, we will prove that the following two conditions are
	equivalent and imply those of \ref{example:Neumann-boundary}.
	\begin{enumerate}
		\item \label{item:Neumann-Lm:dual} If $K$ is a compact subset
		of $U$, then
		\begin{equation*}
			\sup \big \{ {\textstyle\int S_\natural \bullet \Der
			\theta (x) \ud V \, (x,S)} \with \text{$\theta \in
			\Theta$, $\spt \theta \subset K$,
			$\| V \|_{(q)} ( \theta ) \leq 1$}
			\big \} < \infty.
		\end{equation*}
		\item \label{item:Neumann-Lm:summability} The measure $\|
		\updelta V \|$ is a Radon measure, $\langle
		\mathbf h (V,\cdot), \tau \rangle \in \mathbf
		L_p^{\textup{loc}} ( \| V \|, \mathbf R^n )$, and
		\begin{equation*}
			{\textstyle\int S_\natural \bullet \Der \theta (x) \ud
			V \, (x,S)} = - {\textstyle\int \mathbf h (V,x)
			\bullet \theta (x) \ud \| V \| \, x} \quad
			\text{whenever $\theta \in \Theta$},
		\end{equation*}
		where $\tau : U \to \Hom ( \mathbf R^n, \mathbf R^n )$ is
		defined by
		\begin{equation*}
			\text{$\tau (x) = \mathbf 1_{\mathbf R^n}$ for $x \in
			U \without B$}, \quad \text{$\tau (x) =
			\Tan(B,x)_\natural$ for $x \in B$}.
		\end{equation*}
	\end{enumerate}
	Clearly, \eqref{item:Neumann-Lm:summability} implies
	\eqref{item:Neumann-Lm:dual} and \eqref{item:Neumann-Lm:dual} implies
	the conditions of \ref{example:Neumann-boundary}.  Thus, if
	\eqref{item:Neumann-Lm:dual} holds, then $\| \updelta V \|$ is a Radon
	measure by \ref{example:Neumann-boundary} and
	\eqref{item:Neumann-Lm:summability} follows by
	\ref{example:distribution-submanifolds} as the function $\langle
	\mathbf h (V,\cdot), \tau \rangle$ is $\| V \|$ almost characterised
	by the fact that, for $\| V \|$ almost all $x$,
	\begin{equation*}
		- \mathbf h (V,x) \bullet \theta (x) = \lim_{r \to 0+} \frac{
		( \updelta V ) ( b_{x,r} \cdot \theta )}{\| V \| \, \mathbf B
		(x,r)} \quad \text{whenever $\theta \in \Theta$},
	\end{equation*}
	where $b_{x,r}$ is the characteristic function of $\mathbf B (x,r)$ on
	$U$, which is true since $\theta$ is continuous and $\lim_{r \to 0+}
	\| \updelta V \| \, \mathbf B (x,r) \big / \| V \| \, \mathbf B (x,r)
	< \infty$ for $\| V \|$ almost all $x$ by \cite[2.8.18,
	2.9.5]{MR41:1976}.  Taking $\psi$ as in
	\ref{example:Neumann-boundary}, the preceding conditions imply
	\begin{equation*}
		( \| \updelta V \| - \| \updelta V \|_{\| V \|} ) \restrict |
		\langle \boldsymbol \upeta (V,\cdot), \tau \rangle | = 0, \quad
		\psi = \| V \| \restrict | \langle \mathbf h (V,\cdot), \tau
		\rangle |;
	\end{equation*}
	in fact, the first equation follows from \cite[2.9.2]{MR41:1976}, as
	$\psi$ is absolutely continuous with respect to $\| V \|$, and
	entails the second equation by \cite[2.8]{DOI_10.4171_RMI_1487} and
	\ref{remark:polar_decomposition}.  Finally, if $E$ is $\| V \| + \|
	\updelta V \|$ measurable, $V \boundary E = 0$, and $W = V \restrict E
	\times \mathbf G (n,m)$, then we employ $\updelta W = ( \updelta V )
	\restrict E$ in conjunction with \ref{remark:polar_decomposition} and
	\ref{lemma:zero_distributional_boundary} to verify
	\begin{equation*}
		\sup \big \{ {\textstyle\int S_\natural \bullet \Der \theta
		(x) \ud W \, (x,S)} \with \text{$\theta \in \Theta$, $\spt
		\theta \subset K$, $\| W \|_{(q)} ( \theta
		) \leq 1$} \big \} < \infty
	\end{equation*}
	whenever $K$ is a compact subset of $U$.
\end{example}

\begin{remark} \label{remark:free-boundary-constrained}
	By \ref{remark:immersion_varifold}, whenever $N$ is an $m$ dimensional
	manifold-with-boundary of class~$2$, $F : N \to U$ is a proper
	immersion of class $2$, $F[N] \subset M$, and $V$ is associated with
	$(F,U)$, the varifold $V$ satisfies the conditions of
	\ref{example:Neumann-Lm} with $p = \infty$ if
	\begin{equation*}
		\mathbf n (F,c) \in \Nor (M,F(c)) \quad \text{for $c \in
		\partial N$};
	\end{equation*}
	in case $F|\partial N$ is an embedding, the condition on the exterior
	normal is necessary whenever $1 < p \leq \infty$.
	The special case that $U = \mathbf R^n$, $M$ compact, $N \subset M$,
	$F = \mathbf 1_N$, and $\mathbf h (F,\cdot) = 0$, pertaining to the
	\emph{constrained free boundary problem}, is of particular importance;
	a general existence theorem of such $N$ was obtained in case $2 \leq m
	\leq 6$, $n=m+1$, and $M$ of class $\infty$ in \cite[Theorem
	1.1]{MR4285846}.  In the classical subcase that $n = 3$ and $M =
	\mathbf B (0,1)$, existence of $N$ with prescribed number of boundary
	components and genus $0$ or with prescribed genus and connected
	boundary was treated in \cite[Theorem 1.6]{MR3461367} and
	\cite[Theorem 1.1]{MR4524829}, respectively.
\end{remark}

\begin{remark} \label{remark:free-boundary-sliding}
	If $E$ is an $\mathscr H^m$ measurable subset of $M$ which meets every
	compact subset of $U$ in an $(\mathscr H^m,m)$ rectifiable set, and
	\begin{equation*}
		\mathscr H^m ( E \cap \spt \theta ) \leq \mathscr H^m ( \phi
		(t, \cdot ) [E] \cap \spt \theta ) \quad \text{for $t \in
		\mathbf R$ and $\theta \in \Theta$},
	\end{equation*}
	where $\phi$ is the flow associated with $\theta$, then $V = \mathbf
	v_m ( E ) \in \mathbf{RV}_m (U)$ satisfies the conditions of
	\ref{example:Neumann-boundary} and of \ref{example:Neumann-Lm} with
	$\psi = 0$ by \cite[2.27]{DOI_10.4171_RMI_1487}.  According to
	\cite[Theorem 5.16]{MR3975493}, our treatment therefore locally
	applies to \emph{minimal sets} in $U$ with \emph{sliding boundary
	conditions} given by $M$ and $B$ in the sense of \cite{MR3975493}; in
	the special case $m = 2$, $n=3$, $U = \mathbf R^n$, and $M$ compact,
	an existence result for such sets is established in \cite[Theorem
	8.1]{MR4220652}.  Adopting the terminology of \cite[Section
	12]{MR3800850} and noting \cite[Lemma 2.1.2]{MR4489608}, we obtain the
	following proposition: \emph{If $M$ is a compact $n$ dimensional
	submanifold-with-boundary of class $2$ of $\mathbf R^n$, $B = \partial
	M$, $G$ is a commutative group, $L$ is a subgroup of the $(m-1)$-th
	Čech homology group of $B$, $\mathscr{\check C} ( B, L, G)$ denotes
	the family of closed subsets of $\mathbf R^n$ spanning $L$, $M \supset
	E \in \mathscr{\check C} ( B, L, G )$,
	\begin{equation*}
		\mathscr H^m ( E \without B ) = \inf \{ \mathscr H^m ( F
		\without B ) \with M \supset F \in \mathscr{\check C} ( B, L,
		G ) \},
	\end{equation*}
	and $E \without B$ is $(\mathscr H^m,m)$ rectifiable, then $E \without
	B$ satisfies the preceding conditions}; for given $(M,n,G,L,m)$ such
	that the infimum is finite, the existence of such a set $E$ follows
	from \cite[Theorem 3.20]{MR3800850} with $U$ replaced by $\mathbf R^n
	\without B$.  The preceding proposition remains valid when ``$\without
	B$'' is omitted both in its hypotheses and its conclusion; however,
	there is no known existence resulting in this case.  Finally, it is not
	clear whether such sets $E$ must be minimal in $\mathbf R^n$ with
	sliding boundary conditions given by $M$ and $B$, see \cite[Remark
	7.7]{MR3329849}.
\end{remark}

\begin{remark} \label{remark:free-boundary-Willmore}
	Two related classes of \emph{curvature varifolds with boundary} and
	$p$-th power summable \emph{weak second fundamental form} in the sense
	of \cite{MR1412686} occur in \cite[Theorems 4 and 5]{MR4787572}; the
	first one satisfies the conditions of \ref{example:Neumann-boundary}
	and, if $p>1$, of \ref{example:Neumann-Lm}, whereas the second one
	satisfies \ref{remark:Neumann-rectifiability} and
	\ref{example:Neumann-Lm} with $p = 2$.
\end{remark}

\begin{remark} \label{remark:free-boundary-MCF}
	For $\mathscr L^1$ almost all \emph{times}, rectifiable varifolds
	satisfying the conditions of \ref{example:Neumann-Lm} with $p = 2$
	occur in \emph{Brakke flow with free boundary}, see \cite[Theorems 2.1
	and 2.5]{MR3348119} jointly with \cite{MR3542008}, \cite[Theorems 3.3,
	3.5, and 3.6\,(A1)]{MR3953134}, and \cite[Theorem 9.1]{MR4048443}.
	Such varifolds also appear in \emph{level set mean curvature flow with
	Neumann boundary conditions}, see \cite[Theorem 2.6]{MR4612634}.
\end{remark}

\begin{lemma} \label{lemma:calculus}
	Suppose $m$, $r$, and $\gamma$ are positive real numbers, $0 < \delta
	\leq (8m\gamma)^{-1}$, $m \geq 1$, $g : \{ s \with r/4 \leq s \leq r
	\} \to \mathbf R$ is a nondecreasing function, $g(r/4) \geq ( \delta
	r/4)^m$, and
	\begin{equation*}
		g(s)^{1-1/m} \leq 2^{-2m-2} \delta^{-1} s^{-1} g(2s) + 2^{-2m}
		g(2s)^{1-1/m} + \gamma g'(s)
	\end{equation*}
	for $\mathscr L^1$ almost all $s$ with $r/4 \leq s \leq r/2$.

	Then, $g(r) \geq \delta^m r^m$.
\end{lemma}

\begin{proof}
	If the lemma were false, noting that, for $r/4 \leq s \leq r/2$, we
	would have
	\begin{gather*}
		g(2s) \leq \delta^m r^m \leq 2^{2m} \inf \{ g(s), \delta^m s^m
		\}, \\
		2^{-2m-2} \delta^{-1} s^{-1} g (2s) + 2^{-2m} g(2s)^{1-1/m}
		\leq 2^{1-2m} g(2s)^{1-1/m} \leq 2^{-1} g(s)^{1-1/m},
	\end{gather*}
	we would conclude $(2m\gamma)^{-1} \leq (g^{1/m})'(s)$ for $\mathscr
	L^1$ almost all $s$ with $r/4 \leq s \leq r/2$ so that
	\begin{equation*}
		(8m\gamma)^{-1} r \leq {\textstyle\int_{r/4}^{r/2} (g^{1/m})'
		\ud \mathscr L^1} \leq g(r)^{1/m} < \delta r
	\end{equation*}
	by \cite[2.9.19]{MR41:1976}, in contradiction to $\delta \leq
	(8m\gamma)^{-1}$.
\end{proof}

\begin{theorem} \label{thm:Neumann-lower-density-ratio-bound}
	Suppose $m$, $U$, $X$, and $f$ are as in
	\ref{example:Neumann-boundary}, $G$ is an open subset of $U \cap X$,
	$\kappa = \sup \| \Der^2 f \| [G] < \infty$,
	\begin{enumerate}
		\item if $m = 1$, then $V$ and $\tau$ are as in
		\ref{example:Neumann-boundary}, and
		\begin{equation*}
			{\textstyle \int_G | \langle \boldsymbol \upeta
			(V,\cdot), \tau \rangle | \ud \| \updelta V \| \leq
			2^{-3} \boldsymbol \upgamma (1)^{-1}},
		\end{equation*}
		\item if $m \geq 2$, then $V$ and $\tau$ are as in
		\ref{example:Neumann-Lm}\,\eqref{item:Neumann-Lm:summability}
		with $p = m$, and
		\begin{equation*}
			{\textstyle \big ( \int_G | \langle \mathbf h (V,
			\cdot ), \tau \rangle |^m \ud \| V \| \big )^{1/m}
			\leq 2^{-2m-1} \boldsymbol \upgamma (m)^{-1}},
		\end{equation*}
	\end{enumerate}
	and $\boldsymbol \Uptheta^m ( \| V \|, x ) \geq 1$ for $\| V \|$
	almost all $x \in G$.

	Then, there holds
	\begin{equation*}
		\| V \| \, \mathbf B (a,r) \geq \delta^m r^m, \quad
		\text{where $\delta = 1 \big / \big ( 2^{2m+2} \boldsymbol
		\upgamma (m) (1+2m\kappa r) \big )$},
	\end{equation*}
	whenever $a \in \spt \| V \|$, $0 < r < \infty$, and $\mathbf B (a,r)
	\subset G$.
\end{theorem}

\begin{proof}
	Taking $n$ as in \ref{example:Neumann-Lm}, $V \restrict G \times
	\mathbf G (n,m) \in \mathbf{RV}_m (U)$ by
	\ref{remark:Neumann-rectifiability}.  It suffices to consider $a$ with
	$\boldsymbol \Uptheta^m ( \| V \|, a) \geq 1$.  We define $g : \{ s
	\with 0 < s \leq r \} \to \mathbf R$ by
	\begin{equation*}
		g(s) = \| V \| \, \mathbf B (a,s) \quad \text{for $0 < s \leq
		r$}.
	\end{equation*}
	The isoperimetric inequality applied to $i_\# ( V \restrict
	\mathbf B (a,s) \times \mathbf G (n,m)) \in \mathbf {RV}_m ( \mathbf
	R^n)$, where $i : U \to \mathbf R^n$ is the inclusion map, in
	conjunction with \cite[8.7, 8.29]{MR3528825}, yields
	\begin{equation*}
		g(s)^{1-1/m} \leq \boldsymbol \upgamma (m) \big ( \| \updelta
		V \| \, \mathbf B (a,s) + g'(s) \big )
	\end{equation*}
	for $\mathscr L^1$ almost all $s$ with $0 < s \leq r$.  Since
	\begin{equation*}
		\| \updelta V \| \, \mathbf B (a,s) \leq 2^{-2m} \boldsymbol
		\upgamma (m)^{-1} g(2s)^{1-1/m} + s^{-1} (1+2m\kappa r) g(2s),
	\end{equation*}
	for $0 < s \leq r/2$ by \ref{example:Neumann-boundary},
	\ref{example:Neumann-Lm}, and Hölder's inequality, we infer
	\begin{equation*}
		g(s)^{1-1/m} \leq 2^{-2m-2} \delta^{-1} s^{-1} g(2s) + 2^{-2m}
		g(2s)^{1-1/m} + \boldsymbol \gamma (m) g'(s)
	\end{equation*}
	for $\mathscr L^1$ almost all $s$ with $0 < s \leq r/2$.  Recalling
	$\boldsymbol \gamma (m) \geq \boldsymbol \alpha (m)^{-1/m}/m$ from
	\cite[2.4]{MR2537022}, we observe that $\delta^m < \boldsymbol \alpha
	(m)$. Applying \ref{lemma:calculus} with $\gamma =
	\boldsymbol \gamma (m)$ and $r$ replaced by $2^{-i}r$
	for nonnegative integers $i$ then yields the conclusion.
\end{proof}

\begin{remark} \label{remark:Neumann-lower-density-ratio-bound}
	The preceding theorem adapts \cite[8.3]{MR0307015} to the present
	setting.  A possible strengthening analogous to \cite[2.5]{MR2537022}
	remains open though.
\end{remark}

\begin{corollary} \label{corollary:Neumann-locally-finite-decompositions}
	Suppose $m$ and $V$ satisfy the equivalent conditions of
	\ref{example:Neumann-Lm} with $p = m \geq 2$ and
	$\boldsymbol \Uptheta^m ( \| V \|, x ) \geq 1$ for $\| V \|$ almost
	all $x$.

	Then, every decomposition of $V$ is locally finite.
\end{corollary}

\begin{proof}
	Recalling $\| \updelta V \|$ is a Radon measure by
	\ref{example:Neumann-Lm}\,\eqref{item:Neumann-Lm:summability}, that
	$V$ is rectifiable by \ref{remark:Neumann-rectifiability}, we verify
	the condition of \ref{remark:criterion-locally-finite}: Noting
	\ref{lemma:zero_distributional_boundary} and taking $U$ and $B$ as in
	\ref{example:Neumann-Lm}, this follows from \cite[2.5]{MR2537022} for
	$a \in U \without B$ and from
	\ref{thm:Neumann-lower-density-ratio-bound} for $a \in B$.
\end{proof}

\begin{remark} \label{remark:Neumann-locally-finite-decompositions}
	By \cite[6.11]{MR3528825}, the conclusion also holds if $m = 1$, $V$
	satisfies the conditions of \ref{example:Neumann-boundary}, and
	$\boldsymbol \Uptheta^1 ( \| V \|, x ) \geq 1$ for $\| V \|$ almost
	all $x$.
\end{remark}

\section{Criteria for indecomposability with respect to a family of
generalised weakly differentiable real valued functions}

Firstly, in
\ref{def:Riemannian_distance}--\ref{remark:connectedness-immersions}, we treat
varifolds associated with immersions.  This includes a constancy theorem in
\ref{thm:constancy} and Theorem \ref{Thm:connected-immersions} of the
introductory section in \ref{thm:indecomposability_immersions}.  Secondly, in
\ref{thm:indecomposable_chains}--%
\ref{remark:indecomposable_chains_decomposable_varifolds}, we proceed to
varifolds associated with integral chains with coefficients in a complete
normed commutative group; in particular, we provide Theorem
\ref{Thm:Indecomposability-G-chains} of the introductory section in
\ref{thm:indecomposable_chains}.  Finally, we study
several purely varifold-geometric settings in
\ref{lemma:decompositions}%
--\ref{remark:Neumann-indecomposable} and derive
Theorem \ref{Thm:indecomposability-type-G} and Corollaries
\ref{Corollary:equivalence}, \ref{Thm:E}, and
\ref{Thm:Neumann} of the introductory section in
\ref{thm:indecomposable_type_Psi},
\ref{corollary:indecomposability-equivalence},
\ref{corollary:indecomposable_type_Gamma}, and
\ref{corollary:Neumann-indecomposable}, respectively.

\begin{definition} \label{def:Riemannian_distance}
	Suppose the pair $(M,g)$~consists of a connected
	manifold-with-boundary~$M$ of class~$1$ and a Riemannian metric~$g$
	on~$M$ of class~$0$.
	
	Then, the \emph{Riemannian distance~$\sigma$ of~$(M,g)$} is the
	function on~$M \times M$ whose value at $(c,z) \in M \times M$ equals
	the infimum of the set of numbers
	\begin{equation*}
		{\textstyle \int_K \langle (C'(y), C'(y)), g(C(y))
		\rangle^{1/2} \ud \mathscr L^1 \, y}
	\end{equation*}
	corresponding to all locally Lipschitzian%
	\begin{footnote}
		{That is, $C$ is continuous and $\phi \circ C$ is locally
		Lipschitzian whenever $\phi$ is a chart of~$M$ of class~$1$. A
		posteriori, this is equivalent to $C$~being locally
		Lipschitzian with respect to~$\sigma$.}
	\end{footnote}%
	functions $C$ mapping some compact interval~$K$ into~$M$ with $C (
	\inf K ) = c$ and $C ( \sup K ) = z$.
\end{definition}

\begin{remark}
	In our development, $g$ is always induced by an immersion into
	$\mathbf R^n$.
\end{remark}

\begin{lemma} \label{lemma:immersions}
	Suppose $m$ and $n$ are positive integers, $m \leq n$, $M$~is a
	connected $m$~dimensional manifold-with-boundary of class~$1$, $F : M
	\to \mathbf R^n$ is an immersion of class~$1$, $g$~is the Riemannian
	metric on $M$ induced by $F$ from $\mathbf R^n$, and $\sigma$ is the
	Rie\-man\-ni\-an distance associated with~$(M,g)$.
	
	Then, the function~$\sigma$ is a metric on~$M$ inducing the given
	topology on~$M$ and $F_\# \mathscr H^k_\sigma = \mathscr H^k \restrict
	\Number ( F, \cdot)$ whenever $0 \leq k < \infty$.
\end{lemma}

\begin{proof}
	We first verify that one may reduce the statement to the case that
	$F$~is an embedding, hence to case that $M \subset \mathbf R^n$ and $F
	= \mathbf 1_M$.  Clearly, $|z-\zeta| \leq \sigma (z,\zeta)$ for
	$z,\zeta \in M$, hence $\mathscr H^k (S) \leq \mathscr H_\sigma^k (S)$
	for $S \subset M$.  On the other hand, given $1 < \lambda < \infty$
	and $c \in M$, there exists $\delta > 0$ such that
	\begin{equation*}
		\sigma (z,\zeta) \leq \lambda |z-\zeta| \quad \text{whenever
		$z,\zeta \in M \cap \mathbf B (c,\delta)$};
	\end{equation*}
	in fact, we observe that it is sufficient to note that the
	chart~$\psi$ of~$\mathbf R^n$ of class~$1$ occurring in the definition
	of submanifold-with-boundary of~$\mathbf R^n$ of class~$1$
	(see~\cite[p.\,30]{MR1336822}) may be required to satisfy $\Der \psi
	(c) = \mathbf 1_{\mathbf R^n}$.  This in particular implies $\mathscr
	H_\sigma^k (S) \leq \lambda^k \mathscr H^k ( S )$ for $S \subset M
	\cap \mathbf B (c,\delta)$ and the conclusion follows.
\end{proof}

\begin{remark} \label{remark:hausdorff_measure}
	Denoting by $h$~the Riemannian metric on~$\partial M$ induced by~$F|
	\partial M$ from $\mathbf R^n$ and by $\rho$ the Riemannian distance
	associated with $(\partial M,h)$, the preceding lemma yields in
	particular $\mathscr H^{m-1}_\rho (S) = \mathscr H^{m-1}_\sigma (S)$
	for $S \subset \partial M$.  Moreover, the measures $\mathscr
	H^m_\sigma$ and~$\mathscr H^{m-1}_\rho$ agree with the usual
	Riemannian measures (see~\cite[Section~2.5]{MR1390760}) associated
	with $(M,g)$ and $(\partial M,h)$, respectively,
	by~\cite[3.2.46]{MR41:1976}.
\end{remark}

\begin{theorem} \label{thm:constancy}
	Suppose $m$, $M$, and $\sigma$ are as in
	\ref{lemma:immersions}, $C$ is a subset of $M$, and
	$\mathscr H^{m-1}_\sigma ( ( \Bdry C ) \without \partial M ) = 0$.
	
	Then, either $\mathscr H^m_\sigma ( C ) = 0$ or $\mathscr H^m_\sigma (
	M \without C ) = 0$.
\end{theorem}

\begin{proof}
	Noting that $M \without \partial M$ is connected and $\mathscr
	H^m_\sigma ( \partial M ) = 0$ by \ref{remark:hausdorff_measure}, we
	assume $\partial M = \varnothing$.  Next, we observe that it is
	sufficient to prove that there holds either $\mathscr H^m_\sigma (C
	\cap \dmn \phi ) = 0$ or~$\mathscr H^m_\sigma ( ( \dmn \phi ) \without
	C ) = 0$ whenever $\phi$ is a chart of~$M$ of class~$1$ satisfying
	$\im \phi = \mathbf R^m$, as $M$~is covered by the domains of a
	countable collection of such charts.  To verify this dichotomy, we let
	$A = \phi [C]$ and infer $\mathscr H^{m-1} ( \Bdry A) = 0$
	from~\cite[3.2.46]{MR41:1976}, hence $\Bdry A = \varnothing$ if $m=1$
	and $\mathscr I^{m-1}_1 ( \Bdry A ) = 0$ if $m>1$
	by~\cite[2.10.15]{MR41:1976}.  This implies $A$~is of locally finite
	perimeter by~\cite[4.5.11]{MR41:1976} and that $\boundary ( \mathbf
	E^m \restrict A ) = 0$ by~\cite[4.5.6\,(1)]{MR41:1976}.  Consequently,
	there holds $\mathscr L^m (A) = 0$ or $\mathscr L^m ( \mathbf R^m
	\without A ) = 0$ by the constancy theorem~\cite[4.1.7]{MR41:1976},
	whence we deduce the assertion by~\cite[3.2.46]{MR41:1976}.
\end{proof}

\begin{remark}
	Alternatively to using \cite[4.5.6, 4.5.11]{MR41:1976}, one could
	employ an argument based on capacity (see, e.g.,
	\cite[5.7]{MR3528825}) to infer $\boundary ( \mathbf E^m \restrict A )
	= 0$ from~$\mathscr H^{m-1} (\Bdry A) = 0$.
\end{remark}

\begin{theorem} \label{thm:indecomposability_immersions}
	Suppose $m$ and~$n$ are positive integers, $m \leq n$, $U$ is an open
	subset of $\mathbf R^n$, $\Lambda$ is the class of all locally
	Lipschitzian real valued functions with domain $U$, $M$ is a connected
	$m$~dimensional manifold-with-boundary of class $2$, $F : M \to U$ is
	a proper immersion of class~$2$, and $V$ is the varifold associated
	with $(F,U)$.
	
	Then, $V$ is indecomposable of type $\Lambda$.
\end{theorem}

\begin{proof}
	Clearly, $V$ is rectifiable and $\| \updelta V \|$~is a Radon measure
	by~\ref{remark:immersion_varifold}.  Suppose $f : U \to \mathbf R$ is
	locally Lipschitzian, let
	\begin{equation*}
		B(y) = \{ x \with f(x) = y \}, \quad E(y) = \{ x \with
		f(x)>y\}
	\end{equation*}
	for $y \in \mathbf R$, note $\Bdry F^{-1} [E(y)] \subset F^{-1}
	[B(y)]$, and define $\sigma$ as in \ref{lemma:immersions}.  Since
	$\boldsymbol \Uptheta^m ( \| V \|, x ) = \Number(F,x)$ for $\mathscr
	H^m$~almost all $x \in U$ by \cite[3.5\,(1b)]{MR0307015}
	and~\ref{remark:immersion_varifold}, we employ \cite[2.33,
	2.34]{DOI_10.4171_RMI_1487}, \cite[2.10.25]{MR41:1976}, and
	\ref{lemma:immersions} to infer
	\begin{align*}
		\| V \boundary E(y) \| (U) & = {\textstyle\int_{B(y)}}
		\boldsymbol \Uptheta^m ( \| V \|, x ) \ud \mathscr H^{m-1} \,
		x \\
		& = {\textstyle\int_{B(y)}} \Number(F,x)\ud \mathscr H^{m-1}
		\, x = ( F_\# \mathscr H^{m-1}_\sigma ) ( B(y) )
	\end{align*}
	for $\mathscr L^1$ almost all $y$.  Whenever $V \boundary E (y) = 0$
	for such~$y$, we apply~\ref{thm:constancy} with $C = F^{-1} [E(y)]$ to
	conclude that
	\begin{equation*}
		\text{either $F_\# \mathscr H^m_\sigma ( E(y) ) = 0$} \quad
		\text{or $F_\# \mathscr H^m_\sigma ( U \without E(y)) = 0$}.
	\end{equation*}
	Since $F_\# \mathscr H^m_\sigma = \| V \|$ by \ref{lemma:immersions}
	and~\ref{remark:immersion_varifold}, the conclusion follows.
\end{proof}

\begin{remark} \label{remark:connectedness-immersions}
	Simple examples with $\Number(F,x) = 2$ for $x \in \im F$, yield that
	conversely even indecomposability of~$V$ need not imply connectedness
	of~$M$.
\end{remark}

\begin{theorem} \label{thm:indecomposable_chains}
	Suppose $m$ and $n$ are positive integers, $m \leq n$, $U$ is an open
	subset of $\mathbf R^n$, $G$ is a complete normed commutative group,
	$S \in \mathbf I_m( U, G)$ is indecomposable, $V \in \mathbf{RV}_m(
	U)$ is characterised by $\| V \| = \| S \|$, $\| \updelta V \|$ is a
	Radon measure, and $\Lambda$ denotes the class of all locally
	Lipschitzian real valued functions with domain $U$.
	
	Then, $V$ is indecomposable of type $\Lambda$.
\end{theorem}

\begin{proof}
	Suppose $f: U \to \mathbf R$ is locally Lipschitzian, define $E(y) =
	\{ x \with f(x) > y\}$ for $y \in \mathbf R$, and let $Y = \mathbf R
	\cap \{ y \with V \boundary E(y)=0 \}$.  Firstly, \cite[2.33,
	2.34]{DOI_10.4171_RMI_1487} yield
	\begin{equation*}
		(\mathscr H^{m-1} \restrict \{ x \with f(x) = y \}) \restrict
		\boldsymbol \Uptheta^m(\|V\|, \cdot) = 0 \quad \text{for
		$\mathscr L^1$ almost all $y \in Y$}.
	\end{equation*}
	We then infer $\langle S,f,y \rangle = 0$ for $\mathscr L^1$ almost
	all $y \in Y$ by \cite[3.8]{DOI_10.4171_RMI_1487} and hence
	\begin{equation*}
		S \restrict E(y) \in \mathbf I_m ( U, G ) \quad \text{and}
		\quad \boundary_G ( S \restrict E(y) ) = ( \boundary_G S )
		\restrict E(y)
	\end{equation*}
	for $\mathscr L^1$ almost all $y \in Y$ by
	\cite[4.13\,(8)]{DOI_10.4171_RMI_1487}.  In view of
	\cite[3.5]{DOI_10.4171_RMI_1487}, the indecomposability of $S$ finally
	yields that either $\| S \| ( E(y) ) = 0$ or $\| S \| ( U \without
	E(y) ) = 0$ for such $y$.
\end{proof}

\begin{remark} \label{remark:set-indecomposable_chains}
	The indecomposability hypothesis on $S$ may be weakened to the
	requirement that there exists no Borel subset $E$ of $U$ satisfying
	\begin{equation*}
		\| S \| ( E ) > 0, \quad \| S \| ( U \without E ) > 0, \quad
		\text{and} \quad \boundary_G ( S \restrict E ) = ( \boundary_G
		S) \restrict E.
	\end{equation*}
	In \cite[Definition 1.1\,(2)]{MR4814697}, the strictly%
	\begin{footnote}%
		{We may consider $m=n=1$, $U = \mathbf R$, $G = \mathbf Z$,
		and $Q = \boldsymbol [ 0,1 \boldsymbol ] + \boldsymbol [ 0, -1
		\boldsymbol ] \in \mathbf I_1 ( \mathbf R )$.}
	\end{footnote}%
	stronger condition of \emph{set-in\-de\-com\-pos\-abil\-i\-ty}---still
	strictly weaker than indecomposability---is introduced in the context
	of the notion of \emph{flat $G$ chains} defined in \cite{MR1738045}.
\end{remark}

\begin{remark} \label{remark:indecomposable_chains_decomposable_varifolds}
	It may happen that $V$ is decomposable, see \ref{example:four-arcs}.
\end{remark}

\begin{lemma} \label{lemma:decompositions}
	Suppose $m$ and $n$ are positive integers, $m \leq n$, $U$ is an open
	subset of $\mathbf R^n$, $V \in \mathbf V_m ( U )$, $\| \updelta V \|$
	is a Radon measure, $\Phi$ is the family of connected components of
	$\spt \| V \|$, and $\Xi$ is a locally finite decomposition of $V$.
	
	Then, the following four statements hold.
	\begin{enumerate}
		\item \label{item:decompositions:refinement} If $C \in \Phi$,
		then $C = \bigcup \{ \spt \| W \| \with W \in \Xi, C \cap \spt
		\| W \| \neq \varnothing \}$.
		\item \label{item:decompositions:finite} If $K$ is a compact
		subset of $U$, then $\card ( \Phi \cap \{ C \with C \cap K
		\neq \varnothing \} ) < \infty$.
		\item \label{item:decompositions:open} If $C \in \Phi$, then
		$C$ is open relative to $\spt \| V \|$.
		\item \label{item:decompositions:connected_components} If $C
		\in \Phi$, then $\spt ( \| V \| \restrict C ) = C$ and $V
		\boundary C = 0$.
	\end{enumerate}
\end{lemma}

\begin{proof}
	Our hypotheses yield $\spt \| V \| = \bigcup \{ \spt \| W \| \with W
	\in \Xi \}$; hence, \cite[6.8]{MR3528825} implies
	\eqref{item:decompositions:refinement}.  Moreover,
	\eqref{item:decompositions:refinement} implies
	\eqref{item:decompositions:finite} and
	\eqref{item:decompositions:finite} implies
	\eqref{item:decompositions:open}.  Finally,
	\eqref{item:decompositions:open} and \cite[6.5]{MR3528825} yield
	\eqref{item:decompositions:connected_components}.
\end{proof}

\begin{remark}
	The proof is almost identical to \cite[6.11, 6.14]{MR3528825}.
\end{remark}

\begin{remark}
	For nonrectifiable varifolds $V$, no decomposition needs to exist even
	if $\updelta V = 0$, see \cite[4.12\,(3)]{MR3777387}.
\end{remark}

\begin{lemma} \label{lemma:spt-partitions}
	Suppose $m$ and $n$ are positive integers, $m \leq n$, $U$ is an open
	subset of $\mathbf R^n$, $V \in \mathbf{RV}_m ( U )$, $\| \updelta V
	\|$ is a Radon measure, and every decomposition of $V$ is locally
	finite.
	
	Then, there holds
	\begin{equation*}
		\spt \| V \| = \bigcup \{ \spt \| W \| \with W \in \Pi \}
		\quad \text{whenever $\Pi$ is a partition of $V$}.
	\end{equation*}
\end{lemma}

\begin{proof}
	With a decomposition $\Xi$ related to $\Pi$ as in
	\ref{remark:partition-vs-decomposition}, we apply
	\ref{lemma:decompositions}\,\eqref{item:decompositions:refinement}.
\end{proof}

\begin{theorem} \label{thm:indecomposable_type_Psi}
	Suppose $m$ and $n$ are positive integers, $m \leq n$, $U$ is an open
	subset of $\mathbf R^n$, $V \in \mathbf{RV}_m (U)$, $\| \updelta V \|$
	is a Radon measure, every decomposition of $V$ is locally finite,
	$\spt \| V \|$ is connected, and
	\begin{equation*}
		\Gamma = \mathbf T(V) \cap \{ f \with \textup{$\dmn f = \spt
		\| V \|$, $f$ is continuous} \}.
	\end{equation*}
	
	Then, whenever $f \in \Gamma$ and $E(y) = \{ y \with f(x)>y \}$ for $y
	\in \mathbf R$, the set
	\begin{equation*}
		\big \{ y \with \| V \| ( E(y)) > 0, \| V \| ( U \without
		E(y)) > 0, V \boundary E(y) = 0 \big \}
	\end{equation*}
	is countable; in particular, $V$ is indecomposable of type $\Gamma$.
\end{theorem}

\begin{proof}
	Suppose $f \in \Gamma$ and $\Pi$ denotes a partition of $V$ along $f$,
	see \ref{thm:decomposition_adapted_to_fct}.  Since $f$ is continuous
	and $\dmn f$ is connected, $\im f$ is an interval.  Moreover, we have
	\begin{equation*}
		\spt \| V \| = \bigcup \{ \spt \| W \| \with W \in \Pi \}
	\end{equation*}
	by \ref{lemma:spt-partitions}.  It follows
	\begin{equation*}
		\im f \subset \bigcup \{ J(W) \with W \in \Pi \}
	\end{equation*}
	by \ref{remark:utility}\,\eqref{item:utility:continuous}, where $J(W)
	= \spt f_\# \| W \|$.  Abbreviating
	\begin{equation*}
		B = \big \{ y \with \text{$\| V \| ( E(y)) > 0$, $\| V \| ( U
		\without E(y) ) > 0$, and $V \boundary E(y) = 0$} \big \},
	\end{equation*}
	we conclude
	\begin{equation*}
		B \cap \bigcup_{W \in \Pi} \{ y \with \inf J(W) < y < \sup
		J(W) \} = \varnothing.
	\end{equation*}
	because we have $V \boundary E(y) = W \boundary E(y) \neq 0$ whenever
	$\inf J(W) < y < \sup J(W)$ and $W \in \Pi$ by \ref{remark:partition}
	and \ref{remark:partition-along-pi}.  Noting $B \subset \im f$, the
	set $B$ is accordingly contained in the countable set $\bigcup \{
	\Bdry J(W) \with W \in \Pi \}$; in particular, $\mathscr L^1(B)=0$.
\end{proof}

\begin{remark} \label{remark:indecomposability_immersions}
	By \ref{example:indecomposability-type-f}, our estimate is sharp and
	$V$ may be decomposable.  This also shows that indecomposability of
	type $\Gamma$ and indecomposability differ which, by
	\cite[8.7]{MR3528825}, answers the second question posed in
	\cite[Section~A]{scharrer:MSc}.
\end{remark}

\begin{remark}
	If the hypotheses of \ref{thm:locally_finite_near_boundary} are
	satisfied with $B = \varnothing$ and $\spt \| V \|$ is connected, then
	geometrically significant members of $\Gamma$ which may fail to be
	Lipschitzian are the geodesic distance functions on $\spt \| V \|$ to
	points in $\spt \| V \|$, see \cite[6.8\,(2), 6.11]{MR3626845}; more
	generally, this applies to functions in the image of the embedding of
	the \emph{local Sobolev space with respect to $V$ and exponent $q$}
	satisfying $q>m$ if $m>1$, into the space of real valued continuous
	functions on $\spt \| V \|$, see \cite[7.12]{MR3626845}.
\end{remark}

\begin{corollary} \label{corollary:indecomposability-equivalence}
	Suppose $m$ and $n$ are positive integers, $m \leq n$, $U$ is an open
	subset of $\mathbf R^n$, $V \in \mathbf{RV}_m (U)$, $\| \updelta V \|$
	is a Radon measure, every decomposition of $V$ is locally finite, and
	$\Gamma = \mathbf T(V) \cap \{ f \with \textup{$\dmn f = \spt \| V
	\|$, $f$ is continuous} \}$.
	
	Then, the following three statements are equivalent.
	\begin{enumerate}
		\item The set $\spt \| V \|$ is connected.
		\item The varifold $V$ is indecomposable of type $\Gamma$.
		\item The varifold $V$ is indecomposable of type $\mathscr E (
		U, \mathbf R )$.
	\end{enumerate}
\end{corollary}

\begin{proof}
	We combine \ref{thm:indecomposability_connected} and
	\ref{thm:indecomposable_type_Psi}.
\end{proof}

\begin{corollary} \label{corollary:indecomposable_type_Gamma}
	Suppose $V$ is as in
	\ref{example:regularity-hypotheses}, the set $\spt \| V \|$ is
	connected, and $\Gamma = \mathbf T(V) \cap \{ f \with \textup{$\dmn f
	= \spt \| V \|$, $f$ is continuous} \}$.
	
	Then, $V$ is indecomposable of type $\Gamma$.
\end{corollary}

\begin{proof}
	We combine \ref{thm:locally_finite_near_boundary} and
	\ref{thm:indecomposable_type_Psi}.
\end{proof}

\begin{remark} \label{remark:clamped-Willmore}
	Integral varifolds satisfying the hypotheses with $m = 2$, $n = 3$,
	and $U = \mathbf R^3$ occur in the minimisation of the Willmore energy
	with \emph{clamped boundary condition} amongst connected surfaces; see
	\cite[Theorem 4.1]{MR4141858}.
\end{remark}

\begin{corollary} \label{corollary:Neumann-indecomposable}
	Suppose $m$ and $V$ satisfy the equivalent conditions of
	\ref{example:Neumann-Lm}, $\boldsymbol \Uptheta^m ( \| V \|, x ) \geq
	1$ for $\| V \|$ almost all $x$, $\spt \| V \|$ is connected, and
	\begin{equation*}
		\Gamma = \mathbf T (V) \cap \{ f \with \textup{$\dmn f = \spt
		\| V \|$, $f$ is continuous} \}.
	\end{equation*}

	Then, $V$ is indecomposable of type $\Gamma$.
\end{corollary}

\begin{proof}
	Recalling \ref{remark:Neumann-rectifiability} and
	\ref{example:Neumann-Lm}\,\eqref{item:Neumann-Lm:summability}, we
	combine \ref{corollary:Neumann-locally-finite-decompositions} and
	\ref{thm:indecomposable_type_Psi}.
\end{proof}

\begin{remark} \label{remark:Neumann-indecomposable}
	By \ref{remark:Neumann-rectifiability},
	\ref{remark:Neumann-locally-finite-decompositions}, and
	\ref{thm:indecomposable_type_Psi}, the conclusion also holds if $m=1$,
	$V$ satisfies the conditions of \ref{example:Neumann-boundary},
	$\boldsymbol \Uptheta^1 ( \| V \|, x ) \geq 1$ for $\| V \|$ almost
	all $x$, and $\spt \| V \|$ is connected.
\end{remark}

\medskip \noindent \textsc{Affiliations}

\medskip \noindent Ulrich Menne \smallskip \newline
Department of Mathematics \\
National Taiwan Normal University \\
No.88, Sec.4, Tingzhou Rd. \\
Wenshan Dist., \textsc{Taipei City 116059 \\
	Taiwan(R.\ O.\ C.)}

\medskip \noindent Christian Scharrer \smallskip \newline Institute for Applied Mathematics
\newline University of Bonn
\newline Endenicher Allee 60 \newline 
\textsc{53115 Bonn} \\ \textsc{Germany}

\medskip \noindent \textsc{Email addresses}

\medskip \noindent
\href{mailto:Ulrich.Menne@math.ntnu.edu.tw}{Ulrich.Menne@math.ntnu.edu.tw}
\quad
\href{mailto:Scharrer@iam.uni-bonn.de}{Scharrer@iam.uni-bonn.de}


\begin{thebibliography}{ACMM01}

\bibitem[ACMM01]{MR1812124}
Luigi Ambrosio, Vicent Caselles, Simon Masnou, and Jean-Michel Morel.
\newblock Connected components of sets of finite perimeter and applications to
  image processing.
\newblock {\em J. Eur. Math. Soc. (JEMS)}, 3(1):39--92, 2001.
\newblock URL: \url{https://doi.org/10.1007/PL00011302}.

\bibitem[AFP00]{MR2003a:49002}
Luigi Ambrosio, Nicola Fusco, and Diego Pallara.
\newblock {\em Functions of bounded variation and free discontinuity problems}.
\newblock Oxford Mathematical Monographs. The Clarendon Press Oxford University
  Press, New York, 2000.

\bibitem[Aim23]{MR4612634}
Satoru Aimi.
\newblock Level set mean curvature flow, with {N}eumann boundary conditions.
\newblock {\em Hokkaido Math. J.}, 52(1):41--64, 2023.
\newblock URL: \url{https://doi.org/10.14492/hokmj/2020-426}.

\bibitem[All72]{MR0307015}
William~K. Allard.
\newblock On the first variation of a varifold.
\newblock {\em Ann. of Math. (2)}, 95(3):417--491, 1972.
\newblock URL: \url{https://doi.org/10.2307/1970868}.

\bibitem[All75]{MR0397520}
William~K. Allard.
\newblock On the first variation of a varifold: boundary behavior.
\newblock {\em Ann. of Math. (2)}, 101(3):418--446, 1975.
\newblock URL: \url{https://doi.org/10.2307/1970934}.

\bibitem[Alm00]{MR1777737}
Frederick~J. Almgren, Jr.
\newblock {\em Almgren's big regularity paper}, volume~1 of {\em World
  Scientific Monograph Series in Mathematics}.
\newblock World Scientific Publishing Co. Inc., River Edge, NJ, 2000.
\newblock $Q$-valued functions minimizing Dirichlet's integral and the
  regularity of area-minimizing rectifiable currents up to codimension 2, With
  a preface by Jean E.\ Taylor and Vladimir Scheffer.
\newblock URL: \url{https://doi.org/10.1142/9789812813299}.

\bibitem[BBG{\etalchar{+}}95]{MR1354907}
Philippe B{\'e}nilan, Lucio Boccardo, Thierry Gallou{\"e}t, Ron Gariepy, Michel
  Pierre, and Juan~Luis V{\'a}zquez.
\newblock An {$L^1$}-theory of existence and uniqueness of solutions of
  nonlinear elliptic equations.
\newblock {\em Ann. Scuola Norm. Sup. Pisa Cl. Sci. (4)}, 22(2):241--273, 1995.
\newblock URL: \url{http://www.numdam.org/item?id=ASNSP_1995_4_22_2_241_0}.

\bibitem[CFS22]{MR4524829}
Alessandro Carlotto, Giada Franz, and Mario~B. Schulz.
\newblock Free boundary minimal surfaces with connected boundary and arbitrary
  genus.
\newblock {\em Camb. J. Math.}, 10(4):835--857, 2022.
\newblock URL: \url{https://doi.org/10.4310/cjm.2022.v10.n4.a3}.

\bibitem[Cho23]{MR4609162}
Hsin-Chuang Chou.
\newblock Integral decompositions of varifolds.
\newblock {\em Ann. Global Anal. Geom.}, 64:Paper No. 3, 16, 2023.
\newblock URL: \url{https://doi.org/10.1007/s10455-023-09908-x}.

\bibitem[Dav14]{MR3329849}
Guy David.
\newblock Should we solve {P}lateau's problem again?
\newblock In {\em Advances in analysis: the legacy of {E}lias {M}. {S}tein},
  volume~50 of {\em Princeton Math. Ser.}, pages 108--145. Princeton Univ.
  Press, Princeton, NJ, 2014.

\bibitem[Dav19]{MR3975493}
Guy David.
\newblock Local regularity properties of almost- and quasiminimal sets with a
  sliding boundary condition.
\newblock {\em Ast\'{e}risque}, 411:ix+377, 2019.
\newblock URL: \url{https://doi.org/10.24033/ast.1077}.

\bibitem[DGA88]{MR1152641}
Ennio De~Giorgi and Luigi Ambrosio.
\newblock New functionals in the calculus of variations.
\newblock {\em Atti Accad. Naz. Lincei Rend. Cl. Sci. Fis. Mat. Natur. (8)},
  82(2):199--210 (1989), 1988.

\bibitem[DM21]{MR4359920}
Luigi De~Masi.
\newblock Rectifiability of the free boundary for varifolds.
\newblock {\em Indiana Univ. Math. J.}, 70(6):2603--2651, 2021.
\newblock URL: \url{https://doi.org/10.1512/iumj.2021.70.9401}.

\bibitem[DS58]{MR0117523}
Nelson Dunford and Jacob~T. Schwartz.
\newblock {\em Linear {O}perators. {I}. {G}eneral {T}heory}.
\newblock With the assistance of W. G. Bade and R. G. Bartle. Pure and Applied
  Mathematics, Vol. 7. Interscience Publishers, Inc., New York; Interscience
  Publishers, Ltd., London, 1958.
\newblock URL:
  \url{https://babel.hathitrust.org/cgi/pt?id=mdp.39015000962400;view=1up;seq=9A}.

\bibitem[Ede20]{MR4048443}
Nick Edelen.
\newblock The free-boundary {B}rakke flow.
\newblock {\em J. Reine Angew. Math.}, 758:95--137, 2020.
\newblock URL: \url{https://doi.org/10.1515/crelle-2017-0053}.

\bibitem[Fan21]{MR4220652}
Yangqin Fang.
\newblock Local {$C^{1,\beta}$}-regularity at the boundary of two dimensional
  sliding almost minimal sets in {$\mathbb{R}^3$}.
\newblock {\em Trans. Amer. Math. Soc. Ser. B}, 8:130--189, 2021.
\newblock URL: \url{https://doi.org/10.1090/btran/40}.

\bibitem[Fed59]{MR0110078}
Herbert Federer.
\newblock Curvature measures.
\newblock {\em Trans. Amer. Math. Soc.}, 93:418--491, 1959.
\newblock URL: \url{https://doi.org/10.1090/S0002-9947-1959-0110078-1}.

\bibitem[Fed69]{MR41:1976}
Herbert Federer.
\newblock {\em Geometric measure theory}.
\newblock Die Grundlehren der ma\-the\-ma\-ti\-schen Wissenschaften, Band 153.
  Springer-Verlag New York Inc., New York, 1969.
\newblock URL: \url{https://doi.org/10.1007/978-3-642-62010-2}.

\bibitem[FK18]{MR3800850}
Yangqin Fang and Sławomir Kolasiński.
\newblock Existence of solutions to a general geometric elliptic variational
  problem.
\newblock {\em Calc. Var. Partial Differential Equations}, 57(3):Art. 91, 71,
  2018.
\newblock URL: \url{https://doi.org/10.1007/s00526-018-1348-4}.

\bibitem[FS16]{MR3461367}
Ailana Fraser and Richard Schoen.
\newblock Sharp eigenvalue bounds and minimal surfaces in the ball.
\newblock {\em Invent. Math.}, 203(3):823--890, 2016.
\newblock URL: \url{https://doi.org/10.1007/s00222-015-0604-x}.

\bibitem[GM24]{MR4814697}
Michael Goldman and Benoit Merlet.
\newblock Set-decomposition of normal rectifiable {$G$}-chains via an abstract
  decomposition principle.
\newblock {\em Rev. Mat. Iberoam.}, 40(6):2073--2094, 2024.
\newblock URL: \url{https://doi.org/10.4171/rmi/1504}.

\bibitem[GT01]{MR1814364}
David Gilbarg and Neil~S. Trudinger.
\newblock {\em Elliptic partial differential equations of second order}.
\newblock Classics in Mathematics. Springer-Verlag, Berlin, 2001.
\newblock Reprint of the 1998 edition.
\newblock URL: \url{https://doi.org/10.1007/978-3-642-61798-0}.

\bibitem[Hir94]{MR1336822}
Morris~W. Hirsch.
\newblock {\em Differential topology}, volume~33 of {\em Graduate Texts in
  Mathematics}.
\newblock Springer-Verlag, New York, 1994.
\newblock Corrected fifth printing.
\newblock URL: \url{https://doi.org/10.1007/978-1-4684-9449-5}.

\bibitem[ITIT69]{MR0276438}
A.~Ionescu~Tulcea and C.~Ionescu~Tulcea.
\newblock {\em Topics in the theory of lifting}.
\newblock Ergebnisse der Mathematik und ihrer Grenzgebiete, Band 48.
  Springer-Verlag New York Inc., New York, 1969.
\newblock URL: \url{https://doi.org/10.1007/978-3-642-88507-5}.

\bibitem[Kag19]{MR3953134}
Takashi Kagaya.
\newblock Convergence of the {A}llen-{C}ahn equation with a zero {N}eumann
  boundary condition on non-convex domains.
\newblock {\em Math. Ann.}, 373(3-4):1485--1528, 2019.
\newblock URL: \url{https://doi.org/10.1007/s00208-018-1720-x}.

\bibitem[KM17]{MR3625810}
Sławomir Kolasiński and Ulrich Menne.
\newblock Decay rates for the quadratic and super-quadratic tilt-excess of
  integral varifolds.
\newblock {\em NoDEA Nonlinear Differential Equations Appl.}, 24(2):Art. 17,
  56, 2017.
\newblock URL: \url{https://doi.org/10.1007/s00030-017-0436-z}.

\bibitem[KM24]{MR4787572}
Ernst Kuwert and Marius M\"uller.
\newblock Curvature varifolds with orthogonal boundary.
\newblock {\em J. Lond. Math. Soc. (2)}, 110(3):Paper No. e12976, 42, 2024.
\newblock URL: \url{https://doi.org/10.1112/jlms.12976}.

\bibitem[Lab22]{MR4489608}
Camille Labourie.
\newblock Solutions of the (free boundary) {R}eifenberg {P}lateau problem.
\newblock {\em Adv. Calc. Var.}, 15(4):913--927, 2022.
\newblock URL: \url{https://doi.org/10.1515/acv-2020-0067}.

\bibitem[LW25]{arXiv:2512.19227v1}
Yu~Tong Liu and Myles Workman.
\newblock The space-time-{G}rassmann measure of the {B}rakke flow, 2025.
\newblock \href {http://arxiv.org/abs/2512.19227v1}
  {\path{arXiv:2512.19227v1}}.

\bibitem[LZ21]{MR4285846}
Martin Man-Chun Li and Xin Zhou.
\newblock Min-max theory for free boundary minimal hypersurfaces
  {I}---{R}egularity theory.
\newblock {\em J. Differential Geom.}, 118(3):487--553, 2021.
\newblock URL: \url{https://doi.org/10.4310/jdg/1625860624}.

\bibitem[Man96]{MR1412686}
Carlo Mantegazza.
\newblock Curvature varifolds with boundary.
\newblock {\em J. Differential Geom.}, 43(4):807--843, 1996.
\newblock URL:
  \url{http://projecteuclid.org/getRecord?id=euclid.jdg/1214458533}.

\bibitem[Men04]{snulmenn:diploma_thesis}
Ulrich Menne.
\newblock Ein maßtheoretischer {Z}ugang zu {U}nterhalbstetigkeit und
  {K}onvergenz von {F}unktionalen, 2004.
\newblock Diploma thesis, University of Bonn.

\bibitem[Men09]{MR2537022}
Ulrich Menne.
\newblock Some applications of the isoperimetric inequality for integral
  varifolds.
\newblock {\em Adv. Calc. Var.}, 2(3):247--269, 2009.
\newblock URL: \url{https://doi.org/10.1515/ACV.2009.010}.

\bibitem[Men16a]{MR3528825}
Ulrich Menne.
\newblock Weakly differentiable functions on varifolds.
\newblock {\em Indiana Univ. Math. J.}, 65(3):977--1088, {\noopsort{a}}2016.
\newblock URL: \url{https://doi.org/10.1512/iumj.2016.65.5829}.

\bibitem[Men16b]{MR3626845}
Ulrich Menne.
\newblock Sobolev functions on varifolds.
\newblock {\em Proc. Lond. Math. Soc. (3)}, 113(6):725--774,
  {\noopsort{b}}2016.
\newblock URL: \url{https://doi.org/10.1112/plms/pdw023}.

\bibitem[Men19]{MR3936235}
Ulrich Menne.
\newblock Pointwise differentiability of higher order for sets.
\newblock {\em Ann. Global Anal. Geom.}, 55(3):591--621, 2019.
\newblock URL: \url{https://doi.org/10.1007/s10455-018-9642-0}.

\bibitem[Men21]{MR4241804}
Ulrich Menne.
\newblock Pointwise differentiability of higher-order for distributions.
\newblock {\em Anal. PDE}, 14(2):323--354, 2021.
\newblock URL: \url{https://doi.org/10.2140/apde.2021.14.323}.

\bibitem[Mon14]{MR3148123}
Andrea Mondino.
\newblock Existence of integral {$m$}-varifolds minimizing {$\int\vert A\vert
  ^p$} and {$\int\vert H\vert ^p,\,p>m,$} in {R}iemannian manifolds.
\newblock {\em Calc. Var. Partial Differential Equations}, 49(1-2):431--470,
  2014.
\newblock URL: \url{https://doi.org/10.1007/s00526-012-0588-y}.

\bibitem[MS18]{MR3777387}
Ulrich Menne and Christian Scharrer.
\newblock An isoperimetric inequality for diffused surfaces.
\newblock {\em Kodai Math. J.}, 41(1):70--85, 2018.
\newblock URL: \url{https://doi.org/10.2996/kmj/1521424824}.

\bibitem[MS24]{arXiv:1709.05504v3}
Ulrich Menne and Christian Scharrer.
\newblock A priori bounds for geodesic diameter. {P}art {III}. {A}
  {S}obolev-{P}oincaré inequality and applications to a variety of geometric
  variational problems, 2024.
\newblock Accepted in \emph{Duke Math. J.}
\newblock \href {http://arxiv.org/abs/1709.05504v3}
  {\path{arXiv:1709.05504v3}}.

\bibitem[MS25]{DOI_10.4171_RMI_1487}
Ulrich Menne and Christian Scharrer.
\newblock A priori bounds for geodesic diameter. {P}art {I}. {I}ntegral chains
  with coefficients in a complete normed commutative group.
\newblock {\em Rev. Mat. Iberoam.}, 41:29--72, 2025.
\newblock URL: \url{https://doi.org/10.4171/rmi/1487}.

\bibitem[MT15]{MR3348119}
Masashi Mizuno and Yoshihiro Tonegawa.
\newblock Convergence of the {A}llen-{C}ahn equation with {N}eumann boundary
  conditions.
\newblock {\em SIAM J. Math. Anal.}, 47(3):1906--1932, 2015.
\newblock URL: \url{https://doi.org/10.1137/140987808}.

\bibitem[MT16]{MR3542008}
Masashi Mizuno and Yoshihiro Tonegawa.
\newblock Erratum to ``{C}onvergence of the {A}llen-{C}ahn equation with
  {N}eumann boundary conditions''.
\newblock {\em SIAM J. Math. Anal.}, 48(4):3035--3036, 2016.
\newblock URL: \url{https://doi.org/10.1137/16M1074059}.

\bibitem[NP20]{MR4141858}
Matteo Novaga and Marco Pozzetta.
\newblock Connected surfaces with boundary minimizing the {W}illmore energy.
\newblock {\em Math. Eng.}, 2(3):527--556, 2020.
\newblock URL: \url{https://doi.org/10.3934/mine.2020024}.

\bibitem[Rei71]{MR0273082}
William~T. Reid.
\newblock {\em Ordinary differential equations}.
\newblock John Wiley \& Sons, Inc., New York-London-Sydney, 1971.

\bibitem[Sak96]{MR1390760}
Takashi Sakai.
\newblock {\em Riemannian geometry}, volume 149 of {\em Translations of
  Mathematical Monographs}.
\newblock American Mathematical Society, Providence, RI, 1996.
\newblock Translated from the 1992 Japanese original by the author.

\bibitem[San19]{MR3978264}
Mario Santilli.
\newblock Rectifiability and approximate differentiability of higher order for
  sets.
\newblock {\em Indiana Univ. Math. J.}, 68(3):1013--1046, 2019.
\newblock URL: \url{https://doi.org/10.1512/iumj.2019.68.7645}.

\bibitem[Sch16]{scharrer:MSc}
Christian Scharrer.
\newblock Relating diameter and mean curvature for varifolds, 2016.
\newblock MSc thesis, University of Potsdam.
\newblock URL: \url{http://nbn-resolving.de/urn:nbn:de:kobv:517-opus4-97013}.

\bibitem[Whi99]{MR1738045}
Brian White.
\newblock The deformation theorem for flat chains.
\newblock {\em Acta Math.}, 183(2):255--271, 1999.
\newblock URL: \url{https://doi.org/10.1007/BF02392829}.

\end{thebibliography}
\end{document}